\documentclass{article}

\usepackage{amssymb,amsmath,amsthm,enumerate,graphicx,mathtools,subcaption,color,url}
\usepackage[margin=1in]{geometry}

\usepackage{pgfbaseplot,pgfplots,tikz,ifthen,tikz-3dplot}
\usetikzlibrary{decorations.text,hobby,arrows.meta}

\usepackage[makeroom]{cancel}

\newtheorem*{theorem*}{Theorem}
\newtheorem{theorem}{Theorem}
\newtheorem{lemma}[theorem]{Lemma}
\newtheorem*{lemma*}{Lemma}
\newtheorem{proposition}[theorem]{Proposition}

\newtheorem{notation}[theorem]{Notation}
\newtheorem{definition}[theorem]{Definition}
\newtheorem{remark}[theorem]{Remark}

\def\Diff{\mathrm{Diff}}
\def\qtt{F_{\mathrm{quad}}^{23}}
\def\qot{F_{\mathrm{quad}}^{12}}
\def\oot{F_{\mathrm{oct}}^{12}}
\def\Fsec{F_{\mathrm{sec}}}
\def\R{\mathcal{R}}
\def\RR{\mathbb{R}}
\def\D{\mathcal{D}}
\def\S{\mathcal{S}}
\def\Id{\mathrm{Id}}
\def\akep{\alpha_{\mathrm{Kep}}}
\def\L{\Lambda}
\def\eps{\varepsilon}
\def\LPM{\mathcal{L}}
\def\F{\mathcal{F}}
\def\E{\mathcal{E}}

\def\cD{\mathcal{D}}
\def\F{\mathcal{F}}
\def\Id{\mathrm{Id}}
\def\rr{\rho}
\def\kk{\kappa}

\def\wt{\widetilde}

\mathtoolsset{showonlyrefs}

\title{A counterexample to the theorem of Laplace-Lagrange\\
  on the stability of semimajor axes} 

\author{Andrew Clarke \and Jacques Fejoz \and Marcel Guardia}

\newcommand{\Addresses}{{
  \bigskip
  \footnotesize

  \textsc{Andrew Clarke, \newline Departament de Matem\`atiques i Inform\`atica, Universitat de Barcelona,   Gran Via, 585, 08007 Barcelona, Spain}\par\nopagebreak
  \textit{E-mail address}: \texttt{andrew.clarke@ub.edu}

  \medskip

  \textsc{Jacques Fejoz, \newline Universit\'e Paris Dauphine PSL, CEREMADE,  Place du Mar\'echal de Lattre de Tassigny,  75016 Paris, France \newline  Observatoire de Paris PSL, IMCCE,  77 avenue Denfert Rochereau, 75014 Paris, France}\par\nopagebreak
  \textit{E-mail address}: \texttt{jacques.fejoz@dauphine.psl.eu}

  \medskip

  \textsc{Marcel Guardia, \newline Departament de Matem\`atiques i Inform\`atica, Universitat de Barcelona,
  Gran Via, 585, 08007 Barcelona, Spain \newline 
  Centre de Recerca Matem\`atica, 
  Edifici C, Campus Bellaterra,  08193 Bellaterra, Spain}\par\nopagebreak
  \textit{E-mail address}: \texttt{guardia@ub.edu}

}}

\begin{document}

\maketitle

\begin{abstract}
  A longstanding belief has been that the semimajor axes, in the Newtonian planetary problem, are stable. In the course of the XIX century, Laplace, Lagrange and others gave stronger and stronger arguments in this direction, thus culminating in what has commonly been referred to as the first Laplace-Lagrange stability theorem. In the problem with $3$ planets, we prove the existence of orbits along which the semimajor axis of the outer planet undergoes large random variations thus disproving the theorem of Laplace-Lagrange. The time of instability varies as a negative power of the masses of the planets. The orbits we have found fall outside the scope of the theory of Nekhoroshev-Niederman because they are not confined by the conservation of angular momentum and because the Hamiltonian is not (uniformly) convex with respect to the Keplerian actions. 
\end{abstract}


\tableofcontents

\section{Introduction}
\label{sec:intro}
\subsection{On the stability of semimajor axes}

Consider the $4$-body problem, namely the motion of $4$ bodies, numbered from $0$ to $3$, moving in the $3$-dimensional space and subject to the Newtonian universal attraction:
\begin{equation}
  \label{eq:4bp}
  \ddot x_j=\sum_{\substack {0 \leq i \leq 3 \\ i\neq j}} m_i\frac{x_i-x_j}{\|x_i-x_j\|^3},
\end{equation}
where $x_j \in \mathbb{R}^3$ is the position and $m_j>0$ the mass of body $j$. Of particular interest is the planetary problem, where the masses of bodies $1,2,3$ (planets) are small with respect to body $0$ (Sun), and where each planet revolves around the Sun along an approximate, slowly deforming Keplerian ellipse. In the first approximation, the problem consists of three uncoupled Kepler problems whose ellipses are fixed in space, together with their geometric elements determining the shape of the ellipses and their position in space. The question is to determine the long term influence of the mutual attraction of planets on the elliptical positions and elements. In this article, we will also consider the hierarchical problem, where masses are fixed (or within some compact set of $(0,\infty)$) and successive semimajor axes' ratios $a_j/a_{j+1}$ are small.

Euler and Lagrange had failed to prove the stability of semimajor axes of planets in the Solar System. In 1776, in a commendable \textit{tour de force} Laplace was able to overcome the difficulties his predecessors had met. He wrote \cite{Laplace:1776}:

\begin{quote}
  J'ai trouv\'e [que l'in\'egalit\'e s\'eculaire des demi grands axes est] absolument nulle; d'o\`u je conclus que l'alt\'eration du mouvement moyen de Jupiter, si elle existe, n'est point due \`a l'action de Saturne.\footnote{In modern English: I have found that the variations of the semimajor axis of Jupiter, under the influence of Saturn, have zero average.}  
\end{quote}
Here Laplace is neglecting second order terms in the masses of the planets, as well as third order terms in the eccentricities and inclinations of planets. 

Lagrange later proved that this result holds for arbitrary eccentricities and inclinations. This is the ``first stability theorem of Laplace and Lagrange''.  About the 1808 M\'emoire of Lagrange~\cite{Lagrange:1808}, Arago commented: ``Le 17 ao\^ut 1808, [Lagrange] lit au Bureau des longitudes, et le lundi suivant 22, \`a l'Acad\'emie des sciences, un des plus admirables M\'emoires qu'ait jamais trac\'es la plume d'un math\'ematicien'' (F. Arago, \textit{\OE uvres compl\`etes}, 1854, p.~654).\footnote{On 17 August 1808, and on the following Monday 22, at the Acad\'emie des sciences, [Lagrange] reads one of the most magnificient memoirs ever written by a mathematician. This work was entitled: Memoir on the theory of the variations of planets' elements, and in particular of the variations of semimajor axes of their orbits.}

Poisson later proved that the conclusion of the theorem holds at the second order in the masses of the planets~\cite{Poisson:1809}. His proof is a cornerstone of  Hamiltonian perturbation theory, but is lengthy and complicated. Lagrange simplified it substantially (see his \OE uvres, t. VI, p. 735), but to the point where  Lagrange's argument is flawed, as his editor M. Serret mentions. The later correction made in \cite{Lagrange:1853} is not satisfactory either, as Mathieu noticed~\cite{Mathieu:1875}... (see \cite{Laskar:2010:Poincare, Fejoz:2013:200} and references therein). 

Nowadays the first stability theorem of Laplace-Lagrange-Poisson-Mathieu is a simple consequence of the existence of the Delaunay coordinates for the two-body problem. In these symplectic coordinates, the variable which is conjugate to the fast Keplerian angle (mean anomaly) is a function of the semimajor axis. So, outside Keplerian resonances, for the (first order) secular system obtained by averaging out the mean anomalies, semimajor axes are first integrals. 

In order to explain the irregularities of Jupiter and Saturn, Laplace called on comets. Comets had unknown masses, so it was a convenient argument (which actually was a fortunate motivation for Laplace to get interested in probabilities). Yet, there is an intricate interplay between small parameters in the parameter space (masses of the planets, distance to mean motion resonances, distance to circular motions, etc.). It is a mistake to infer the stability of the semimajor axes from the low order analysis that had been carried out, and indeed, averaging out the outer mean anomalies becomes irrelevant when the mean motion of outer planets is slower than the secular dynamics of inner planets.

More recently, after the proof of Arnold's theorem on the existence of a set of positive Lebesgue measure of invariant tori in the planetary problem~\cite{Arnold:1963,fejoz2004arnold}, Herman has speculated that ``in some respect Lagrange and Laplace, against Newton, are correct in the sense of measure theory and that in the sense of topology, the above question [on the stability], in some respect, could show Newton is correct'' \cite{herman1998icm}.

It is the purpose of the present article to disprove the belief in the Laplace-Lagrange stability of the semimajor axis, as well as Herman's conjectural dichotomy. Instability occurs on a set of positive Lebesgue measure of the $4$-body problem in the planetary regime, in a time which is an inverse power of the masses of the planets.

\subsection{Main results}\label{sec:mainresults}

Consider $4$ bodies whose motion is governed by Newton's equation~\eqref{eq:4bp}. We will assume that $m_0 \neq m_1$.\footnote{If the four masses are not equal to each other, this condition is always satisfied up to renumbering the masses, i.e. up to switching the roles of bodies.} For the sake of simplicity, let us first focus on the ``hierarchical regime''; it is the asymptotic regime where masses are fixed, while body $2$ revolves around and far away from bodies $0$ and $1$, and body $3$ revolves around and even farther away from bodies $0$, $1$ and $2$. (We will make some more precise hypotheses below.) Each body thus primarily undergoes the attraction of one other body: bodies $0$ and $1$ are close to being isolated, body $2$ primarily undergoes the attraction of a fictitious body located at the center of mass of $0$ and $1$, and body $3$ primarily undergoes the attraction of a fictitious body located at the center of mass of $0$, $1$ and $2$. We think of body $0$ as the Sun and of the three other bodies as planets. The Jacobi coordinates are well suited for this regime, but we defer their definition to a later stage. Assuming that the center of mass is fixed, the small displacements of the Sun may be recovered from the positions of the planets.

Some notation: let $a_1$, $a_2$ and $a_3$ be the semimajor axes of the planets, $e_1$, $e_2$ and $e_3$ be their eccentricities, and $C_1$, $C_2$ and $C_3$ their angular momenta. In the hierarchical regime, for eccentricities bounded away from $1$, $a_1\ll a_2\ll a_3$. Even further (and unlike in~\cite{clarke2022why}), we will consider a \emph{strongly hierarchical regime}, where not only the semimajor axes ratios $\alpha_i = a_i/a_{i+1}$ are small, but even the ratios of the ratios $\alpha_i/\alpha_{i+1}$ are small, in the following quantitative manner:
\begin{equation}\label{def:regime}
  a_1 = O(1) \ll a_2 \ll a_3^{1/3}.
\end{equation}
Here is the rough description of the scales of times:
\begin{itemize}
\item The fastest frequencies are the mean motions (Keplerian frequencies) of the two inner planets. Since $a_1 \ll a_2$, these inner mean motions do not interfere, which allows us to average out the mean anomalies, without resonances.
  
\item The next frequencies are the secular frequencies of the two inner planets. They govern the rotation of the plane of the ellipses around their angular momentum vector $C_1 + C_2$, and the rotation of the ellipses in their plane, as well as the quasiperiodic oscillations of the corresponding inclinations and eccentricities. The dynamics of the truncated relevant normal form (``quadrupolar dynamics'' of planets $1$ and $2$) is still integrable, as noticed by Harrington~\cite{Harrington:1968}, due to the fact that the quadrupolar Hamiltonian does not depend on the argument of the outer pericenter $g_2$.

\item In the strongly hierarchical regime, the outer semimajor axis is so large that the mean motion of planet $3$ is slower than secular frequencies of the two inner planets. 

\item Then come the secular frequencies of the (outer) planet $3$, approximately determined by the quadrupolar Hamiltonian of planets $2$ and $3$. The conservation of the total angular momentum vector $C = C_1 + C_2 + C_3 \simeq C_3$ prevents significant changes in the plane of the outer ellipse, or of the product $a_3\sqrt{1-e_3^2}$. On the other hand, it does not prevent major (joint) changes in $a_3$ or $e_3$.
\end{itemize}

Similarly to the regime studied in~\cite{clarke2022why}, along the orbits we will prove the existence of the two inner planets will be close to the hyperbolic secular singularity of the quadrupolar Hamiltonian or to the associated stable and unstable manifolds. In
particular, their mutual inclination will be large. 

We will pay special attention to two quantities:
\begin{itemize}
\item the semimajor axis $a_3$ of the outer planet
\item the normalized angular momentum $\tilde C_2 \in \mathbb{B}^3$ of planet $2$, defined as the vector orthogonal to the plane of its ellipse and whose norm is $\| \tilde C_2 \| = \sqrt{1-e_2^2}$.
\end{itemize}

\begin{theorem}[Main theorem]\label{thm:main}
Consider the 4-body problem with masses $m_j>0$, $j=0,1,2, 3$ with $m_0 \neq m_1$.  For every finite itinerary $\tilde C_2^1,...,\tilde C_2^k \in \mathbb{B}^3$, $a_3^1,...,a_3^k \in[1,+\infty)$ and every $\delta>0$, there exists an open set of initial conditions whose trajectories realise the prescribed itinerary up to precision $\delta$. 
\end{theorem}

This theorem is a consequence of Theorems \ref{thm:MainHierarch:Deprit} and \ref{thm:main:planetary:Deprit} below, which contain a more detailed description of the diffusing orbits.

Let us make some comments on Theorem \ref{thm:main}.
\begin{itemize}
\item The drifting time needed to follow the prescribed itinerary in the theorem satisfies
  \begin{equation}\label{def:timefinal}
   0< T < C(m_0,m_1,m_2,m_3)  \, \frac{N}{\delta^{\kappa} },
  \end{equation}
  where $C$ is a constant depending only on the masses and  the exponent $\kappa>0$ does not depend on $N$ nor on the itinerary. To be more precise, call $\alpha_i = a_i/a_{i+1}$, $i=1,2$, the semimajor axis ratios. As $\delta$ tends to zero, the $\alpha_i$'s will be chosen polynomially smaller, and the drifting time itself depends polynomially on the $\alpha_i$'s. 
  
 \item As stated, Theorem \ref{thm:main} assumes small semi major axis ratios, for fixed masses. In Section~\ref{sec:planetary}, we provide asymptotic estimates when we let the masses of the planets tend to $0$, i.e. in the planetary regime where $m_j = \rho \, \tilde m_j$ for $j=1,2,3$ with $\rho>0$ small. Then, one possibility is to let the semimajor axes of planets 1 and 2 tend to 0 as  $\rho\to 0$. In that case, the drifting time satisfies
  \begin{equation}\label{def:timefinal2}
   0< T <C(m_0,\tilde m_1,\tilde m_2,\tilde m_3)  \, \frac{N}{\delta^{\kappa} \rho^\nu}.
  \end{equation}
Another possibility is to place planets 1 and 2 at a uniform distance (with respect to $\rho$) from the Sun and place planet 3 very far away, so that $a_3\sim \rr^{-2/3}$.  That is the setting considered in Theorem \ref{thm:main:planetary:Deprit} below, where we provide the concrete exponent $\nu=35/3$.

Note that the instability time is polynomial with respect to the masses of the planets. See Section \ref{sec:Nekhoroshev} below for the comparison of the regime of Theorem \ref{thm:main} with those regimes where Nekhoroshev Theory can be applied to prove exponential stability of the semimajor axes.

\item The novelty of the unstable behavior presented in this paper compared to that of \cite{clarke2022why} is the evolution of the semimajor axis $a_3$ of the third planet. Indeed, in the moderately hierarchical regime considered in \cite{clarke2022why}, $a_3$ is stable, whereas in the strongly hierarchical regime it can follow any prescribed itinerary (see Section \ref{sec:moderatestsrong} below for a comparison between the two regimes).  

On the contrary, the changes in the normalized angular momentum of the second planet are the same in both regimes. Let us briefly describe what the changes in $\tilde C_2$ imply in terms of the orbital elements of the second planet.
Indeed, fixing a prescribed itinerary $\tilde C_2^0,..., \tilde C_2^N \in B^3$ is equivalent to prescribing any itinerary in: the eccentricity $e_2^k$, the mutual inclination $\theta_{23}^k$ between planets $2$ and $3$, and the longitude $h_2^k$ of the node of planet $2$, for $k=0\ldots N$. Then, we can construct an orbit and times $t_0< t_1 < \cdots < t_N$ such that the osculating orbital elements satisfy
\begin{equation}\label{def:drifte2theta23}
		|e_2(t_k)-e_2^k|\leq \delta,\qquad
		|\theta_{23}(t_k)-\theta_{23}^k|\leq \delta,\qquad 
		|h_2(t_k)-h_2^k|\leq \delta \qquad\text{for}\quad k=0,1,...,N.
\end{equation}
As already mentioned, the angular momentum of the third body is almost constant and therefore, the evolution of $e_3$ is determined by the evolution of $a_3$.

Finally, the evolution of the eccentricity $e_1$ of the first planet, and the mutual inclination $\theta_{12}$ between planets 1 and 2, cannot be controled since they are prescribed by the diffusion mechanism. Let us briefly mention that: 
\begin{itemize}
\item The eccentricity $e_1$ does change but it can start arbitrarily close to 0. That is, the initial configuration can have all planets performing close to circular motion.
\item The mutual inclination $i_{12}$ always stays above 55 degrees. 
\end{itemize}
One can see \cite{clarke2022why} for a more detailed description of the evolution of $e_1$ and $i_{12}$.

\item A further step would be to estimate the local probability of instability in some given time \cite[Section 6.4.7]{Arnold:2006}.

\item In our Solar System, semimajor axes seem very stable. There are some exceptions. Notably, the semimajor axis of the Moon is drifting. But this is due to non-Hamiltonian, tidal effects \cite{Farhat:2022}. Also, at the early stages of our Solar System, planets migrated towards the exterior of the Solar System. But this migration too is a non-conservative phenomenon, explained by the interaction with the planetesimal disk \cite{Gomes:2004}.

  Orbits described in theorem~\ref{thm:main} show wild variations of elliptical elements, and, plausibly, subsequent collisions of neighboring planets and their accretion. We may conjecture that only the observation of many extra-solar systems might exhibit one day such transient behavior.
\end{itemize}

\subsection{Remark on Nekhoroshev theory and weak convexity}
\label{sec:Nekhoroshev}

Due to the proper degeneracy of the Keplerian approximation, standard Nekhoroshev theory does not apply in a straightforward way to the planetary problem. Yet it has been successfully extended to the planetary problem \cite{Nekhoroshev:1977, Niederman:1996} (see also \cite[6.3.4]{Arnold:2006}). In particular, Niederman proved the following conditional result regarding a Hamiltonian perturbation of a properly degenerate integrable system: provided that the actions in the degenerate (i.e. secular) directions remain in some bounded region, the actions conjugate to the fast angles are stable over an exponentially long time. He then showed that this model can be applied to the planetary problem. In the neighborhood of coplanar and circular ellipses (the maximum of the angular momentum), the conservation of the angular momentum prevents the degenerate actions (encoding eccentricities and inclinations) to undergo any substantial instability, so the actions conjugate to the fast angles (encoding the semimajor axes) are indeed stable over an exponentially long time.

The regimes of the $4$-body problem studied in the present paper differ from Niederman's work in two respects:
\begin{itemize}
\item In the planetary regime, the conservation of the angular momentum does not prevent secular variables from drifting because of the high inclination of the two inner planets.
\item In the hierarchical problem, the convexity of the fast, Keplerian part is weak, because of the large outer semimajor axes.
\end{itemize}
Hence neither Nekhoroshev theory nor Niederman's adaptation applies. (Incidentally, a proof \textit{ad absurdum} is that the conclusion of Theorem \ref{thm:main} would contradict Nekhoroshev theory.)

Regarding the weak convexity (for a numerical investigation, see \cite{Guzzo:2011}), let us mention the following open question. Consider the toy Hamiltonian
\[H(\theta,r) = r_1^2 + r_2^2 + \epsilon^\alpha r_3^2 + \epsilon \, f(\theta,r),\]
where $\epsilon \ll 1$ and $0 \leq \alpha \leq 1$. $H$ is Nekhoroshev-stable for $\alpha=0$ and, trivially, unstable for $\alpha=1$. But, more precisely, how does the radius of confinement of $r$ deteriorate as $\alpha$ grows from $0$ to $1$? The classical proof as well as more recent examples should provide a precise answer to this question. 

\subsection{Main ideas of the proof of Theorem \ref{thm:main} and moderately  versus strongly hierarchical regimes}\label{sec:moderatestsrong}

The orbits constructed in Theorem \ref{thm:main} rely on an Arnold diffusion mechanism \cite{arnold1964instability}. Progress in the understanding of Arnold diffusion in nearly-integrable Hamiltonian systems in these last decades has been remarkable, especially for two and a half degrees of freedom  (see \cite{Bernard08,Bolotin:1999,ChengY04,DelshamsLS00,DelshamsLS06b,delshams2006biggaps,DelshamsH05,gidea2006topological, Kaloshin:2020,moeckel2002drift,Treschev04}, or \cite{Kaloshin:2016,Cheng:2017,clarke2022arnold,MR3479576,Gelfreich:2008,GelfreichTuraev2017,MR4033892,Treschev:2012} for results in higher dimension). However, most of these results deal with generic nearly-integrable Hamiltonian systems in $C^r$ or $C^\infty$ regularity whereas results in the analytic category, including results in Celestial Mechanics, are rather scarce. See the discussion in Section 1.1 of \cite{clarke2022why} for more details. 

Indeed, even if Arnold in his seminal paper conjectured that his diffusion mechanism should be present in the 3-body problem, as far as the authors know, the \emph{only complete analytical} proofs of Arnold diffusion in celestial mechanics are \cite{clarke2022why,delshams2019instability, GuardiaPS23}.
Other works in the field rely on computer-assisted computations \cite{capinski2017diffusion,fejoz2016kirkwood}, on  computer-assisted proofs \cite{CapinskiGidea}, or on the assumption of a plausible transversality hypothesis~\cite{Xue:2014:4bp}.

In order to prove Theorem \ref{thm:main} we adapt what are usually referred to as the geometric and topological methods of Arnold diffusion. Although some of the geometric ideas could now be considered classical, others are contemporary (in particular a topological shadowing result proven recently by the same authors in \cite{clarke2022topological}). While explaining the overview of the proof, we will compare the moderately hierarchical regime considered in \cite{clarke2022why} with the strongly hierarchical regime considered in the present paper.

The classical geometric approach used to  prove Arnold diffusion both in the present paper and in \cite{clarke2022why} can be broken down into the following steps.
\begin{itemize}
\item Prove the existence of a normally hyperbolic invariant cylinder (See Appendix \ref{app:scattering} for the definition). The ``vertical'' components of the cylinder are the actions in which we want to drift.
\item Prove that the invariant manifolds of the cylinder intersect transversely along homoclinic channels. Orbits in the channel are heteroclinic orbits between different orbits in the cylinder. This is encoded in a \emph{scattering map} \cite{delshams2008geometric}. One can obtain an asymptotic formula for it through Poincar\'e-Melnikov Theory.
\item Construct an iterated function system consisting of the \emph{inner dynamics} and the scattering map, and show that its orbits (called \emph{pseudo-orbits}) display a drift in the action variables.
\item Use a shadowing argument  to obtain orbits which follow closely these pseudo-orbits.
\end{itemize}
To carry out these steps in the 4 body problem, both in \cite{clarke2022why} and in the present paper we consider the hierarchical regime which makes the 4-body problem nearly integrable. In \cite{clarke2022why} we consider what we call the \emph{moderately hierarchical} regime, where we assume 
\begin{equation}\label{def:moderate}
    a_1=O(1) \ll a_2 \ll a_3^{6/11} \ll a_2^{12/11}
\end{equation}
whereas in the present paper we consider the strongly hierarchical regime \eqref{def:regime}. Both regimes lead to a nearly integrable setting. However, they lead to   different  hierarchies of time scales, and to different first-order effective models. 
Let us describe them.  To this end, we express the 4-body problem in a good set of coordinates, which reduces the dimension of the model by eliminating its first integrals. First, we consider Jacobi coordinates to eliminate the translation invariance and then we use the Deprit coordinates to perform the symplectic reduction by rotational symmetry (see \cite{Chierchia:2011}). After this reduction, the 4-body problem becomes a Hamiltonian system with seven degrees of freedom. In Section \ref{sec:DepritResults} we state the main results of this paper in Deprit coordinates. 

The two faster frequencies of the 4-body problem in Deprit coordinates are,  in both the strongly and moderately hierarchical regimes, the mean anomalies of the first two planets. Moreover, they evolve at different time scales to one another. This implies that they can be averaged out up to arbitrarily high order in $a_2^{-1}$. If one ignores the higher order terms, one can reduce the dimension by 2 and end up with a five degree of freedom Hamiltonian depending on $a_1$ and $a_2$, which can be treated as parameters.

In the moderately hierachical regime \eqref{def:moderate}, the third fastest frequency is the mean anomaly of the third planet. Proceeding analogously, in that regime it too can be averaged out up to high order which leads to a 4 degree of freedom Hamiltonian. This Hamiltonian is usually called the secular Hamiltonian since it models the slow evolution of the osculating ellipses.

On the contrary, in the strongly hierachical regime, since the third planet is placed much further away, the third mean anomaly becomes slower and it cannot be averaged out. For this reason, in the present paper we analyse the 5 degree of freedom Hamiltonian which, by an abuse of language, we also call the secular Hamiltonian. It models the slow evolution of the osculating ellipses of the three planets plus the motion of the third planet on its osculating ellipse.

In both regimes, the next step is to expand the secular Hamiltonian in powers of $1/a_2$ and $a_2/a_3$ using the Legendre polynomials. This is done in Section \ref{sec_secularexpansion}. The first term in the expansion is the so-called quadrupolar Hamiltonian of the first two planets, which is integrable, and the second term is the so-called octupolar Hamiltonian which captures the next order of interaction between planets 1 and 2.  The subsequent orders in the expansion involve the interaction between planets 2 and 3. It is at these orders where the analysis of the moderately and strongly hierarchical regimes differs considerably.
In \cite{clarke2022why} we need both the quadrupolar and octupolar Hamiltonians associated to planets 2 and 3 whereas in the present paper the approximate dynamics does not depend on the octupolar term. The reason is that, since we do not average the mean anomaly $\ell_3$, the quadrupolar term adds ``more non-integrability'' to the model. Indeed, the  quadrupolar Hamiltonian depends on all the secular variables which was not the case in \cite{clarke2022why}. 

Next, we analyse the normally hyperbolic cylinder and its invariant manifolds for the secular Hamiltonian. The first appropriate approximation 
is that of the quadrupolar Hamiltonian of planets 1 and 2 (see Section \ref{section_analysisofh0}). It is well known that it posesses a hyperbolic singularity, which appears when the mutual inclination between planets 1 and 2 is larger than around 40 degrees. This hyperbolic singularity corresponds to a normally hyperbolic invariant cylinder in the full phase space. Moreover, the integrability implies that its stable and unstable manifolds coincide along a homoclinic manifold. 


Fenichel Theory \cite{fenichel1971persistence,fenichel1974asymptotic,fenichel1977asymptotic} implies that the normally hyperbolic invariant cylinder is persistent. In Section \ref{section_inner}, we analyse the induced dynamics of  the secular Hamiltonian on the cylinder, usually called the \emph{inner dynamics}. We prove that it is integrable up to an arbitrarily high order and that it has torsion, provided that the mutual inclination of planets $1$ and $2$ is larger than 55 degrees. Note that the cylinder  in the present paper has two dimensions more than the cylinder considered in  \cite{clarke2022why}, as the mean anomaly of planet 3 and the semimajor axis $a_3$ provide additional inner variables.

The results in \cite{fejoz2016secular} combined with classical perturbation techniques imply that the stable and unstable invariant manifolds of the cylinder of the secular Hamiltonian intersect transversely along two homoclinic channels. Orbits in these channels are heteroclinic orbits which ``connect'' different points in the cylinder. Such connections are encoded in the scattering maps (see \cite{delshams2008geometric} and Appendix \ref{app:scattering} for the definition). Section \ref{sec_outer} is devoted to the computation of the first order of these maps by means of Poincar\'e-Melnikov Theory. 

Once the inner dynamics and the outer dynamics (i.e. the scattering maps) have been analysed, the last step is to combine them to achieve drift in the actions. This is done in Section \ref{sec:shadowing}. First, we construct pseudo-orbits (i.e. orbits of the iterated function system consisting of a Poincar\'e map induced by the inner dynamics and the two scattering maps) that follow any prescribed itinerary in the actions such that the scattering maps map ``approximately invariant tori'' of the inner dynamics transversely across other such tori. Then, referring to an argument contained in a previous paper by the authors \cite{clarke2022topological} which provides rather flexible shadowing results, we show that there are orbits of both the secular Hamiltonian and the four-body problem Hamiltonian which follow closely the pseudo-orbits. Moreover, the shadowing methods in \cite{clarke2022topological} allow us also to obtain time estimates.

Finally, in Section \ref{sec:planetary} we explain how to deal with the planetary regime where the masses of the three planets are arbitrarily small.

\subsection*{Acknowledgments}
The authors are grateful to Laurent Niederman for an illuminating discussion regarding Nekhoroshev estimates. 

A. Clarke and M. Guardia are supported by the European Research Council (ERC) under the European Union's Horizon 2020 research and innovation programme (grant agreement No. 757802). This work is part of the
grant PID-2021-122954NB-100 funded by MCIN/AEI/10.13039/501100011033 and ``ERDF A way of making
Europe''. M. Guardia is also supported by the Catalan Institution for Research and Advanced Studies via an ICREA Academia Prize 2019. This work is also supported by the Spanish State Research Agency, through the Severo Ochoa and María de Maeztu Program for Centers and Units of Excellence in R\&D (CEX2020-001084-M). This work is also partially supported by the project of the French Agence Nationale pour la Recherche CoSyDy (ANR-CE40-0014).


\section{Main results in Deprit coordinates}\label{sec:DepritResults}
The first step towards a proof of Theorem~\ref{thm:main} is to find a suitable set of coordinates in which we can analyse the problem. In particular, as is well-known, the 4-body problem has many symmetries which can be exploited to reduce the dimension of the phase space. To this end, in Section \ref{sec:coordinates} we explain how to express the 4-body problem in Jacobi coordinates, thus reducing by translational symmetry, and then pass to Deprit coordinates to reduce by rotational symmetry. In Section \ref{sec:thm:deprit} we state a more detailed version of Theorem \ref{thm:main} in Deprit coordinates.

\subsection{The Jacobi and Deprit coordinates}\label{sec:coordinates}
The 4-body problem is a Hamiltonian system with respect to the Hamiltonian 
\begin{equation}\label{eq_4bpham}
	H= \sum_{0 \leq j \leq 3} \frac{y_j^2}{2m_j} - \sum_{0 \leq i < j \leq 3} \frac{m_i m_j}{\| x_j - x_i \|},
\end{equation}
and the symplectic form $\Omega = dq \wedge dp$ where $x_j\in\RR^3$ is the position of body $j$ and $y_j\in\RR^3$ its conjugate linear momentum.

The Jacobi coordinates  $(q_j,p_j) \in \mathbb{R}^3 \times \mathbb{R}^3$, $j=0,1,2,3$, are defined as 
\begin{equation}
	\begin{cases}
		q_0 = x_0 \\
		q_1 = x_1 - x_0 \\
		q_2 = x_2 - \sigma_{01} \, x_0 - \sigma_{11} \, x_1 \\
		q_3 = x_3 - \sigma_{02} \, x_0 - \sigma_{12} \, x_1 - \sigma_{22} \, x_2
	\end{cases}
	\quad
	\begin{cases}
		p_0 = y_0 + y_1 + y_2 + y_3 \\
		p_1 = y_1 + \sigma_{11} \, y_2 +\sigma_{11} \, y_3 \\
		p_2 = y_2 + \sigma_{22} \, y_3 \\
		p_3 = y_3.
	\end{cases}
\end{equation}
where
\begin{equation}\label{def:sigmaij}
\sigma_{ij} = \frac{m_i}{M_j} \qquad\text{and}\qquad	M_j = \sum_{i=0}^j m_i.
\end{equation}
%
A direct computation implies that this transformation is symplectic, in the sense that $dq \wedge dp = dx \wedge dy$. The Hamiltonian \eqref{eq_4bpham} expressed in these coordinates does not depend on  $q_0$, and therefore $p_0$ is a first integral. Without loss of generality, we may restrict to $p_0=0$ and consider the reduced phase space with coordinates $(q_j,p_j)_{j=1,2,3}$. Then, the Hamiltonian \eqref{eq_4bpham} becomes
\begin{equation}\label{def:HKepplusper}
	H = F_{\mathrm{Kep}} + F_{\mathrm{per}}
\end{equation}
where
\begin{align}
	F_{\mathrm{Kep}} = &\sum_{j=1}^3 \left( \frac{p_j^2}{2 \mu_j} - \frac{\mu_j M_j}{\| q_j \|} \right) \label{eq_fkep}\\
\begin{split}
	F_{\mathrm{per}} ={}& \sum_{j=2}^3 \frac{\mu_j M_j}{\| q_j \|} - \frac{m_0 \, m_2}{\| q_2 + \sigma_{11} \, q_1 \|} - \frac{m_0 \, m_3}{\| q_3 + \sigma_{22} \, q_2 + \sigma_{11} \, q_1 \|} - \frac{m_1 \, m_2}{\| q_2 - \sigma_{01} \, q_1 \|}  \\
	& - \frac{m_1 \, m_3}{\| q_3 + \sigma_{22} \, q_2 + (\sigma_{11} - 1) \, q_1 \|} - \frac{m_2 \, m_3}{\| q_3 + (\sigma_{22} - 1) \, q_2 \label{eq_fper}
		\|}
\end{split}
\end{align}
with the reduced masses $\mu_j$ defined, for each $j=1,2,3$, by
\begin{equation}
	\mu_j^{-1} = M_{j-1}^{-1} + m_j^{-1}.
\end{equation}

The next step is to pass to Deprit coordinates, which are well suited to the symmetry of rotations. These coordinates were discovered originally by Deprit \cite{deprit1983}, but their efficacy in the $N$-body problem was noticed only recently by Chierchia and Pinzari \cite{chierchia2011deprit}. Let us denote by
\begin{equation}
	C_j = q_j \times p_j
\end{equation} 
the angular momentum of the $j^{\mathrm{th}}$ fictitious  body and let 
\begin{equation}
	C = C_1 + C_2 + C_3
\end{equation}
be the  total angular momentum vector.
Let $k_j$ be the $j^{\mathrm{th}}$ element of the standard orthonormal basis of $\mathbb{R}^3$ and define the nodes $\nu_j$ by
\begin{equation}
	\nu_1=\nu_2=C_1 \times C_2, \quad \nu_3 = (C_1 + C_2) \times C_3, \quad \nu_4=k_3 \times C.
\end{equation}
For a non-zero vector $z \in \mathbb{R}^3$ and two non-zero vectors $u,v$ lying in the plane orthogonal to $z$, denote by $\alpha_z (u,v)$ the oriented angle between $u,v$, with orientation defined by the right hand rule with respect to $z$. Denote by $\Pi_j$ the pericenter of $q_j$ on its Keplerian ellipse.
The Deprit variables $(\ell_j,L_j,\gamma_j,\Gamma_j,\psi_j,\Psi_j)_{j=1,2,3}$ are defined as follows:
\begin{itemize}
	\item $\ell_j$ is the mean anomaly of $q_j$ on its Keplerian ellipse;
	\item
	$L_j = \mu_j \sqrt{M_j a_j}$;
	\item
	$\gamma_j= \alpha_{C_j}(\nu_j,\Pi_j)$;
	\item
	$\Gamma_j = \left\| C_j \right\|$;
	\item
	$\psi_1 = \alpha_{(C_1+C_2)}(\nu_3,\nu_2)$, $\psi_2=\alpha_C(\nu_4,\nu_3)$, $\psi_3 = \alpha_{k_3} (k_1, \nu_4)$;
	\item
	$\Psi_1 = \| C_1 + C_2 \|$, $\Psi_2 = \| C_1 + C_2 +C_3 \| = \| C \|$, $\Psi_3 = C \cdot k_3$.
\end{itemize}
The Deprit variables are analytic and symplectic over the open subset $\cD$ in which the $3$ terms of $F_{\mathrm{Kep}}$ are negative, the eccentricities of the Keplerian ellipses belong to $(0,1)$ and the nodes $\nu_j$ are nonzero (see \cite{chierchia2011deprit,deprit1983} or \cite[Appendix A]{clarke2022why}). Actions $\Psi_2$ and $\Psi_3$ are commuting first integrals. 

The orbital elements can be expressed in terms of Deprit cordinates:
\begin{itemize}
\item The osculating eccentricities are defined by
\begin{equation}\label{def:eccentricity}
	e_j=\sqrt{1-\frac{\Gamma_j^2}{L_j^2}},\qquad j=1,2,3.
\end{equation}
\item The mutual inclination $i_{12}$ between planets $1$ and $2$, measured as the oriented angle between $C_1$ and $C_2$, is defined via its cosine by
\begin{equation}\label{def:inclinationi12}
	\cos i_{12} = \frac{\Psi_1^2 - \Gamma_1^2 - \Gamma_2^2}{2 \, \Gamma_1 \, \Gamma_2}.
\end{equation}
\item The mutual inclination $i_{23}$ between planet 3 and the inner two planets, measured as the oriented angle between $S_1 = C_1 + C_2$ and $C_3$, is defined via its cosine by
\begin{equation}\label{def:inclinationi23}
	\cos i_{23}=\frac{\Psi_2^2 - \Gamma_3^2-\Psi_1^2}{2\Gamma_3\Psi_1}.
\end{equation}
\end{itemize}

\subsection{Arnold diffusion in Deprit coordinates}\label{sec:thm:deprit}
In this section we give a precise reformulation of Theorem \ref{thm:main} in terms of the Deprit coordinates. 
To this end, let us recall that we assume the masses $m_0,m_1,m_2,m_3>0$ are fixed and satisfy $m_0\neq m_1$. We consider a regime of increasingly separated bodies. In terms of the semimajor axes, we assume that 
\begin{equation}\label{eq_mainassumption}
  a_1\ll a_2 \ll a_3^{1/3}
\end{equation}
which, in terms of Deprit, reads
\begin{equation}\label{eq_mainassumptiondeprit}
L_1\ll L_2 \ll L_3^{1/3}.
\end{equation}
We assume that the eccentricities of the bodies are uniformly bounded away from 0 and 1, and therefore 
the other Deprit actions $\Gamma_j$, $j=1,2,3$, and $\Psi_j$, $j=1,2$, have significantly different sizes. 
Indeed they satisfy
\begin{equation}\label{def:psigamma}
	\Gamma_i\sim L_i\quad \text{for}\quad i=1,2,3\qquad\text{and}\qquad \Psi_i\sim \Gamma_{i+1}\qquad \text{for}\qquad i=1,2. 
\end{equation}

Let us explain which actions may drift under these assumptions. The semimajor axes $L_1$ and $L_2$ are almost constant due to the fact that the conjugate angles $\ell_1$, $\ell_2$ evolve faster than all other variables, and on different time scales to one another. As a result, $\ell_1$, $\ell_2$ can be averaged out of the Hamiltonian up to arbitrarily high order and any splitting of separatrices in the $L_1, L_2$ directions is exponentially small. Moreover, recall that $\Psi_2$ is the total angular momentum, which is a first integral. Since $\Gamma_1,\Gamma_2\ll\Gamma_3$ this implies that $\Gamma_3\sim\Psi_2$ is almost constant. However in this case, as in \cite{clarke2022why}, relatively small proportions of angular momentum can transfer from $\Gamma_3$ to $\Gamma_2$ to create a significant change in the orbital elements of the second planet. Indeed, $\Gamma_3$ can drift from 
\begin{equation}\label{def:Gamma3size}
	\Gamma_3\sim \Psi_2-\Psi_1 \qquad\text{to}\qquad  \Gamma_3\sim \Psi_2+\Psi_1.
\end{equation}
This corresponds to having $C_3$ and $C_1+C_2$ close to parallel and either with the same sign or with opposite sign. That is, this evolution in $\Gamma_3$ implies a large drift in the inclination $i_{23}$ (see \eqref{def:inclinationi23}), or equivalently in $\theta_{23}$ (as defined in Section \ref{sec:mainresults}).

Since $\Gamma_2 \ll \Gamma_3$, this transfer of angular momentum between bodies can cause dramatic changes in $\Gamma_2 \in(0,L_2)$. Indeed, we are able to show that it drifts from 
\begin{equation}\label{def:Gamma2size}
	\Gamma_2\sim L_2  \qquad\text{to}\qquad \Gamma_2\sim  0.
\end{equation}
Equivalently, the orbital ellipse of the second planet can swing from being near-circular ($e_2\sim 0$) to being highly eccentric ($e_2\sim 1$). 

Finally, note that $L_3$ must satisfy $L_3>\Gamma_3\sim\Psi_2$. Moreover, by taking $L_3\to+\infty$ while $\Psi_2$ is fixed (recall that it is a first integral) one has that $e_3\to 1$. Since we are considering a  hierarchical regime, where the orbits of the bodies are increasingly (and uniformly) separated, one has to constrain the possible growth of $L_3$. If one fixes $0<\kk\ll 1$, then one can consider $L_3$ transitioning from 
\begin{equation}\label{eq_l3driftrange}
L_3\sim (1+\kk)\Psi_2\qquad \text{to}\qquad L_3\sim \frac{\Psi_2}{\kk}.
\end{equation}
Since $\Gamma_3$ is almost constant, this is equivalent to changing at the same time the semimajor axis as above (i.e. transitioning from $a_3 \sim (1 + \kappa)^2 \Psi_2^2$ to $a_3 \sim \frac{\Psi_2^2}{\kappa^2}$) and the eccentricity $e_3$ as 
\[
e_3\sim \sqrt{\kk}\qquad \text{to}\qquad e_3\sim \sqrt{1-\kk^2}.
\]

The next theorem shows that such transitions are possible and that one  can freely  vary $\Gamma_2$, $\Gamma_3$ and $L_3$ within their ``allowed'' ranges.

\begin{theorem}\label{thm:MainHierarch:Deprit}
	Fix masses $m_0, m_1, m_2, m_3>0$ such that 
	\begin{equation}\label{def:massconditions:2}
		m_0\neq m_1. 
	\end{equation}
	There exists $\xi$ with $0 < \xi \ll 1$ and $\alpha_1,\alpha_2,\beta>0$ such that the following is satisfied.

	Fix $N\geq 1$ any $\{\nu^k\}_{k=0}^N\subset (0,1)$, $\{\eta^k\}_{k=0}^N\subset (-1,1)$,  $\{\zeta^k\}_{k=0}^N\subset (1,+\infty)$ and constants 
	$L_1^0$, $L_2^0$ and $\Psi_2^0$  satisfying 
	\begin{equation}\label{eq_semimajorassumptionsinlvariable}
	L_1^0\in \left[1 - \xi,1 + \xi\right],\qquad	L_1^0\ll L_2^0 \qquad
	(L_2^0)^3\ll \Psi_2^0\qquad\text{and}\qquad 
		|\Psi_1^0-\nu_0L_2^0|\leq \frac{L_1^0}{\sqrt{3}} + \xi.
		\end{equation}
	Then, 
	there exists an orbit 
	of the Hamiltonian $H$ in \eqref{def:HKepplusper} expressed in Deprit coordinates  and  times $\{t_k\}_{k=0}^N$ satisfying
	\[
	t_0=0 \qquad\text{and}\qquad  |t_k|\leq \left(L_2^0\right)^{\alpha_1}\left(\Psi_2^0\right)^{\alpha_2},\, k\geq 1
	\]
	such that
	\[
\begin{aligned}
		|\Gamma_2(t_k)-\nu_k L_2^0|\leq  (L_2^0)^{-\beta}\\
		|\Psi_2^0-\Gamma_3(t_k)-\eta_k\Psi_1(t_k)|\leq  (L_2^0)^{-\beta}\\
	   |L_3(t_k)-\zeta_k\Psi_2^0|\leq  (L_2^0)^{-\beta}.
 \end{aligned}
	\]
	Moreover,
	\[
		|\Gamma_1(t_k)-L_1^0|\leq  (L_2^0)^{-\beta}
		,\qquad\qquad
		|\Psi_1(t_k)-\Gamma_2(t_k)-M_k|\leq  (L_2^0)^{-\beta}
	\]
	where $M_k\in \left(0,\frac{L_1^0}{\sqrt{3}} + \xi \right)$ is determined by
	\[
	\frac{M_k^2}{(L_2^0-\Gamma_2(t_k))^{3/2}}=\frac{M_0^2}{(L_2^0-\Gamma_2^0)^{3/2}}  \qquad \text{and}\qquad M_0=\Psi_1^0-\Gamma_2^0
	\]
	whereas for all $t\in [0,t_N]$,
	\[
		|\Gamma_3(t)-\Gamma_3^0|\leq 2L_2^0
		, \qquad  \text{and}\qquad 
		|L_j(t)-L_j^0|\leq (L_2^0)^{-\beta}
		\quad \text{for}\quad j=1,2.
	\]
	
\end{theorem}

Theorem \ref{thm:MainHierarch:Deprit} is proved in Sections \ref{sec_secularexpansion}-\ref{sec:shadowing}.
It pertains to fixed values of the masses, in the sense that the increasing separation of the semimajor axes depend on the mass choices. 
We now want to obtain an analogous statement in the planetary regime, namely when bodies 1, 2, and 3 are assumed to have small mass. More concretely, $ m_0\sim 1$ and $m_i=\rr \, \tilde m_i$ for $i=1,2,3$ with $\tilde m_i\sim 1$ and $0<\rr\ll 1$. 

To deal with the planetary regime we consider scaled Deprit coordinates. Indeed, for fixed semimajor axes, the  Deprit actions all have size $\rr$. Then, to be able to capture their  drift along the diffusing orbits, it is convenient to perform the conformally symplectic scaling
\begin{equation}\label{def:planetaryscaling}
L=\rr \check{L}, \quad  \Gamma=\rr\check{\Gamma}, \quad  \Psi=\rr \check{\Psi}, \quad  
\end{equation}
%
%
\begin{theorem}\label{thm:main:planetary:Deprit}
Fix $m_0,\wt m_1,\wt m_2,\wt m_3>0$  and consider the Hamiltonian $H$ in \eqref{def:HKepplusper} expressed in Deprit coordinates with  masses $m_0, m_j=\rr \wt m_j$ with $j=1,2,3$.  Then, there exists $0<\kappa\ll 1$, $\alpha_1,\alpha_2,\beta>0$, such that the following is satisfied.
	
Fix $N\geq 1$ any $\{\nu^k\}_{k=0}^N\subset (0,1)$, $\{\eta^k\}_{k=0}^N\subset (-1,1)$,  $\{\zeta^k\}_{k=0}^N\subset (1,+\infty)$ and constants 
$\check{L}_1^0$, $\check{L}_2^0$ and $\check{\Psi}_2^0$  satisfying 
\begin{equation}\label{eq_semimajorassumptionsinlvariable2}
\check{L}_1^0\in \left[\frac12,2\right],\qquad	\check{L}_1^0\ll \check{L}_2^0, \qquad  \check{\Psi}_2^0\gg\rr^{-1/3}\qquad \text{and}\qquad |\check{\Psi}_1^0-\nu_0\check{L}_2^0|\leq \kk.
\end{equation}
Then, 
	there exists an orbit 
	of the Hamiltonian $H$ in \eqref{def:HKepplusper} expressed in scaled Deprit coordinates  and  times $\{t_k\}_{k=0}^N$ satisfying
	\[
	t_0=0 \qquad\text{and}\qquad  |t_k|\leq C(\check{L}_2^0)\rr^{-35/3},\, k\geq 1,
	\]
where $C(\check{L}_2^0)$ is a constant depending on $\check{L}_2^0$ but independent of $\rr$,	such that
	\[
\begin{aligned}
		|\check{\Gamma}_2(t_k)-\nu_k \check{L}_2^0|\leq  (\check{L}_2^0)^{-\beta}\\
		|\check{\Psi}_2^0-\check{\Gamma}_3(t_k)-\eta_k\check{\Psi}_1(t_k)|\leq  (\check{ L}_2^0)^{-\beta}\\
	   |\check{L}_3(t_k)-\zeta_k\check{\Psi}_2^0|\leq  (\check{L}_2^0)^{-\beta}.
 \end{aligned}
	\]
	Moreover,
	\[
		| \check{\Gamma}_1(t_k)-\check{L}_1^0|\leq  (\check{L}_2^0)^{-\beta}
		,\qquad\qquad
		|\check{\Psi}_1(t_k)-\check{\Gamma}_2(t_k)-M_k|\leq  (\check{L}_2^0)^{-\beta}
	\]
	where $M_k\in (0,\kappa)$ is determined by
	\[
	\frac{M_k^2}{(\check{L}_2^0-\check{\Gamma}_2(t_k))^{3/2}}=\frac{M_0^2}{(\check{L}_2^0-\check{\Gamma}_2^0)^{3/2}}  \qquad \text{and}\qquad M_0=\check{\Psi}_1^0-\check{\Gamma}_2^0
	\]
	whereas for all $t\in [0,t_N]$,
	\[
		|\check{\Gamma}_3(t)-\check{\Gamma}_3^0|\leq 2\check{L}_2^0
		, \qquad  \text{and}\qquad 
		|\check{L}_j(t)-\check{L}_j^0|\leq (\check{L}_2^0)^{-\beta}
		\quad \text{for}\quad j=1,2.
	\]
\end{theorem}

\section{Computation of the secular Hamiltonian} \label{sec_secularexpansion}

In this section we compute the secular Hamiltonian in three steps: first, we expand the perturbing function \eqref{eq_fper} using the Legendre polynomials; next, we observe that the mean anomalies $\ell_1$, $\ell_2$ are faster than the other variables, and we use this fact to perform a near-identity symplectic coordinate transformation that averages the angles $\ell_1$, $\ell_2$ out of the perturbing function up to arbitrarily high order; finally, we make a further symplectic coordinate transformation so that the new action variables are all of order 1, and we use these variables to expand the Taylor series of the secular Hamiltonian. 

This section has strong similarities with the computation of the secular Hamiltonian in~\cite{clarke2022why}. We do include it for the convenience of the reader and because terms describing the motion of planet $3$ differ significantly.

\subsection{Expansion of the perturbing function in Legendre polynomials}

Since $\| q_j \| = O(a_j) = O( L_j^2)$, the assumption \eqref{eq_mainassumption} implies that $\| q_1 \| \ll \| q_2 \| \ll \| q_3 \|$. Denote by $\zeta_j$ the angles between $q_j$ and $q_{j+1}$ for $j=1,2$, and denote by $P_n$ the Legendre polynomial of degree $n$. Observe that we can write the perturbing function $F_{\mathrm{per}}$, defined by \eqref{eq_fper}, as
\begin{equation}
F_{\mathrm{per}} = F_{\mathrm{per}}^{12} + F_{\mathrm{per}}^{23} + O \left( \frac{1}{a_3^3} \right)
\end{equation}
where
\begin{equation}
F_{\mathrm{per}}^{12} =\frac{\mu_2 M_2}{\| q_2 \|} - \frac{m_0 \, m_2}{\| q_2 + \sigma_{11} \, q_1 \|} - \frac{m_1 \, m_2}{\| q_2 - \sigma_{01} \, q_1 \|} 
=- \frac{\mu_1 m_2}{\| q_2 \|} \sum_{n=2}^{\infty} \tilde{\sigma}_{1,n} P_n (\cos \zeta_{1}) \left( \frac{\| q_1 \|}{\| q_2 \|} \right)^n \label{eq_perfn12}
\end{equation}
is the perturbing function of the inner 3-body problem, and where 
\begin{equation} \label{eq_perfn23}
F_{\mathrm{per}}^{23} = - \frac{\mu_2 m_3}{\| q_3 \|} \sum_{n=2}^{\infty} \tilde{\sigma}_{2,n} P_n (\cos \zeta_{2}) \left( \frac{\| q_2 \|}{\| q_3 \|} \right)^n
\end{equation}
describes the interactions between bodies 2 and 3, with
\begin{equation}
\tilde{\sigma}_{1,n} = \sigma_{01}^{n-1} + (-1)^n \sigma_{11}^{n-1}, \quad \tilde{\sigma}_{2,n} = (\sigma_{02} + \sigma_{12})^{n-1} + (-1)^n \sigma_{22}^{n-1},
\end{equation}
and $\sigma_{ij}$ are defined in \eqref{def:sigmaij}.

\subsection{Averaging of the mean anomalies $\ell_1$ and $\ell_2$}

In Deprit coordinates the Kepler Hamiltonian $F_{\mathrm{Kep}}$, defined by \eqref{eq_fkep}, is given by 
\begin{equation}\label{eq_fkepdeprit}
F_{\mathrm{Kep}} = - \sum_{j=1}^3 \frac{\mu_j^3 \, M_j^2}{2 \, L_j^2}.
\end{equation}
From Hamilton's equations of motion we see that the first order of $\dot \ell_j$ is $\frac{\partial F_{\mathrm{Kep}}}{\partial L_j}  = \frac{\mu_j^3 \, M_j^2}{L_j^3}$. Since the first order term in $F_{\mathrm{per}}$ is of order $\frac{\|q_1\|^2}{\| q_2\|^3} = O \left( L_2^{-6} \right)$, it follows that the angles $\ell_1$, $\ell_2$ are faster than all other variables. Therefore standard averaging arguments imply that we can perform a near-identity coordinate transformation so that, in the new coordinates, the Hamiltonian does not depend on the angles $\ell_1$, $\ell_2$ up to arbitrarily (but finitely) high order terms. Effecting the coordinate transformation, the Hamiltonian $H= F_{\mathrm{Kep}} + F_{\mathrm{per}}$ becomes
\begin{equation}\label{eq_4bphamave}
F = F_{\mathrm{Kep}} + \tilde F_{\mathrm{sec}} + \frac{1}{L_2^{10}} R_1 + \frac{1}{L_3^6} R_2,
\end{equation}
where $F_{\mathrm{Kep}}$ is given by \eqref{eq_fkepdeprit}, the Hamiltonian $\tilde F_{\mathrm{sec}}$ is defined by 
\begin{equation} \label{eq_fsectilde}
\tilde F_{\mathrm{sec}} = F_{\mathrm{sec}}^{12} + F_{\mathrm{sec}}^{23} + O \left( \frac{1}{L_3^6} \right),
\end{equation}
with
\begin{align}
F_{\mathrm{sec}}^{12} =& \frac{1}{\left( 2 \pi \right)^2} \int_{\mathbb{T}^2} F_{\mathrm{per}}^{12} \, d \ell_1 \, d \ell_2 =  - \frac{1}{\left( 2 \pi \right)^2} \frac{\mu_1 m_2}{\| q_2 \|} \sum_{n=2}^{\infty} \tilde{\sigma}_{1,n} \int_{\mathbb{T}^2} P_n (\cos \zeta_{1}) \left( \frac{\| q_1 \|}{\| q_2 \|} \right)^n \, d \ell_1 \, d \ell_2  \label{eq_sec12def}\\
 F_{\mathrm{sec}}^{23} =& \frac{1}{ 2 \pi } \int_{\mathbb{T}} F_{\mathrm{per}}^{12} d \ell_2 = - \frac{1}{ 2 \pi } \frac{\mu_2 m_3}{\| q_3 \|} \sum_{n=2}^{\infty} \tilde{\sigma}_{2,n} \int_{\mathbb{T}} P_n (\cos \zeta_{2}) \left( \frac{\| q_2 \|}{\| q_3 \|} \right)^n \, d \ell_2. \label{eq_sec23def}\\
\end{align}
Moreover the remainder term $R_1$ depends only on the variables on which $F_{\mathrm{per}}^{12}$ depends. 

\begin{remark}\label{remark_quadoctterminology}
We use the following terminology and notation. 
\begin{enumerate}
\item
Using the usual terminology from the literature, we refer to the terms in the expansion of $F_{\mathrm{sec}}^{12}$ obtained by setting $n=2,3$ in \eqref{eq_sec12def} as the quadrupolar, octupolar (respectively) Hamiltonians of the interaction between bodies 1 and 2, and we write $F_{\mathrm{quad}}^{12}$, $F_{\mathrm{oct}}^{12}$ (respectively) to denote these Hamiltonians. 
\item
In addition, we refer to the $n=2$ term in the expansion \eqref{eq_sec23def} of $F_{\mathrm{sec}}^{23}$ as the quadrupolar Hamiltonian of the interaction between bodies 2 and 3, and we write $F_{\mathrm{quad}}^{23}$ to denote this Hamiltonian. Note that this terminology is generally reserved in the literature for the term obtained by averaging both $\ell_2$ and $\ell_3$ (see for example \cite{clarke2022why}); however, in this paper, the angle $\ell_3$ is slower than, for example $\gamma_1$ (see Proposition \ref{proposition_secularexpansion} below), and so it cannot be averaged from the perturbing function. We therefore consider this terminology and notation to be appropriate in this instance. 
\end{enumerate}
\end{remark}

In this paper we require only the Hamiltonians $F_{\mathrm{quad}}^{12}$, $F_{\mathrm{oct}}^{12}$, $F_{\mathrm{quad}}^{23}$. Expanding the first two terms of \eqref{eq_sec12def} and \eqref{eq_sec23def} and using the notation of Remark \ref{remark_quadoctterminology} we obtain
\begin{equation}\label{eq_sec12sec23exp}
F_{\mathrm{sec}}^{12} = - \frac{\mu_1  m_2}{ \left( 2 \pi \right)^2} \left( F_{\mathrm{quad}}^{12} + \tilde \sigma_{1,3} \, F_{\mathrm{oct}}^{12} + O \left( \frac{a_1^4}{a_2^5} \right) \right), \quad F_{\mathrm{sec}}^{23} = - \frac{\mu_2  m_3}{  2  \pi } \left( F_{\mathrm{quad}}^{23}  + O \left( \frac{a_2^3}{a_3^4} \right) \right),
\end{equation}
where
\begin{equation}
\begin{dcases}
F_{\mathrm{quad}}^{12} =& \int_{\mathbb{T}^2} P_2 \left( \cos \zeta_1 \right) \, \frac{ \| q_1 \|^2}{ \| q_2 \|^3} \, d \ell_1 \, d \ell_2, \quad F_{\mathrm{oct}}^{12} = \int_{\mathbb{T}^2} P_3 \left( \cos \zeta_1 \right) \, \frac{ \| q_1 \|^3}{ \| q_2 \|^4} \, d \ell_1 \, d \ell_2 \\
F_{\mathrm{quad}}^{23} =& \int_{\mathbb{T}} P_2 \left( \cos \zeta_2 \right) \, \frac{ \| q_2 \|^2}{ \| q_3 \|^3} \, d \ell_2
\end{dcases}
\end{equation}
since $\tilde{\sigma}_{j,2}=1$ for $j=1,2$. In the following Lemma (proved in \cite{clarke2022why}), we compute $F_{\mathrm{quad}}^{12}$ and $F_{\mathrm{oct}}^{12}$ explicitly in terms of Deprit coordinates; we could perform similar computations to compute $F_{\mathrm{quad}}^{23}$, but the resulting expression would be very long. Instead, we compute $F_{\mathrm{quad}}^{23}$ in Section \ref{sec_seexp} via a Taylor expansion, after making a suitable linear symplectic coordinate transformation. 

\begin{lemma}\label{lemma_quadoct12comp}
The quadrupolar and octupolar Hamiltonians of bodies 1 and 2 are given by
\begin{equation} \label{eq_quad12unscaled}
F_{\mathrm{quad}}^{12}= \frac{a_1^2}{8 \, a_{2}^3 \, (1-e_{2}^2)^{\frac{3}{2}}} \left( \left( 15 \, e_1^2 \cos^2 \gamma_1 - 12 \, e_1^2 - 3 \right) \sin^2 i_{12} + 3 e_1^2 + 2 \right)
\end{equation}
and 
\begin{equation} \label{eq_oct12unscaled}
\begin{split} 
F_{\mathrm{oct}}^{12} ={}& - \frac{15}{64} \frac{a_1^3}{a_2^4} \frac{e_1 \, e_2}{\left(1 -e_2^2 \right)^{\frac{5}{2}}} \\
& \times \left\{
\begin{split}
\cos \gamma_1 \, \cos \gamma_2 \left[
\begin{split}
\frac{\Gamma_1^2}{L_1^2} \left(5 \, \sin^2 i_{12} \left(6-7 \cos^2 \gamma_1 \right) - 3 \right) \\
-35 \, \sin^2 \gamma_1 \, \sin^2 i_{12} + 7
\end{split}
\right] \\
+ \sin \gamma_1 \, \sin \gamma_2 \, \cos i_{12} \left[
\begin{split}
\frac{\Gamma_1^2}{L_1^2} \left(5 \, \sin^2 i_{12} \left(4 - 7 \, \cos^2 \gamma_1 \right) - 3 \right) \\
- 35 \sin^2 \gamma_1 \, \sin^2 i_{12} + 7
\end{split}
\right]
\end{split}
\right\} 
\end{split}
\end{equation}
respectively, where the eccentricity $e_j$ of the $j^{\mathrm{th}}$ Keplerian ellipse is defined by \eqref{def:eccentricity}, and where $i_{12}$ is the mutual inclination of Keplerian bodies $1$ and $2$, defined by
\begin{equation}\label{eq_cosi12def}
\cos i_{12} = \frac{\Psi_1^2 - \Gamma_1^2 - \Gamma_2^2}{2 \, \Gamma_1 \, \Gamma_2}.
\end{equation}
\end{lemma}

\subsection{Taylor expansion of the secular Hamiltonian}\label{sec_seexp}

In order to perform the subsequent analysis, we divide the phase space into strips where the actions live in some bounded region. In each strip we perform an affine coordinate transformation so that the new actions have order 1, thus allowing us to perform a further Taylor expansion of the secular Hamiltonian. 

Recall we assume that the semimajor axes satisfy \eqref{eq_mainassumption}, which in Deprit coordinates corresponds to the assumption \eqref{eq_mainassumptiondeprit} (if the masses are assumed to be fixed). Now, the variables $\Gamma_2, \Psi_1$ are of order $L_2$, while $\Gamma_3, \Psi_2$ are of order $L_3$. Fix some large positive value $L_3^*$ of $L_3$. The total angular momentum $ \Psi_2$ is conserved, so we write
\begin{equation}\label{eq_totalangmomdelta2}
\Psi_2 = \delta_2 \, L_3^*
\end{equation}
for some fixed $\delta_2>0$. We  make the symplectic change of variables:
\begin{equation} \label{eq_changeofcoordstilde}
	\begin{dcases}
		L_3 = L_3^* + \tilde L_3, \quad & \tilde \ell_3 = \ell_3\\
		\tilde{\Psi}_1 = \Psi_1 - \delta_1 \, L_2, \quad & \tilde{\psi}_1 = \psi_1 + \gamma_2 \\
		\tilde{\Gamma}_2=\Psi_1-\Gamma_2, \quad & \tilde{\gamma}_2=- \gamma_2 \\
		\tilde{\Gamma}_3 = \Psi_2 - \Gamma_3 - \delta_3 \, L_2, \quad & \tilde{\gamma}_3 =-  \gamma_3 
	\end{dcases}
\end{equation}
where $\delta_1, \delta_3 >0$ are constant with respect to the Hamiltonian $\tilde F_{\mathrm{sec}}$. Note that this symplectic transformation does not modify the variables $\gamma_1,\Gamma_1$ (or indeed $\ell_3$). We assume that $L_3^*$ is chosen so that $\tilde L_3 = O(1)$. 

Moreover, we assume that
\begin{equation} \label{eq_gamma2positive}
	\tilde{\Gamma}_2 > 0
\end{equation}
as the case where $\tilde{\Gamma}_2$ is negative can be treated analogously.
Furthermore, we assume that the new actions $\Gamma_1, \tilde{\Gamma}_2, \tilde{\Gamma}_3, \tilde{\Psi}_1$ live in a compact set away from the origin which is independent of $L_2$ and $L_3^*$. 

\begin{remark}\label{remark_coordtransftilde}
We make the following remarks regarding this coordinate transformation and notation. 
\begin{enumerate}
\item
By choosing different values of the constants $\delta_1, \delta_2, \delta_3$ (and adjusting $L_3^*$ so that \eqref{eq_totalangmomdelta2} still holds) we can focus on any relevant region of the phase space, in order to make the coordinate transformation \eqref{eq_changeofcoordstilde}. This fact will be of importance in Section \ref{section_shadowingsubsec}, where we construct trajectories that drift through many of these different regions. 
\item
Observe that we could equally have used the total angular momentum $\Psi_2$ as a parameter instead of $\delta_2$, as $\Psi_2, L_3^*$ are constant, and  \eqref{eq_totalangmomdelta2} implies that $\delta_2 = \frac{\Psi_2}{L_3^*}$. The reason that we have used the notation $\delta_2$ instead is to maintain consistency with \cite{clarke2022why}; indeed, we use several results from \cite{clarke2022why} and it is easier to compare the formulas if the notation is the same. 
\end{enumerate}
\end{remark}

Using the coordinates \eqref{eq_changeofcoordstilde}, the Keplerian Hamiltonian $F_{\mathrm{Kep}}$ takes the form
\begin{equation}\label{eq_fkepexpansion}
F_{\mathrm{Kep}} = \tilde F_{\mathrm{Kep}} + \frac{1}{\left( L_3^* \right)^3} \, \akep \, \tilde L_3 - \frac{1}{\left( L_3^* \right)^4} \frac{3}{2} \alpha_{\mathrm{Kep}} \, \tilde L_3^2 + O \left( \frac{1}{\left( L_3^* \right)^5} \right)
\end{equation}
where
\begin{equation}\label{eq_keplerconst}
	\tilde F_{\mathrm{Kep}} = - \sum_{j=1}^2 \frac{ \mu_j^3 \, M_j^2}{2 \, L_j^2} - \frac{ \mu_3^3 \, M_3^2}{2 \, \left( L_3^* \right)^2}, \quad \akep = \mu_3^3 \, M_3^2. 
\end{equation}
Since the angle $\tilde \ell_3$ evolves slower than some of the secular angles (i.e. $\gamma_1$, $\tilde \gamma_2$, $\tilde \psi_1$; see Proposition \ref{proposition_secularexpansion} below) we must consider it as a secular variable. Therefore, we include $F_{\mathrm{Kep}} - \tilde F_{\mathrm{Kep}}$ in the secular Hamiltonian $F_{\mathrm{sec}}$, which we define by 
\begin{equation}\label{eq_fsec}
F_{\mathrm{sec}} = \tilde F_{\mathrm{sec}} + \left( F_{\mathrm{Kep}} - \tilde F_{\mathrm{Kep}} \right)
\end{equation}
where $\tilde F_{\mathrm{sec}}$ is defined by \eqref{eq_fsectilde}.

In the following proposition, we identify the first appearance of each of the secular variables in the secular Hamiltonian $F_{\mathrm{sec}}$. In addition, we identify the first appearance of products of trigonometric functions of $\tilde \psi_1$, $\tilde \gamma_3$, $\tilde \ell_3$ with functions of $\gamma_1$, $\Gamma_1$, $\tilde \gamma_2$; this is of significance as these are the terms that will contribute to the Poincar\'e-Melnikov computation in Section \ref{sec_outer}. 

\begin{proposition}\label{proposition_secularexpansion}
The secular Hamiltonian $F_{\mathrm{sec}}$, defined by \eqref{eq_fsec}, can be expanded in the form
\begin{equation}\label{eq_fsecexpansion}
F_{\mathrm{sec}} = c + \sum_{i,j=0}^{\infty} \eps^i \, \mu^j \, F_{ij}
\end{equation}
where $\eps = \frac{1}{L_2}$, $\mu = \frac{L_2}{L_3}$, and where the terms in the expansion satisfy the following properties. 
\begin{enumerate}

\item \label{item_propquad12}
The first two nontrivial terms in the expansion are $F_{6,0} = \alpha_0^{12} \, H_0^{12}$, $F_{7,0} = \alpha_1^{12} \, H_1^{12}$ where $\alpha_i^{12}$ are nontrivial constants, and where the Hamiltonians $H_0^{12}$, $H_1^{12}$ are defined by \eqref{eq_H012def} and \eqref{eq_H112def} respectively, are integrable, and do not depend on the masses. The Hamiltonians $H_0^{12}$, $H_1^{12}$ are the first order terms from $F_{\mathrm{quad}}^{12}$ (see \eqref{eq_quad12unscaled}). The variables $\gamma_1, \Gamma_1, \tilde{\Gamma}_2$ appear in $H_0^{12}$, and the action $\tilde{\Psi}_1$ first appears in $H_1^{12}$. 

\item \label{item_propoct12}
The angle $\tilde{\gamma}_2$ first appears in $H_2^{12}$, which is contained in $F_{8,0}$. The Hamiltonian $H_2^{12}$ is defined by \eqref{eq_foct12firstorder}, and is the first order term in the expansion of $F_{\mathrm{oct}}^{12}$ (see \eqref{eq_oct12unscaled}), up to a multiplicative constant. 

\item \label{item_proplaction}
The action $\tilde L_3$ first appears in $F_{3,3}$, specifically in the term of order $\left( L_3^* \right)^{-3}$ in the expansion of $F_{\mathrm{Kep}} - \tilde F_{\mathrm{Kep}}$ (see \eqref{eq_fkepexpansion}, \eqref{eq_keplerconst}, and \eqref{eq_fsec}). 

\item \label{item_propquad23p1}
Each of the angles $\tilde \psi_1, \tilde \gamma_3, \tilde \ell_3$ first appears in the Hamiltonian $H_0^{23}$ (see \eqref{eq_h023h123def}) which is contained in $F_{2,6}$. 

\item
The action $\tilde \Gamma_3$ first appears in the Hamiltonian $\tilde \Gamma_3 \tilde H_3$ (see \eqref{eq_qttexp} and \eqref{eq_k0k1def}) which is contained in $F_{3,6}$. 

\item
The Hamiltonian $H_1^{23}$, contained in $F_{3,6}$, presents the first products of functions of each of the angles $\tilde \psi_1, \tilde \gamma_3, \tilde \ell_3$ with functions of $\gamma_1, \Gamma_1, \tilde \gamma_2$. 

\end{enumerate}
\end{proposition}

\begin{proof}
The proposition follows from Lemmas \ref{lemma_quad12expansion}, \ref{lemma_oct12expansion}, and \ref{lemma_quad23exp} below, upon comparing the orders of the coefficients of each Hamiltonian using the assumptions \eqref{eq_mainassumptiondeprit}. 
\end{proof}

\begin{notation}
Throughout this paper, in order to simplify notation, we use ellipsis to mean the following. Fix some sufficiently large integer $r \in \mathbb{N}$. The notation $F = \epsilon^i \, \mu^j \, G + \cdots $ means that there are $\eta_1, \, \eta_2 \in \mathbb{N}_0$, not both 0, and a positive constant $C$ such that
\[
\left\| F - \epsilon^i \, \mu^j \, G \right\|_{C^r} \leq C \, \epsilon^{i + \eta_1} \, \mu^{j + \eta_2}. 
\]
Moreover, we use the expression \emph{nontrivial constant} to mean a constant depending only on the masses and the parameters $\delta_j$ that is nonzero for all $m_0, \, m_1, \, m_2, \, m_3 >0$ satisfying \eqref{def:massconditions:2}, all $\delta_1, \, \delta_2 \in (0,1)$, and all $\delta_3 \in (-1,1)$. 
\end{notation}

\begin{remark}
Although the secular Hamiltonian depends on the mean anomaly $\tilde \ell_3$, the most natural way to compute such Hamiltonians is by using the true anomaly, denoted by $v_3$. Throughout this paper, we will write the dependence of the Hamiltonian on the variable $\tilde \ell_3$ (and indeed later on similar variables $\ell_3'$, $\hat \ell_3$ obtained via near-identity coordinate transformations); however the dependence of the Hamiltonian on this variable will be seen only implicitly through the Hamiltonian's dependence on the true anomaly $v_3$. When we need to differentiate such a Hamiltonian, say $K$, with respect to the mean anomaly $\tilde \ell_3$, we obtain $ \frac{\partial K}{\partial \tilde \ell_3} =  \frac{\partial K}{\partial v_3} \frac{\partial v_3}{\partial \tilde \ell_3}$, and we notice that, due to Kepler's second law, we have
\begin{equation}
\frac{\partial v_3}{\partial \tilde \ell_3} = \frac{\partial v_3}{\partial  \ell_3} = \left( 1 - e_3^2 \right)^{- \frac{3}{2}} \left( 1 + e_3 \cos v_3 \right)^2 = \delta_2^{-3} \left( 1 + \sqrt{1 - \delta_2^2} \, \cos v_3 \right)^2 + O \left( \frac{L_2}{L_3^*} \right),
\end{equation}
where we have expanded the eccentricity $e_3$ using \eqref{def:eccentricity} and \eqref{eq_changeofcoordstilde}. 
\end{remark}

The quadrupolar and octupolar Hamiltonians $\qot$ and $\oot$ of bodies 1 and 2 take the same form as in \cite{clarke2022why}, as described in the following two lemmas (see \cite{clarke2022why} for the proofs). 

\begin{lemma}\label{lemma_quad12expansion}
	The Hamiltonian $F_{\mathrm{quad}}^{12}$ can be written in the  variables \eqref{eq_changeofcoordstilde} as 
	\begin{equation}
		F_{\mathrm{quad}}^{12} = \tilde{c}_0^{12} + \frac{1}{L_2^6} \,  \alpha_0^{12} \,  H_0^{12} + \frac{1}{L_2^7} \, \alpha_1^{12} \, H_1^{12} + \frac{1}{L_2^8} \, \tilde{\alpha}_2 \, \tilde{H}_2 + \cdots
	\end{equation}
	where
	\begin{align}
		H_0^{12} ={}& \left( 1 -\frac{\Gamma_1^2}{L_1^2} \right) \left[ 2 - 5 \left(1 - \frac{\tilde{\Gamma}_2^2}{\Gamma_1^2} \right) \sin^2 \gamma_1 \right] + \frac{\tilde{\Gamma}_2^2}{L_1^2} \label{eq_H012def} \\
		H_1^{12} ={}& \left( 3 H_0^{12} \left( \gamma_1, \Gamma_1, \tilde{\Gamma}_2 \right) - 1 \right) \tilde{\Psi}_1 - 4 \tilde{\Gamma}_2 H_0^{12} \left( \gamma_1, \Gamma_1, \tilde{\Gamma}_2 \right) + 3 \tilde{\Gamma}_2 - \frac{\Gamma_1^2 \tilde{\Gamma}_2}{L_1^2} \label{eq_H112def} \\
		\tilde{H}_2 ={}& \left( 3 H_0^{12} \left( \gamma_1, \Gamma_1, \tilde{\Gamma}_2 \right) - 1 \right) \tilde{\Psi}_1^2 + \left( 6 - 8 H_0^{12} \left( \gamma_1, \Gamma_1, \tilde{\Gamma}_2 \right) - 2 \frac{\Gamma_1^2}{L_1^2} \right) \tilde{\Gamma}_2 \tilde{\Psi}_1 \\
		& \quad + \frac{1}{8}  \Bigg[ \sin^2 \gamma_1 \, \left( 5\,\Gamma_{1}^2 - {{5\,\Gamma_{1}^4}\over{\,L_{1}^2}}+{{210\,
				\Gamma_{1}^2 \,  \tilde{\Gamma}_{2}^2}\over{\,L_{1}^2}}-{{205\,
				\tilde{\Gamma}_{2}^4}\over{\,L_{1}^2}}+{{205\,\tilde{\Gamma}_{2}^4}\over{
				\,\Gamma_{1}^2}} \right) \label{eq_quad12termoforderepsilonsquared}\\
		& \quad + {{\Gamma_{1}^4}\over{\,L_{1}^2}}-{{66\,\Gamma_{1}^2 \,\tilde{\Gamma}_{2}^2}\over{
				\,L_{1}^2}}+{{41\,\tilde{\Gamma}_{2}^4}\over{\,L_{1}^2}}+40\,\tilde{\Gamma}_{2}^2  \Bigg]
	\end{align}
	and
	\begin{equation}
		\alpha_0^{12} = {{3\,L_{1}^4\,M_{2}^3\,\mu_{2}^6}\over{8\,M_{1}^2\,\delta_{1}^3\,
				\mu_{1}^4}}, \quad \alpha_1^{12} = -{{3\,L_{1}^4\,M_{2}^3\,\mu_{2}^6}\over{8\,M_{1}^2\,\delta_{1}^4\,
				\mu_{1}^4}}, \quad \tilde{\alpha}_2 = {{3\,L_{1}^4\,M_{2}^3\,\mu_{2}^6}\over{4\,M_{1}^2\,\delta_{1}^5\,
				\mu_{1}^4}}. 
	\end{equation}
	Moreover $F_{\mathrm{quad}}^{12}$ is integrable. 
\end{lemma}

\begin{lemma}\label{lemma_oct12expansion}
	The Hamiltonian $F_{\mathrm{oct}}^{12}$ can be written in the rescaled variables \eqref{eq_changeofcoordstilde} as
	\begin{equation}
		F_{\mathrm{oct}}^{12} = \frac{1}{L_2^8} \, \alpha_2^{12} \, H_2^{12} + \cdots 
	\end{equation}
	where
	\begin{equation}
		\begin{split} 
			H_{2}^{12} ={}& \sqrt{1 - \frac{\Gamma_1^2}{L_1^2}}  \left\{
			\begin{split}
				\cos \gamma_1 \, \cos \tilde{\gamma}_2 \left[
				\begin{split}
					\frac{\Gamma_1^2}{L_1^2} \left(5 \, \left( 1 - \frac{\tilde{\Gamma}_2^2}{\Gamma_1^2}\right) \left(6-7 \cos^2 \gamma_1 \right) - 3 \right) \\
					-35 \, \sin^2 \gamma_1 \, \left( 1 - \frac{\tilde{\Gamma}_2^2}{\Gamma_1^2}\right) + 7
				\end{split}
				\right] \\
				+ \frac{\tilde{\Gamma}_2}{\Gamma_1} \, \sin \gamma_1 \, \sin \tilde{\gamma}_2 \,  \left[
				\begin{split}
					\frac{\Gamma_1^2}{L_1^2} \left(5 \, \left( 1 - \frac{\tilde{\Gamma}_2^2}{\Gamma_1^2}\right) \left(4 - 7 \, \cos^2 \gamma_1 \right) - 3 \right) \\
					- 35 \sin^2 \gamma_1 \, \left( 1 - \frac{\tilde{\Gamma}_2^2}{\Gamma_1^2}\right) + 7
				\end{split}
				\right]
			\end{split}
			\right\} \label{eq_foct12firstorder}
		\end{split}
	\end{equation}
	and
	\begin{equation}
		\alpha_2^{12} = - \frac{15}{64} \frac{L_1^6 \mu_2^8 M_2^4}{\mu_1^6 M_1^3} \frac{\sqrt{1 - \delta_1^2}}{\delta_1^5}.
	\end{equation}
\end{lemma}

The quadrupolar Hamiltonian $\qtt$ of bodies 2 and 3 is expanded in the following lemma.

\begin{lemma}\label{lemma_quad23exp}
	We can write the Hamiltonian $\qtt$ as
	\begin{equation}\label{eq_qttexp}
	\qtt = \frac{L_2^4}{\left( L_3^* \right)^6} \,  \alpha_0^{23} \, K_0 + \frac{L_2^3}{\left( L_3^* \right)^6} \, \alpha_1^{23} \, K_1 + \cdots
	\end{equation}
	where the Hamiltonians $K_0$, $K_1$ are themselves defined via the expansions
	\begin{equation}\label{eq_k0k1def}
	\begin{dcases}
		K_0 = H_0^{23} \left( \tilde \gamma_3, \tilde \psi_1, v_3 \right) + \frac{1}{L_2} \left[ \tilde \Gamma_3 \, \tilde H_3 \left( \tilde \gamma_3, \tilde \psi_1, v_3 \right) + \tilde{H_4} \left( \tilde \gamma_3, \tilde \psi_1, v_3, \tilde \Psi_1 \right) \right] + \cdots, \\
		K_1 = H_1^{23} \left( \tilde \gamma_3, \tilde \psi_1, v_3, \Gamma_1, \tilde \Gamma_2 \right)+ \cdots,
	\end{dcases}
	\end{equation}
	and we have
	\begin{equation}\label{eq_h023h123def}
		\begin{dcases}
			H_0^{23} = \left( 1 + \sqrt{1 - \delta_2^2} \, \cos v_3 \right)^3 \Big[ A_0 \left( \tilde \gamma_3, v_3 \right) \, \cos^2 \tilde \psi_1 + B_0 \left( \tilde \gamma_3, v_3 \right) \, \cos \tilde \psi_1 \, \sin \tilde \psi_1 + C_0 \left( \tilde \gamma_3, v_3 \right) \Big], \\
			H_1^{23} =  \left( 1 + \sqrt{1 - \delta_2^2} \, \cos v_3 \right)^3 \sqrt{ \Gamma_1^2 - \tilde \Gamma_2^2} \, \left[ A_1 \left( \tilde \gamma_3, \tilde \psi_1, v_3 \right) \, \cos \tilde \gamma_2 + B_1 \left( \tilde \gamma_3, \tilde \psi_1, v_3 \right) \, \sin \tilde \gamma_2 \right]
		\end{dcases}
	\end{equation}
	where the trigonometric polynomials $A_0$, $B_0$, $C_0$ are given by 
	\begin{align}
		A_0 \left( \tilde \gamma_3, v_3 \right) ={}& 15 \left[ \left(\delta_{3}^2-{\frac{\delta_{3}^2}{\delta_1^2}}- \left(1 - \delta_{1}^2 \right)\right) \, \sin ^2\left(
		v_{3}-\tilde \gamma_3\right)
		+ \left( 1 -\delta_{1}^2 \right) \right] \label{eq_a0trigpoly}\\
		B_0 \left( \tilde \gamma_3, v_3 \right) ={}& {30\,\frac{\delta_{3}}{\delta_1}\, \left(1 - \delta_1^2 \right) \,\cos \left(v_{3}-\tilde \gamma_3\right)\,\sin \left(v_{3}-
			\tilde \gamma_3\right)} \label{eq_b0trigpoly}\\
		C_0 \left( \tilde \gamma_3, v_3 \right) ={}& \left({15\,\frac{\delta_{3}^2}{\delta_1^2}}-12\,\delta_{3}^2-3\,
		\delta_{1}^2 \right)\,\sin ^2\left(v_{3}-\tilde \gamma_3\right)+6\,\delta_{1}^2-5 \label{eq_c0trigpoly}
	\end{align}
	and the trigonometric polynomials $A_1$, $B_1$, $\tilde H_3$ are given by
	\begin{align*}
		A_1 \left( \tilde \gamma_3, \tilde \psi_1, v_3 \right) ={}& \delta_{1}\,\delta_{3}\,\cos \tilde \psi_{1}\,\sin ^2\left(v_{3}-
		\tilde \gamma_3\right)-\delta_{1}^2\,\sin \tilde \psi_{1}\,\cos \left(v_{3}-
		\tilde \gamma_3\right)\,\sin \left(v_{3}-\tilde \gamma_3\right) \\
		B_1 \left( \tilde \gamma_3, \tilde \psi_1, v_3 \right) ={}& \left( 4\,\delta_{1}\,\delta_{3}-5\,\frac{\delta_{3}}{\delta_1} \right) \,\sin \tilde \psi_{1}\,\sin ^2\left(v_{3}
		-\tilde \gamma_3\right)  +\left( 4\,\delta_{1}^2 - 5 \right)\,\cos \tilde \psi_{1}
		\,\cos \left(v_{3}-\tilde \gamma_3\right)\,\sin \left(v_{3}-\tilde \gamma_3\right) \\
		\tilde H_3 \left( \tilde \gamma_3, \tilde \psi_1, v_3 \right) ={}& \left( 1 + \sqrt{1 - \delta_2^2} \, \cos v_3 \right)^3 \bigg[   \left( 30 \, \delta_3 \left(1 - \frac{1}{\delta_1^2} \right)\,\cos ^2\tilde \psi_{1}+30\,\frac{\delta_{3}}{\delta_1^2}-24\,\delta_{3} \right)\,\sin ^2\left(
		v_{3}-\tilde \gamma_3\right) \\
		\quad & +{30\, \frac{1}{\delta_1}\, \left( 1 - \delta_1^2 \right) \, \cos \tilde \psi_{1}\,\sin \tilde \psi_{1}\,\cos \left(
			v_{3}-\tilde \gamma_3\right)\,\sin \left(v_{3}-\tilde \gamma_3\right)} \bigg].
	\end{align*}
\end{lemma}

\begin{proof}
	By definition, we have
	\begin{equation}\label{eq_quad23beforeaver}
		F_{\mathrm{quad}}^{23} = \int_{\mathbb{T}} P_2 \left( \cos \zeta_{2} \right) \frac{\| q_2 \|^2}{\| q_3 \|^3} \, d \ell_2.
	\end{equation}
	Denote by $\R_1(\theta)$, $\R_3(\theta)$ the rotation matrix by an angle $\theta$ around the $x, z$-axis respectively, and let $I_3 = \R_3 (\pi)$. Write $\bar{q}_j= \| q_j \|^{-1} \, q_j$, and $\bar{Q}_j = (\cos (\gamma_j + v_j), \sin(\gamma_j + v_j),0)$ where $v_j$ is the true anomaly corresponding to the mean anomaly $\ell_j$. By Proposition 4.1 of \cite{pinzari2009kolmogorov}, we have
	\[
\begin{split}
	\bar{q}_2& = \R_3 (\psi_3) \, \R_1 (i) \, \R_3 (\psi_2) \, \R_1 (\tilde{i}_2) \, \R_3 (\psi_1) \, I_3 \, \R_1 (i_2) \, \bar{Q}_2\\
	\bar{q}_3& = \R_3 (\psi_3) \, \R_1 (i) \, \R_3 (\psi_2) \, I_3 \, \R_1 (i_3) \, \bar{Q}_3
	\end{split}
	\]
	where
	\begin{equation}\label{eq_inclinationsdef1}
		\cos i = \frac{\Psi_3}{\Psi_2}, \qquad \cos \tilde{i}_2 = \frac{\Psi_2^2 + \Psi_1^2 - \Gamma_3^2}{2 \, \Psi_1 \, \Psi_2}, \qquad
		\cos i_2 = \frac{\Gamma_2^2 + \Psi_1^2 - \Gamma_1^2}{2 \, \Psi_1 \, \Gamma_2}, \qquad \cos i_3 = \frac{\Gamma_3^2 + \Psi_2^2 - \Psi_1^2}{2 \, \Psi_2 \, \Gamma_3}.
	\end{equation}
	Since the last 3 rotations performed in each expression $\bar{q}_2, \bar{q}_3$ are the same, they can be ignored in the computation of $\cos \zeta_2 = \bar{q}_2 \cdot \bar{q}_3$.

	First, we focus on the rotations by the angles $i_2, i_3$. Observe that, in our rescaled variables,
	\[
	\cos i_2 = 1 - L_2^{-2} \frac{\Gamma_1^2 - \tilde{\Gamma}_2^2}{2 \, \delta_1^2} + O \left( L_2^{-3} \right), \quad \sin i_2 = L_2^{-1} \frac{ \sqrt{\Gamma_1^2 - \tilde{\Gamma}_2^2}}{\delta_1} + O \left(L_2^{-2} \right),
	\]
	\[
	\cos i_3 = 1  + O \left( \frac{L_2}{L_3^*}  \right), \quad \sin i_3 =  O \left(  \frac{L_2}{L_3^*} \right).
	\]
	Therefore we can write
	\[
	\R_1 (i_2) = \Id + L_2^{-1} M_1 + O \left( L_2^{-2} \right), \quad \R_1 (i_3) = \Id + O \left(  \frac{L_2}{L_3^*} \right)
	\]
	where
	\[
	M_1 = \left(
	\begin{matrix}
		0 & 0 & 0 \\
		0 & 0 & -b_1 \\
		0 & b_1 & 0
	\end{matrix}
	\right) \quad \mathrm{with} \quad 
	b_1 = \frac{\sqrt{\Gamma_1^2 - \tilde{\Gamma}_2^2}}{\delta_1} . 
	\]
	We thus obtain the expression 
	\begin{equation}\label{eq_coszeta2exp}
		\cos \zeta_2 = \bar{q}_2 \cdot \bar{q}_3 = W_0 + \frac{1}{L_2} W_1 + O \left(  \frac{L_2}{L_3^*},  \frac{1}{L_2^2} \right)
	\end{equation}
	where
\[
\begin{split}
W_0 &= \R_1 (\tilde{i}_2) \, \R_3 (\psi_1) \, I_3 \, \bar{Q}_2 \cdot I_3 \, \bar{Q}_3,\\
		W_1& = \R_1 (\tilde{i}_2) \, \R_3 (\psi_1) \, I_3 \, M_1 \, \bar{Q}_2 \cdot I_3 \, \bar{Q}_3.
		\end{split}\]
	Recall the Legendre polynomial of degree 2 is $P_2 (x)= \frac{1}{2} \left( 3 \, x^2 - 1 \right)$. Thus
	\begin{equation}\label{eq_legendrep2exp}
		P_2 (\cos \zeta_2) = P_2 (W_0) + 3 \, \frac{1}{L_2} \, W_0 \, W_1 + O \left(  \frac{L_2}{L_3^*}  \right).
	\end{equation}
	
	In what follows, we integrate the two terms on the right-hand side of \eqref{eq_legendrep2exp} separately using the technique introduced in Appendix C of \cite{fejoz2002quasiperiodic}. We have
	\[
	\begin{dcases}
		\| q_2 \| = a_2 \, \rho_2, \quad \rho_2 = 1 - e_2 \, \cos u_2 \\
		\| q_3 \| = a_3 \, \left( 1 - e_3^2 \right) \varrho_3, \quad \frac{1}{\varrho_3} = 1 + e_3 \, \cos v_3
	\end{dcases}
	\]
	where $u_2$ is the eccentric anomaly of body 2 and $v_3$ is the true anomaly of body 3. By differentiating the Kepler equation $\ell_2 = u_2 - e_2 \, \sin u_2$ we obtain $d \ell_2 = \rho_2 \, d u_2$. Combining these formulas with \eqref{eq_quad23beforeaver} we see that
	\[
	F_{\mathrm{quad}}^{23} = \frac{a_2^2}{a_3^3 \, (1 - e_3^2)^3} \left(  K_0 + \frac{1}{L_2} \, 3 \, K_1 + O \left( \frac{L_2}{L_3^*} , \, \frac{1}{L_2^2} \right) \right)
	\]
	where
	\[
	K_0 = \varrho_3^{-3} \, \int_{\mathbb{T}} P_2 \left( W_0 \right) \, \rho_2^3 \, d u_2, \quad K_1 = \varrho_3^{-3} \, \int_{\mathbb{T}} W_0 \, W_1  \, \rho_2^3 \, d u_2. 
	\]
	Using the expressions
	\[
	\rho_2 \, \cos v_2 = \cos u_2 - e_2, \quad \rho_2 \, \sin v_2 = \sqrt{1 - e_2^2} \, \sin u_2,
	\]
	we can eliminate the angle $v_2$ from the integrands of $K_0$ and $K_1$. The result is a trigonometric polynomial which can be computed by quadrature and combined with the expansions
	\[
	e_2^2 = (1 - \delta_1^2) + \frac{1}{L_2} \, 2 \delta_1 \left( \tilde{\Gamma}_2 - \tilde{\Psi}_1 \right) + O \left( L_2^{-2} \right), \quad e_3^2 = 1 - \delta_2^2 + O \left( \frac{L_2}{L_3^*} \right),
	\]
	and
	\[
	\cos \tilde{i}_2 = \frac{\delta_3}{\delta_1}  + \frac{1}{L_2} \frac{\delta_1 \tilde{\Gamma}_3 - \delta_3 \tilde{\Psi}_1}{\delta_1^2} + O \left( \frac{L_2}{L_3^*}, \, \frac{1}{L_2^2} \right), \quad
	\sin^2 \tilde{i}_2 = \left( 1 - \frac{\delta_3^2}{\delta_1^2} \right) - \frac{1}{L_2} \frac{ 2 \delta_3}{\delta_1} \frac{ \delta_1 \tilde{\Gamma}_3 - \delta_3 \tilde{\Psi}_1}{\delta_1^2}  + O \left( \frac{L_2}{L_3^*} , \frac{1}{L_2^2} \right)
	\]
	to complete the proof of the lemma. 
\end{proof}

\section{First-order integrable dynamics}\label{section_analysisofh0}

The purpose of this section is to establish the existence of a normally hyperbolic invariant manifold for the first-order term $H_0^{12}$, defined by \eqref{eq_H012def}, in the expansion of the secular Hamiltonian.

This section is similar to the corresponding section in~\cite{clarke2022why}. We include it for the sake of completeness.

The Deprit coordinates conveniently produce in the four-body problem the same Hamiltonian $F_{\mathrm{quad}}^{12}$ (see \eqref{eq_quad12unscaled}) as the quadrupolar Hamiltonian of the three-body problem, expressed in Delaunay coordinates (see for example \cite{fejoz2016secular}; see also Section 4 of \cite{clarke2022why}). Therefore we can use the analysis and results of \cite{fejoz2016secular} (up to some errata; see Appendix F of \cite{clarke2022why}). In summary, we will show that $H_0^{12}$ has a hyperbolic periodic orbit with a homoclinic connection (in some system of coordinates); since $H_0^{12}$ is integrable (notice that it does not depend on $\tilde \gamma_2$), this homoclinic connection is a separatrix, lying in the non-transverse homoclinic intersection of the stable and unstable manifolds. 

The Hamiltonian vector field of $H_0^{12}$ in the $\gamma_1$, $\Gamma_1$ directions, obtained by differentiating \eqref{eq_H012def}, is
\begin{equation}\label{eq_H012hvf}
\begin{dcases}
\dot{\gamma}_1 = \frac{\partial H_0^{12}}{\partial \Gamma_1} = \frac{2 \, \Gamma_1}{L_1^2} \left[ 5 \, \left(1 - \frac{\tilde{\Gamma}_2^2}{\Gamma_1^2} \right) \, \sin^2 \gamma_1 - 2 \right] - 10 \, \left( 1 - \frac{\Gamma_1^2}{L_1^2} \right) \, \frac{\tilde{\Gamma}_2^2}{\Gamma_1^3} \, \sin^2 \gamma_1 \\
\dot{\Gamma}_1 = - \frac{\partial H_0^{12}}{\partial \gamma_1} = 5 \, \left(1 - \frac{\Gamma_1^2}{L_1^2} \right) \, \left( 1 - \frac{\tilde{\Gamma}_2^2}{\Gamma_1^2} \right) \, \sin 2 \gamma_1.
\end{dcases}
\end{equation}
This vector field has (among others) two equilibria whenever $\Gamma_1 = L_1$ and 
\begin{equation}\label{eq_equilibriumcondition2}
\tilde \Gamma_2 < L_1 \, \sqrt{\frac{3}{5}}. 
\end{equation}
Indeed, whenever \eqref{eq_equilibriumcondition2} holds there are two solutions $\gamma_1^{\mathrm{min}} \in (0, \pi)$ and $\gamma_1^{\mathrm{max}} = \pi - \gamma_1^{\mathrm{min}}$ to the equation
\begin{equation} \label{eq_equilibriumcondition1}
\sin^2 \gamma_1 = \frac{2}{5 \, \left(1 - \frac{\tilde{\Gamma}_2^2}{L_1^2} \right)}.
\end{equation}
Therefore $(\gamma_1, \Gamma_1) = (\gamma_1^{\mathrm{min,  max}}, L_1)$ are equilibria of the Hamiltonian vector field \eqref{eq_H012hvf}. Moreover, these equilibria are hyperbolic; this can be seen clearly below when we pass to Poincar\'e variables \eqref{eq_poincarevariables}. The Hamiltonian $H_0^{12}$ also depends on $\tilde \Gamma_2$; indeed, we have $\frac{d \tilde \gamma_2}{dt} = \frac{\partial H_0^{12}}{\partial \tilde \Gamma_2} = \frac{2 \, \tilde \Gamma_2}{L_1^2}$. Therefore, lifting the equilibria to the full phase space of $H_0^{12}$ by including the variables $\tilde \gamma_2$, $\tilde \gamma_2$, we obtain the periodic orbits
\begin{equation}\label{def:Zminmax}
Z^0_{\mathrm{min,max}} \left(t, \tilde{\gamma}_2^0 \right) = \left( \gamma_1^{\mathrm{min,max}}, L_1, \tilde{\gamma}_2^0 + \tilde{\gamma}_2^1 (t), \tilde{\Gamma}_2 \right)
\end{equation}
where $\tilde{\gamma}_2^0\in\mathbb{T}$ is the initial condition, and where
\begin{equation} \label{eq_gammafrequency}
\tilde{\gamma}_2^1 (t) = \frac{2 \, \tilde{\Gamma}_2}{L_1^2}t. 
\end{equation}
\begin{remark}\label{remark_inclinationimplieshyperbolicity}
The assumptions we have discussed in the introduction regarding inclination are seen mathematically in \eqref{eq_equilibriumcondition2} (see also Remark \ref{remark_inclinationimpliestwist} below for a refinement of this remark). Indeed, from \eqref{def:inclinationi12} and the change of coordinates \eqref{eq_changeofcoordstilde}, we see that $\cos i_{12} = \frac{\tilde \Gamma_2}{\Gamma_1} + O \left( L_2^{-1} \right)$. Therefore on the circular ellipse $\{ \Gamma_1 = L_1 \}$, the assumptions \eqref{eq_gamma2positive} and \eqref{eq_equilibriumcondition2} imply that $|\cos i_{12}| < \sqrt{\frac{3}{5}} + O \left( L_2^{-1} \right)$. This implies that $i_{12}$ is more than roughly $40^{\circ}$.
\end{remark}

Suppose \eqref{eq_equilibriumcondition2} holds, and recall moreover we have assumed in  \eqref{eq_gamma2positive} that $\tilde{\Gamma}_2>0$. Define the positive constants
\begin{equation} \label{eq_chia2def}
\chi = \sqrt{\frac{2}{3}} \frac{\tilde{\Gamma}_2}{L_1} \frac{1}{\sqrt{1 - \frac{5}{3} \frac{\tilde{\Gamma}_2^2}{L_1^2}}}, \quad A_2 = \frac{6}{L_1} \sqrt{ \frac{2}{3}} \sqrt{1 - \frac{5}{3} \frac{\tilde{\Gamma}_2^2}{L_1^2}}.
\end{equation}
The proof of the following result is identical to the proof of Lemma 3.1 in \cite{fejoz2016secular} (see also Appendix F of \cite{clarke2022why}). 

\begin{lemma}\label{lemma_separatrixformulas}
There is a heteroclinic orbit of $H_0^{12}$ joining $Z^0_{\mathrm{max}}$ and $Z^0_{\mathrm{min}}$ backward and forward in time respectively. It is defined by the equation
\begin{equation}
\left(1 - \frac{\tilde{\Gamma}_2^2}{\Gamma_1^2} \right) \, \sin^2 \gamma_1 = \frac{2}{5}
\end{equation}
where $\gamma_1 \in (\gamma_1^{\mathrm{min}}, \gamma_1^{\mathrm{max}}) \subset (0, \pi)$, and its time parametrisation is given by
\begin{equation}
Z^0 (t, \tilde{\gamma}_2^0) = \left( \gamma_1 (t), \Gamma_1 (t), \tilde{\gamma}_2 (t), \tilde{\Gamma}_2 \right)
\end{equation}
where
\begin{equation} \label{eq_cosgamma1separatrix}
\cos \gamma_1 (t) = \sqrt{ \frac{3}{5}} \, \frac{ \sinh (A_2 \, t)}{\sqrt{\chi^2 + (1+\chi^2) \, \sinh^2 (A_2 \, t)}},
\end{equation}
\begin{equation} \label{eq_Gamma1separatrix}
\Gamma_1 (t) = \tilde{\Gamma}_2 \sqrt{\frac{5}{3}} \, \frac{\sqrt{1 + \frac{3}{5} \frac{L_1^2}{\tilde{\Gamma}_2^2} \, \sinh^2 (A_2 \, t)}}{\cosh (A_2 \, t)},
\end{equation}
and
\begin{equation}\label{eq_gamma22separatrix}
\tilde{\gamma}_2 (t) = \tilde{\gamma}_2^0 + \tilde{\gamma}_2^1 (t) + \tilde{\gamma}_2^2 (t),
\qquad\text{with}\qquad
\tilde{\gamma}_2^2 (t) = \arctan \left( \chi^{-1} \tanh (A_2 \, t) \right).
\end{equation}
\end{lemma}

Even though the Hamiltonian function $H_0^{12}$ is analytic near $\{ \Gamma_1=L_1 \}$, the Deprit coordinates, as is the case with Delaunay coordinates, are singular on this hypersurface (what is the perihelion of a circle?). We therefore introduce the Poincar\'e variables
\begin{equation} \label{eq_poincarevariables}
\xi = \sqrt{2 \, (L_1 - \Gamma_1)} \, \cos \gamma_1, \quad \eta = - \sqrt{2 \, (L_1 - \Gamma_1)} \, \sin \gamma_1.
\end{equation}
This is a symplectic change of variables, in the sense that
$d \xi \wedge d \eta = d \Gamma_1 \wedge d \gamma_1$. 
In these variables, the Hamiltonian $H_0^{12}$ becomes
\begin{equation}\label{eq_H012inpoincarevariables}
\tilde{H}_0^{12} = \frac{1}{L_1} \left[ 2 \, \xi^2 - \left( 3 - 5 \frac{\tilde{\Gamma}_2^2}{L_1^2} \right) \eta^2 \right] + \frac{\tilde{\Gamma}_2^2}{L_1^2} + O_2 \left( \xi^2 + \eta^2 \right)
\end{equation}
and the entire hypersurface $\{ \Gamma_1 = L_1 \}$ becomes a single hyperbolic periodic orbit
\begin{equation}
\left(\xi, \eta, \tilde{\gamma}_2, \tilde{\Gamma}_2 \right) = \left(0,0, \tilde{\gamma}_2^0 + \tilde{\gamma}_2^1 (t), \tilde{\Gamma}_2 \right).
\end{equation}
Moreover, the heteroclinic connection established in Lemma \ref{lemma_separatrixformulas} becomes a homoclinic connection to this hyperbolic periodic orbit. 

On the hyperbolic periodic orbit and the separatrix, the energy is given by $\frac{\tilde{\Gamma}_2^2}{L_1^2}$. It follows that we have a hyperbolic periodic orbit and a homoclinic connection for each positive value of $\tilde{\Gamma}_2$ satisfying \eqref{eq_equilibriumcondition2}. In other words, the Hamiltonian $\tilde{H}_0^{12}$ has a normally hyperbolic invariant manifold given by
\begin{equation} \label{eq_nhim0}
\Lambda_0 = \left\{ \left(\xi, \eta, \tilde{\gamma}_2, \tilde{\Gamma}_2 \right) :(\xi,\eta)=(0,0),  \tilde{\gamma}_2\in\mathbb{T}, \tilde{\Gamma}_2  \in [\zeta_1, \zeta_2] \right\}
\end{equation}
where $\zeta_1, \zeta_2$ satisfy 
\begin{equation}\label{eq_nhimparametersdef}
0<\zeta_1<\zeta_2<L_1 \sqrt{\frac{3}{5}}. 
\end{equation}
Moreover the stable and unstable manifolds of $\Lambda_0$ coincide.

\section{Analysis of the inner dynamics}\label{section_inner}

We can lift the normally hyperbolic invariant manifold $\L_0$ to the full secular phase space by increasing its dimension to include the remaining secular variables $\tilde \psi_1, \tilde \Psi_1, \tilde \gamma_3, \tilde \Gamma_3, \tilde \ell_3, \tilde L_3$, in order to obtain
\[
\tilde \L_0 = \left\{ \left(\xi,\eta,\tilde \gamma_2,\tilde \Gamma_2, \tilde \psi_1, \tilde \Psi_1, \tilde \gamma_3, \tilde \Gamma_3, \tilde \ell_3, \tilde L_3 \right) : \xi = \eta = 0, \, \tilde 
\gamma_2, \tilde \psi_1, \tilde \gamma_3, \tilde \ell_3 \in \mathbb{T}, \, \tilde \Gamma_2 \in [\zeta_1, \zeta_2], \, \tilde \Psi_1, \tilde \Gamma_3, \tilde L_3 \in [-1,1] \right\}.
\]
\sloppy It is clear that this set remains a normally hyperbolic invariant manifold for $H_0^{12}$: the variables $\tilde \psi_1, \tilde \Psi_1, \tilde \gamma_3, \tilde \Gamma_3, \tilde \ell_3, \tilde L_3$ are constant with respect to $H_0^{12}$ so invariance follows; and the hyperbolicity in the normal directions $\xi, \eta$ is preserved. In addition, the variables $\left( \tilde \gamma_2,\tilde \Gamma_2, \tilde \psi_1, \tilde \Psi_1, \tilde \gamma_3, \tilde \Gamma_3, \tilde \ell_3, \tilde L_3 \right)$ define coordinates on $\tilde \L_0$. 

In a neighbourhood of the normally hyperbolic cylinder $\tilde \L_0$ the symplectic form is
\[
\Omega = d \xi \wedge d \eta + d \tilde \Gamma_2 \wedge d \tilde \gamma_2 + d \tilde \Gamma_3 \wedge d \tilde \gamma_3 + d \tilde \Psi_1  \wedge d \tilde \psi_1 + d \tilde L_3 \wedge d \tilde \ell_3,
\]
and so the restriction to $\tilde \L_0$ of $\Omega$ is 
\begin{equation}\label{eq_omega0}
	\Omega_0 = \left. \Omega \right|_{\tilde \L_0} = d \tilde \Gamma_2 \wedge d \tilde \gamma_2 + d \tilde \Gamma_3 \wedge d \tilde \gamma_3 + d \tilde \Psi_1  \wedge d \tilde \psi_1 + d \tilde L_3 \wedge d \tilde \ell_3. 
\end{equation}
The following theorem is the main result of this section. 

\begin{theorem}\label{theorem_inner}
	For any $r \geq 2$ there is $L_2^* > 0$ such that for any $L_2 \geq L_2^*$ and $L_3 \gg L_2^3$ we have the following. 
	\begin{enumerate}
		\item
		The flow of the secular Hamiltonian $F_{\mathrm{sec}}$ has a normally hyperbolic invariant manifold $\L$ that is $O \left( \frac{1}{L_2} \right)$-close to $\tilde \L_0$ in the $C^r$ topology. The restriction to $\L$ of the symplectic form $\Omega$ is closed and nondegenerate. The variables $\left(\tilde \gamma_2,\tilde \Gamma_2, \tilde \psi_1, \tilde \Psi_1, \tilde \gamma_3, \tilde \Gamma_3, \tilde \ell_3, \tilde L_3 \right)$ define coordinates on $\L$, with respect to which $\Omega |_{\L}$ is not necessarily in Darboux form. 
		\item
		Choose $k_1, k_2 \in \mathbb{N}$. There is a coordinate transformation 
		\begin{equation}
\Phi			: \left(\tilde \gamma_2,\tilde \Gamma_2, \tilde \psi_1, \tilde \Psi_1, \tilde \gamma_3, \tilde \Gamma_3, \tilde \ell_3, \tilde L_3 \right) \longmapsto \left(\hat \gamma_2, \hat \Gamma_2, \hat \psi_1, \hat \Psi_1, \hat \gamma_3, \hat \Gamma_3,\hat \ell_3, \hat L_3 \right)
		\end{equation}
		on $\L$ that is $O\left( \eps^3 \right)$-close to the identity in the $C^r$ topology such that 
		\begin{equation}
			\Omega |_{\L} = d \hat \Gamma_2 \wedge d \hat \gamma_2 + d \hat \Psi_1  \wedge d \hat \psi_1 + d \hat \Gamma_3 \wedge d \hat \gamma_3 + d \hat L_3 \wedge d \hat \ell_3,
		\end{equation}
		and the restriction to $\L$ of the secular Hamiltonian $F_{\mathrm{sec}}$ is 
		\begin{equation}
			\hat F = \hat F_0 \left( \hat \Gamma_2, \hat \Psi_1, \hat \Gamma_3, \hat L_3 ; \eps, \mu \right) + \eps^{k_1} \mu^{k_2} \, \hat F_1 \left(  \hat \gamma_2, \hat \Gamma_2, \hat \psi_1, \hat \Psi_1, \hat \gamma_3, \hat \Gamma_3,\hat \ell_3, \hat L_3 ; \eps, \mu \right)
		\end{equation}
		where $\hat F_0 = \eps^6 \, \hat c_0 \, \hat \Gamma_2^2 + \eps^7 \, \hat h_0 \left( \hat \Gamma_2, \hat \Psi_1, \hat \Gamma_3, \hat L_3 ; \eps, \mu \right)$, where $\eps = \frac{1}{L_2}$, $\mu = \frac{L_2}{L_3^*}$, and where $\hat F_j$ are uniformly bounded in the $C^r$ topology as $\eps$, $\mu \to 0$ for $j=0,1$. 
	\end{enumerate}
\end{theorem}

The rest of the section is dedicated to the proof of Theorem \ref{theorem_inner}. Fenichel theory guarantees the persistence of the normally hyperbolic manifold $\tilde \L_0$ for $F_{\mathrm{sec}}$ \cite{fenichel1971persistence,fenichel1974asymptotic,fenichel1977asymptotic}. In particular, the existence is guaranteed of a function $\rho : \tilde \L_0 \to \mathbb{R}^2$ that is $O\left(L_2^{-1}\right)$ small in the $C^r$ topology such that the set
\begin{equation}\label{eq_secnhim}
	\L = \textrm{graph} (\rho) = \left\{ (\rho(x),x) : x \in \tilde \L_0 \right\}
\end{equation}
is a normally hyperbolic invariant manifold for $\Fsec$. The following lemma provides us with information regarding the order at which each secular variable appears in the Taylor expansion of $\rho$. 

\begin{lemma}\label{lemma_graphexpansion}
	The function $\rho$ admits a Taylor expansion of the form
	\[
	\rho = \frac{1}{L_2} \, \rho_0 + \frac{1}{L_2^2} \, \rho_1 + \frac{L_2^{10}}{\left( L_3^* \right)^6} \, \rho_2 + \frac{L_2^{9}}{\left( L_3^* \right)^6} \, \rho_3 + \frac{L_2^{10}}{\left( L_3^* \right)^7} \, \rho_4
	\]
	with
	\[
	\begin{dcases}
		\rho_0 ={}& \rho_0 \left( \tilde \Gamma_2 \right) \\
		\rho_1 ={}& \rho_1 \left( \tilde \gamma_2, \tilde \Gamma_2, \tilde \Psi_1 \right) \\
		\rho_2 ={}& \rho_2 \left( \tilde \gamma_2, \tilde \Gamma_2, \tilde \psi_1, \tilde \Psi_1, \tilde \gamma_3, \tilde \ell_3 \right) \\
		\rho_3 ={}& \rho_2 \left( \tilde \gamma_2, \tilde \Gamma_2, \tilde \psi_1, \tilde \Psi_1, \tilde \gamma_3, \tilde \Gamma_3, \tilde \ell_3 \right) \\
		\rho_4 ={}& \rho_2 \left( \tilde \gamma_2, \tilde \Gamma_2, \tilde \psi_1, \tilde \Psi_1, \tilde \gamma_3, \tilde \Gamma_3, \tilde \ell_3, \tilde L_3 \right) \\
	\end{dcases}
	\]
	where each $\rho_j$ is uniformly bounded in the $C^r$ topology as $L_2, L_3 \to \infty$ (with $L_3\gg L_2^3$). 
\end{lemma}

\begin{proof}
	The proof is analogous to the proof of Lemma 21 in \cite{clarke2022why}, and we do not repeat it here. 
\end{proof}

\begin{lemma}\label{lemma_straightsymp}
	There is a coordinate transformation 
	\begin{equation}\label{eq_coordtransfprime}
\Phi^{(1)}:		\left(\tilde \gamma_2, \tilde \Gamma_2, \tilde \gamma_3, \tilde \Gamma_3, \tilde \psi_1, \tilde \Psi_1, \tilde \ell_3, \tilde L_3 \right) \longmapsto \left(\gamma_2', \Gamma_2', \gamma_3', \Gamma_3', \psi_1' , \Psi_1', \ell_3', L_3' \right)
	\end{equation}
	on $\L$ that is $O \left(\frac{1}{L_2^3} \right)$ close to the identity satisfying 
	\[
	\begin{dcases}
		\Gamma_2' ={}& \tilde \Gamma_2 + \frac{1}{L_2^3} \, P_0' \left(\tilde \gamma_2, \tilde \Gamma_2, \tilde \gamma_3, \tilde \Gamma_3, \tilde \psi_1, \tilde \Psi_1, \tilde \ell_3, \tilde L_3 \right) \\
		\Psi_1' ={}& \tilde \Psi_1 + \frac{L_2^9}{\left(L_3^*\right)^6} \, P_1' \left(\tilde \gamma_2, \tilde \Gamma_2, \tilde \gamma_3, \tilde \Gamma_3, \tilde \psi_1, \tilde \Psi_1, \tilde \ell_3, \tilde L_3 \right)\\
		\Gamma_3' ={}& \tilde \Gamma_3  + \frac{L_2^9}{\left(L_3^*\right)^6} \, P_2' \left(\tilde \gamma_2, \tilde \Gamma_2, \tilde \gamma_3, \tilde \Gamma_3, \tilde \psi_1, \tilde \Psi_1, \tilde \ell_3, \tilde L_3 \right)\\
		L_3' ={}& \tilde L_3 + \frac{L_2^9}{\left(L_3^*\right)^6} \, P_3' \left(\tilde \gamma_2, \tilde \Gamma_2, \tilde \gamma_3, \tilde \Gamma_3, \tilde \psi_1, \tilde \Psi_1, \tilde \ell_3, \tilde L_3 \right) \\
	\end{dcases}
	\]
	such that
	\[
	\left. \Omega \right|_{\L} = d \Gamma_2' \wedge d \gamma_2' + d \Gamma_3' \wedge d \gamma_3' + d \Psi_1'  \wedge d \psi_1' + d L_3' \wedge d \ell_3'. 
	\]
\end{lemma}

\begin{proof}
	Denote by $U \subset \mathbb{R}^2$ the domain of the Poincar\'e variables $(\xi,\eta)$, and by $V \subset \mathbb{R}^4$ the domain of the actions $\left( \tilde \Gamma_2, \tilde \Psi_1, \tilde \Gamma_3, \tilde L_3 \right)$. Define the inclusion map $P : \mathbb{T}^4 \times V \to U \times \left( \mathbb{T}^4 \times V \right)$ by $P(x) = (\rho(x), x)$ where $\rho$ is the function that parametrises the normally hyperbolic cylinder $\L$ via \eqref{eq_secnhim}. Then $\Omega_1 = P^* \Omega$ is the pullback to $\L$ of $\Omega$ in the coordinates \eqref{eq_changeofcoordstilde}. 
	
	The Liouville 1-form $\lambda$ is given by $\lambda = \xi \, d \eta +  \tilde \Gamma_2 \, d \tilde \gamma_2 +  \tilde \Gamma_3 \, d \tilde \gamma_3 +  \tilde \Psi_1  \, d \tilde \psi_1 +  \tilde L_3 \, d \tilde \ell_3$, and we have $\Omega = d \lambda$. Define $\lambda_1 = P^* \lambda$. Then $\Omega_1$ is exact, because $d \lambda_1 = d \left( P^* \lambda \right) = P^* d \lambda = P^* \Omega = \Omega_1$. 
	
	Denote by $\rho_{\xi}, \rho_{\eta}$ the $\xi, \eta$ components of $\rho$ respectively, and by $\rho_{j, \xi}, \rho_{j, \eta}$ the $\xi, \eta$ components of $\rho_j$ respectively for each $j=0,1,2,3,4$. Using \eqref{eq_omega0} we see that
	\[
	\Omega_1 = P^* \Omega = d \rho_{\xi} \wedge d \rho_{\eta} + \Omega_0 =  \Omega_0 + R_1 + R_2
	\]
	with
	\begin{align}
		R_1 ={}& \frac{1}{L_2^2} \left( d \rho_{0, \xi} + \frac{1}{L_2} \, d \rho_{1, \xi} \right) \wedge \left( d \rho_{0, \eta} + \frac{1}{L_2} \, d \rho_{1, \eta} \right) 
		= \frac{1}{L_2^3} \, \left[d \rho_{0, \xi} \wedge d \rho_{1, \eta} + d \rho_{1, \xi} \wedge d \rho_{0, \eta} + \frac{1}{L_2} \, d \rho_{1, \xi} \wedge d \rho_{1, \eta} \right], \\
		R_2 ={}& d \rho_{\xi} \wedge d \rho_{\eta} - R_1
	\end{align}
	where we have used the fact that $\rho_0$ depends only on $\tilde \Gamma_2$. Then $R_1$ is of order $\frac{1}{L_2^3} $, and depends only on $\tilde \gamma_2, \tilde \Gamma_2, \tilde \Psi_1$, whereas $R_2$ is of order $\frac{L_2^9}{\left( L_3^* \right)^6}$ (i.e. the order of $\rho_2$ times the order of $\rho_1$) and depends on all of the secular variables. In what follows, we construct the coordinate transformation \eqref{eq_coordtransfprime} in two steps: the first one eliminates $R_1$ from $\Omega_1$, and the second eliminates $R_2$. 
	
	Suppose there is a coordinate transformation
	\begin{equation}\label{eq_coordtransfmoserstep1}
		h_0 : \left(\tilde \gamma_2, \tilde \Gamma_2, \tilde \psi_1, \tilde \Psi_1, \tilde \gamma_3, \tilde \Gamma_3, \tilde \ell_3, \tilde L_3 \right) \longmapsto  \left(\gamma_2^*, \Gamma_2^*, \psi_1^*, \tilde \Psi_1, \tilde \gamma_3, \tilde \Gamma_3, \tilde \ell_3, \tilde L_3 \right)
	\end{equation}
	that changes only the $\tilde \gamma_2, \tilde \Gamma_2, \tilde \psi_1$ variables, such that
	\begin{equation}\label{eq_moserstep1}
		h_0^* \, \Omega' = \Omega_0 \quad \textrm{where} \quad \Omega' = \Omega_0 + R_1. 
	\end{equation}
	We use Moser's trick from his proof of Darboux's theorem: suppose $h_0 = \phi_{\hat \eps}$ where $\phi_t$ is the time-$t$ map of some nonautonomous vector field $X_t$ and where $\hat \eps = L_2^{-3}$. Upon differentiating \eqref{eq_moserstep1} with respect to $\hat \eps$ and using Cartan's magic formula we get
	\[
	0 = \frac{d}{d \hat \eps} \left[ \phi_{\hat \eps}^* \Omega' \right] = \phi_{\hat \eps}^* \left[ \frac{d}{d \hat \eps} \Omega' + \LPM_{X_{\hat \eps}} \Omega' \right] = \phi_{\hat \eps}^* \left[ \frac{d}{d \hat \eps} \Omega' + i_{X_{\hat \eps}} d \Omega' + d i_{X_{\hat \eps}}  \Omega' \right]
	\]
	where $\LPM_{X_{\hat \eps}}$ is the Lie derivative with respect to $X_{\hat \eps}$, and $i_{X_{\hat \eps}}$ is the contraction operator of $X_{\hat \eps}$. By the same argument as for $\Omega_1$ above, $\Omega' = d \lambda'$ is exact, and so, since $d \Omega' = 0$, we obtain
	\[
	d i_{X_{\hat \eps}} \Omega' = - \frac{d}{d \hat \eps} \Omega' = - \frac{d}{d \hat \eps} d \lambda' = -d \left( \frac{d}{d \hat \eps} \lambda' \right). 
	\]
	Observe that this equation is satisfied by vector fields $X_t$ for which
	\[
	i_{X_{\hat \eps}} \Omega' = - \frac{d}{d \hat \eps} \lambda'. 
	\]
	By inverting $\Omega'$ this can be solved explicitly for $X_t$. Its flow exists at least for a time $\hat \eps$, and its time-$\hat \eps$ map gives the required map $h_0$ as in \eqref{eq_coordtransfmoserstep1}. Indeed, the argument in Appendix \ref{appendix_moser} implies that this coordinate transformation does not affect $\tilde \Psi_1$, because $\rho_0, \, \rho_1$ do not depend on $\tilde \psi_1$. 
	
	In the new coordinates, the symplectic form $\Omega_1$ becomes $\hat \Omega_1 = \Omega_0 + \hat R_2$ where $\hat R_2 = h_0^* \, R_2$. Since $R_2$ is of order $\frac{L_2^9}{\left( L_3^* \right)^6}$, we may repeat the above procedure with $\hat \eps = \frac{L_2^9}{\left( L_3^* \right)^6}$ to complete the proof of the lemma. 
\end{proof}

\begin{lemma}\label{lemma_inneraveraging}
	Choose $k_1, k_2 \in \mathbb{N}$. There is a symplectic coordinate transformation
	\[
	\Phi^{(2)} : \left(\gamma_2', \Gamma_2', \gamma_3', \Gamma_3', \psi_1' , \Psi_1', \ell_3', L_3' \right) \longmapsto \left(\hat \gamma_2, \hat \Gamma_2, \hat \gamma_3, \hat \Gamma_3, \hat \psi_1, \hat \Psi_1, \hat \ell_3, \hat L_3 \right)
	\]
	on $\L$ that is $O \left( \eps^3 \right)$ close to the identity, such that the restriction $\hat F = \left. F_{\mathrm{sec}} \right|_{\L}$ of the secular Hamiltonian to $\L$ becomes 
	\[
	\hat F = \hat F_0 \left(\hat \Gamma_2, \hat \Gamma_3, \hat \Psi_1, \hat L_3; \eps, \mu  \right) + \eps^{k_1} \, \mu^{k_2} \, F_1 \left(\hat \gamma_2, \hat \Gamma_2, \hat \gamma_3, \hat \Gamma_3, \hat \psi_1, \hat \Psi_1, \hat \ell_3, \hat L_3; \eps, \mu  \right)
	\]
	where $\hat F_0 = \eps^6 c_0 \hat \Gamma_2^2 + \eps^7 \hat h_0 \left(\hat \Gamma_2, \hat \Gamma_3, \hat \Psi_1, \hat L_3; \eps, \mu  \right)$, where $\eps = \frac{1}{L_2}$, $\mu = \frac{L_2}{L_3^*}$, and where the $C^2$ norm of  $\hat F$ is uniformly bounded in $\eps, \mu$ for $j=0,1$. Moreover we can approximate the transformations of the actions at first order by
	\begin{align}
		\hat \Gamma_2 ={}& \Gamma_2' + O \left( \frac{1}{L_2^3} \right) \label{eq_gamma2averaging}\\
		\hat{\Gamma}_3 ={}& \Gamma_3'  - \frac{L_2^4}{\left( L_3^* \right)^3} \frac{\alpha_0^{23}}{\akep} \frac{\beta_0 \, \delta_2^3}{6} \left[ \sqrt{1 - \delta_2^2} \, \left(\cos \left(3 v_3' - 2 \gamma_3' \right) + 3 \cos \left(v_3' - 2 \gamma_3' \right) \right) + 3 \cos \left(2 v_3' - 2 \gamma_3' \right) \right]  + \cdots \label{eq_gamma3averaging}\\
		\hat \Psi_1 ={}& \Psi_1' - \frac{L_2^{11}}{\left(L_3^*\right)^6} \frac{\alpha_0^{23}}{\alpha_1^{12}}\frac{\left(1 + \sqrt{1 - \delta_2^2} \, \cos v_3' \right)^3}{2 \left(3 \frac{\left( \Gamma_2' \right)^2}{L_1^2} - 1 \right)} \left[ A_0 \left( \gamma_3', v_3' \right) \, \cos 2 \psi_1' + B_0 \left( \gamma_3', v_3' \right) \, \sin 2 \psi_1' \right] + \cdots \label{eq_psi1averaging}\\
		\hat L_3 ={}& L_3' - \frac{L_2^4}{\left( L_3^* \right)^3} \frac{\alpha_0^{23}}{\akep} \left( \left(1 + \sqrt{1 - \delta_2^2} \, \cos v_3' \right)^3  \left[ \beta_0 \, \sin^2 \left( v_3' - \gamma_3' \right) + \beta_1 \right] - \delta_2^3 \, \left( \frac{\beta_0}{2} + \beta_1 \right) \right) + \cdots \label{eq_l3averaging}
	\end{align}
	where the trigonometric polynomials $A_0$, $B_0$ are defined by \eqref{eq_a0trigpoly}, \eqref{eq_b0trigpoly}, the constants $\beta_j$ are defined by
	\begin{equation}\label{eq_averagingbetaconstants}
		\beta_0 = \frac{1}{2} \left( 9 \, \delta_1^2 - 9 \, \delta_3^2 + 15 \, \frac{\delta_3^2}{\delta_1^2} - 15 \right), \quad \beta_1 = \frac{1}{2} \left( 5 - 3 \, \delta_1^2 \right), 
	\end{equation}
	and where $v_3'$ is the true anomaly corresponding to the mean anomaly $\ell_3'$. 
\end{lemma}

\begin{proof}
Since the frequencies of the angles all have different size the averaging procedure is rather standar, since there are no small divisors.	We explicitly compute the first order of the symplectic coordinate transformation $\Phi^{(2)}$ in the actions $\Gamma_3', \Psi_1', L_3'$ and approximate the order of the transformation in $\Gamma_2'$. In order to do this we search for a Hamiltonian $J$ such that $\phi^1 = \Phi^{(2)}$ where $\phi^t$ is the Hamiltonian flow of $J$. In this case $\Phi^{(2)}$ is necessarily symplectic. First of all, notice that since $\qot$ does not depend on $\gamma_2$, and since the restriction to $\L$ of $\oot$ vanishes at first order, the first order term that could contain $\gamma_2'$ is the term of order $\eps^9$. Since the first order term in the expansion of $\qot$ is $\eps^6 \left(\Gamma_2 \right)'^2$ times a nontrivial constant, we can easily find a Hamiltonian $J_0$ of order $\eps^3$ whose time-1 map gives a symplectic coordinate transformation that averages $\gamma_2'$ from the term of order $\eps^9$, hence \eqref{eq_gamma2averaging}. 
	
	Recall the first order terms containing $\tilde \Psi_1$, $\tilde \psi_1$, respectively, are $H_1^{12}$, $H_0^{23}$. Moreover the coefficients of $H_1^{12}$, $H_0^{23}$ respectively in the expansion of the secular Hamiltonian are $\eps^7 \alpha_1^{12}$, $\eps^2 \mu^6 \alpha_0^{23}$. Notice that, up to higher-order terms (and a multiplicative constant equal to $L_1^2$ in the case of $h_0^{12}$), the restrictions $H_0^{12}$, $H_1^{12}$, respectively, are equal to $h_0^{12}$, $h_1^{12}$ where
	\[
	h_0^{12} = \left( \Gamma_2' \right)^2, \quad  h_1^{12} = H_1^{12}|_{\L} = \left(3 \frac{\left( \Gamma_2' \right)^2}{L_1^2} - 1 \right) \Psi_1' + 2 \Gamma_2'.
	\]
	Moreover the restriction to $\L$ of $H_0^{23}$ is equal to $H_0^{23}$, up to higher order terms. Let $c_0 = L_1^{-2} \alpha_2^{23}$ and define the Hamiltonian $\hat H_0 = \eps^6 c_0 h_0^{12} + \eps^7 \alpha_1^{12} h_1^{12} +\eps^2 \mu^6 \alpha_0^{23} H_0^{23}$, a truncated version of $F_{\mathrm{sec}} |_{\L}$. We now search for a Hamiltonian $J_1= \eps^{-5} \mu^6 \frac{\alpha_0^{23}}{\alpha_1^{12}} \hat J_1$ such that, denoting by $\phi^t_1$ the Hamiltonian flow of $J_1$, the coordinate transformation $\phi^1_1$ eliminates $\psi_1'$ from $H_0^{23}$. Assuming $J_1$ does not depend on $\gamma_2'$, we have
	\begin{align}
		\hat H_0 \circ \phi_1^{-1} ={}& \eps^6 c_0 h_0^{12} \circ \phi_1^{-1} + \eps^7 \alpha_1^{12} h_1^{12} \circ \phi_1^{-1} +\eps^2 \mu^6 \alpha_0^{23} H_0^{23} \circ \phi_1^{-1} \\
		={}& \eps^6 c_0 h_0^{12} + \eps^7 \alpha_1^{12} \left( h_1^{12} - \{ h_1^{12}, J_1 \} \right)+\eps^2 \mu^6 \alpha_0^{23} H_0^{23} + \cdots \\
		={}&\eps^6 c_0 h_0^{12} + \eps^7 \alpha_1^{12} h_1^{12} +\eps^2 \mu^6 \alpha_0^{23} \left( H_0^{23} - \left\{ h_1^{12}, \hat J_1 \right\} \right) + \cdots \label{eq_hamafteravetransf1}
	\end{align}
	Integrating \eqref{eq_h023h123def}, we see that
	\begin{equation}\label{eq_h023psi1average}
		\left\langle H_0^{23} \right\rangle_{\psi_1'} =   \frac{1}{2 \pi} \int_{\mathbb{T}} H_0^{23} \, d \psi_1' = \left( 1 + \sqrt{1 - \delta_2^2} \, \cos v_3' \right)^3 \left[ \frac{1}{2} A_0 (\gamma_3', v_3') + C_0 (\gamma_3', v_3') \right]. 
	\end{equation}
	We would like $\hat J_1$ to satisfy 
	\begin{equation}\label{eq_psi1averagediffeqn}
		\left\{ h_1^{12}, \hat J_1 \right\} = H_0^{23} - \left\langle H_0^{23} \right\rangle_{\psi_1'} = \frac{1}{2} \left( 1 + \sqrt{1 - \delta_2^2} \, \cos v_3' \right)^3 \left[  A_0 (\gamma_3', v_3') \cos 2 \psi_1' + B_0 (\gamma_3', v_3') \sin 2 \psi_1' \right]. 
	\end{equation}
	Choose 
	\[
	\hat J_1 = \frac{1}{4} \left( 3 \frac{\left( \Gamma_2' \right)^2}{L_1^2} - 1 \right)^{-1} \left( 1 + \sqrt{1 - \delta_2^2} \, \cos v_3' \right)^3 \left[  A_0 (\gamma_3', v_3') \sin 2 \psi_1' - B_0 (\gamma_3', v_3') \cos 2 \psi_1' \right].
	\]
	Since $\hat J_1$ does not depend on $\gamma_2'$, neither does $J_1$, and so \eqref{eq_hamafteravetransf1} is satisfied; moreover, it can easily be checked that this choice of $\hat J_1$ solves \eqref{eq_psi1averagediffeqn}. Hamilton's equations of motion imply that
	\[
	\hat \Psi_1 = \Psi_1' - \frac{\partial J_1}{\partial \psi_1'} + \cdots = \Psi_1' - \frac{L_2^{11}}{\left(L_3^*\right)^6} \frac{\alpha_0^{23}}{\alpha_1^{12}} \frac{\partial \hat J_1}{\partial \psi_1'} + \cdots
	\]
	Upon differentiating $\hat J_1$ with respect to $\psi_1'$, we thus obtain \eqref{eq_psi1averaging}. 
	
	The next step of the proof comprises the construction of a Hamiltonian $J_2$ such that the time-1 map of its flow $\phi^t_2$ averages the angle $\ell_3'$ from the secular Hamiltonian at first order by adjusting $L_3'$. The first appearance of $L_3'$ in the secular Hamiltonian is the term $\left( L_3^* \right)^{-3} \alpha_{\mathrm{Kep}} \tilde L_3$ where the constant $\alpha_{\mathrm{Kep}}$ is defined by \eqref{eq_keplerconst}, whereas $\ell_3'$ first appears via $v_3'$ in the term $\eps^2 \mu^6 \alpha_0^{23} \left\langle H_0^{23} \right\rangle_{\psi_1'}$, defined by \eqref{eq_h023psi1average}. The average of $\left\langle H_0^{23} \right\rangle_{\psi_1'}$ with respect to $\ell_3'$ is
	\begin{align}
		\left\langle H_0^{23} \right\rangle_{\psi_1', \ell_3'} ={}& \frac{1}{2 \pi} \int_{\mathbb{T}} \left\langle H_0^{23} \right\rangle_{\psi_1'} \, d \ell_3' = \frac{\delta_2^3}{2 \pi} \int_{\mathbb{T}} \left(1 + \sqrt{1 - \delta_2^2} \, \cos v_3' \right) \, \left[ \beta_0 \, \sin^2 (v_3' - \gamma_3') + \beta_1 \right] \, d v_3' + \cdots \\
		={}& \delta_2^{3} \left( \frac{\beta_0}{2} + \beta_1 \right) + \cdots
	\end{align}
	where the constants $\beta_j$ are defined by \eqref{eq_averagingbetaconstants}, and where we have used the formula
	\begin{equation}\label{eq_keplersecondlaw}
		d \ell_3' = d \ell_3 + \cdots  = \left( 1 + e_3 \, \cos v_3 \right)^{-2} \left( \frac{\Gamma_3}{L_3} \right)^3 \, d v_3 = \delta_2^3 \left(1 + \sqrt{1 - \delta_2^2} \, \cos v_3' \right)^{-2} \, d v_3' + O \left( \frac{1}{L_3^*} \right),
	\end{equation}
	which comes from Kepler's second law. Therefore, writing $J_2 = \left( L_3^* \right)^3 \eps^2 \mu^6 \frac{\alpha_0^{23}}{\akep} \hat J_2$, and assuming that $\hat J_2$ does not depend on $\gamma_2'$, by similar reasoning to \eqref{eq_hamafteravetransf1}, we search for a function $\hat J_2$ satisfying
	\begin{align}
		\frac{\partial \hat J_2}{\partial \ell_3'} ={}& \left\{ L_3', \hat J_2 \right\} = \left\langle H_0^{23} \right\rangle_{\psi_1'} - \left\langle H_0^{23} \right\rangle_{\psi_1', \ell_3'} + \cdots \\
		={}& \left(1 + \sqrt{1 - \delta_2^2} \, \cos v_3' \right)^3 \left[ \beta_0 \, \sin^2 \left( v_3' - \gamma_3' \right) + \beta_1 \right] - \delta_2^3 \, \left( \frac{\beta_0}{2} + \beta_1 \right) + \cdots \label{eq_j2l3average}
	\end{align}
	We could find $\hat J_2$ explicitly by integration; however, since we only need to know
	\[
	\hat L_3 = L_3' - \frac{\partial J_2}{\partial \ell_3'} + \cdots = L_3' - \frac{L_2^4}{\left( L_3^* \right)^3} \frac{\alpha_0^{23}}{\akep} \frac{\partial \hat J_2}{\partial \ell_3'} + \cdots
	\]
	we have already established the approximation \eqref{eq_l3averaging}. 
	
	Finally, in order to average $\gamma_3'$ by adjusting $\Gamma_3'$, we only have to notice that, after averaging the angles $\gamma_2'$, $\psi_1'$, $v_3'$ from $F_{\mathrm{sec}} |_{\L}$, we have the same Hamiltonian as in \cite{clarke2022why} after averaging $\gamma_2'$, $\psi_1'$ from the restriction to $\L$ of what is called the secular Hamiltonian in that paper. Therefore the Hamiltonian designed to average $\gamma_3'$ is of order $\frac{L_2^{13}}{\left(L_3\right)^8}$ (see Lemma 24 of \cite{clarke2022why}). However this does not imply that the first order of the coordinate transformation comes from this term; indeed, the first order of the coordinate transformation could come from the Hamiltonians $J_1$, or $J_2$, as they both depend on $\gamma_3'$, or from the first term averaging $\gamma_2'$ that depends also on $\gamma_3'$. The latter is of order $\frac{L_2^9}{\left( L_3^* \right)^6}$, because the first term in the Hamiltonian depending on both $\gamma_2'$ and $\gamma_3'$ is $H_1^{23}$ which is of order $\frac{L_2^3}{\left( L_3^* \right)^6}$, and which we then average using the term of order $\frac{1}{L_2^6}$. Now, comparing the orders of these two terms (i.e. $\frac{L_2^{13}}{\left(L_3\right)^8}$ and $\frac{L_2^9}{\left( L_3^* \right)^6}$) with the orders $\frac{L_2^{11}}{\left(L_3^*\right)^6}$ and $\frac{L_2^4}{\left( L_3^* \right)^3}$ of $J_1$ and $J_2$ respectively, and using the assumption 
	\eqref{eq_mainassumption}, we see in fact that the first order of the coordinate transformation in the variable $\Gamma_3'$ comes from $J_2$. Therefore 
	\[
	\hat \Gamma_3 = \Gamma_3' - \frac{\partial J_2}{\partial \gamma_3'} + \cdots = \Gamma_3' - \frac{L_2^4}{\left( L_3^* \right)^3} \frac{\alpha_0^{23}}{\akep} \frac{\partial \hat J_2}{\partial \gamma_3'} + \cdots
	\]
	and we have
	\begin{align}
		\frac{\partial \hat J_2}{\partial \gamma_3'} ={}& \frac{\partial }{\partial \gamma_3'} \int \left( \left\langle H_0^{23} \right\rangle_{\psi_1'} - \left\langle H_0^{23} \right\rangle_{\psi_1', \ell_3'} \right) \, d \ell_3' = - \beta_0 \int \left( 1 + \sqrt{1 - \delta_2^2} \, \cos v_3' \right)^3  \, \sin (2 (v_3' - g_3' )) \, d \ell_3' + \cdots \\
		={}& - \beta_0 \, \delta_2^3 \int \left( 1 + \sqrt{1 - \delta_2^2} \, \cos v_3' \right) \, \sin (2 (v_3' - g_3' )) \, d v_3' + \cdots \\
		={}& \frac{\beta_0 \, \delta_2^3}{6} \left[ \sqrt{1 - \delta_2^2} \, \left(\cos \left(3 v_3' - 2 \gamma_3' \right) + 3 \cos \left(v_3' - 2 \gamma_3' \right) \right) + 3 \cos \left(2 v_3' - 2 \gamma_3' \right) \right] + \cdots
	\end{align}
	where we have used \eqref{eq_keplersecondlaw} and \eqref{eq_j2l3average}, from which we obtain \eqref{eq_gamma3averaging}. 
\end{proof}

\section{Computation of the scattering map} \label{sec_outer}
Now that we have analysed the inner dynamics on the normally hyperbolic invariant cylinder $\Lambda$, we must analyse its invariant manifolds, their transverse intersections and the corresponding dynamics. The next theorem, which describes the so-called outer dynamics associated to $\Lambda$ and its invariant manifolds,  is the main result of this section.

\begin{theorem}\label{theorem_outer}
	The stable and unstable invariant manifolds of the normally hyperbolic invariant manifold $\L$ intersect transversely along two homoclinic channels, giving rise to two scattering maps $S_{\pm} : \L \to \L'$ such that 
	\[
	S_{\pm} : \left(\hat \gamma_2, \hat \Gamma_2, \hat \gamma_3, \hat \Gamma_3, \hat \psi_1, \hat \Psi_1, \hat \ell_3, \hat L_3 \right) \longmapsto \left(\hat \gamma_2^*, \hat \Gamma_2^*, \hat \gamma_3^*, \hat \Gamma_3^*, \hat \psi_1^*, \hat \Psi_1^*, \hat \ell_3^*, \hat L_3^* \right)
	\]
	with
	\[
	\begin{dcases}
		\hat \Psi_1^* =& \hat \Psi_1 + \frac{L_2^9}{\left( L_3^* \right)^6} \S_1^{\pm} \left( \hat \psi_1, \hat \gamma_3, \hat \ell_3, \hat \Gamma_2 \right) + \cdots \\ 
		\hat \Gamma_3^* =& \hat \Gamma_3 + \frac{L_2^9}{\left( L_3^* \right)^6} \S_2^{\pm} \left( \hat \psi_1, \hat \gamma_3, \hat \ell_3, \hat \Gamma_2 \right) + \cdots \\
		\hat L_3^* =& \hat L_3+ \frac{L_2^9}{\left( L_3^* \right)^6} \S_3^{\pm} \left( \hat \psi_1, \hat \gamma_3, \hat \ell_3, \hat \Gamma_2 \right) + \cdots
	\end{dcases}
	\]
	where
	\begin{align}
		\begin{split}
			\S_1^{\pm} \left( \hat \psi_1, \hat \gamma_3, \hat \ell_3, \hat \Gamma_2 \right) =& \mp  \alpha_1^{23} \, \left( 1 + \sqrt{1 - \delta_2^2} \, \cos \hat v_3 \right)^3 \kappa \left( \frac{\pi \, \hat{\Gamma}_2}{A_2 \, L_1^2} \right) \,  \frac{\partial B_1}{\partial \hat \psi_1} \left( \hat \psi_1, \hat \gamma_3, \hat \ell_3 \right) + \frac{\alpha_0^{23} \, \alpha_2^{12}}{\alpha_1^{12}}\frac{L_1}{6} \sqrt{\frac{3}{2}} \hat \Gamma_2 \\
			& \quad \times   \sqrt{1 - \frac{5}{3} \frac{\hat \Gamma_2^2}{L_1^2}} \,  \frac{\left(1 + \sqrt{1 - \delta_2^2} \, \cos \hat v_3 \right)^3}{\left(3 \frac{ \hat \Gamma_2^2}{L_1^2} - 1 \right)}  \left[ A_0 \left(\hat \gamma_3, \hat \ell_3 \right) \, \sin 2 \hat \psi_1 - B_0 \left( \hat \gamma_3, \hat \ell_3 \right) \, \cos 2 \hat \psi_1 \right] 
		\end{split} \label{eq_scatteringhat1} \\
		\S_2^{\pm} \left( \hat \psi_1, \hat \gamma_3, \hat \ell_3, \hat \Gamma_2 \right) =& \mp  \alpha_1^{23} \, \left( 1 + \sqrt{1 - \delta_2^2} \, \cos \hat v_3 \right)^3 \kappa \left( \frac{\pi \, \hat{\Gamma}_2}{A_2 \, L_1^2} \right) \,  \frac{\partial B_1}{\partial \hat \gamma_3} \left( \hat \psi_1, \hat \gamma_3, \hat \ell_3 \right) \label{eq_scatteringhat2}\\
		\S_3^{\pm} \left( \hat \psi_1, \hat \gamma_3, \hat \ell_3, \hat \Gamma_2 \right) =& \mp   \alpha_1^{23} \, \delta_2^{-3} \, \left( 1 + \sqrt{1 - \delta_2^2} \, \cos \hat v_3 \right)^4 \, \kappa \left( \frac{\pi \, \hat{\Gamma}_2}{A_2 \, L_1^2} \right) \bigg[ -3 \sqrt{1 - \delta_2^2} \, \sin \hat v_3  \\
		& \quad \times   B_1 \left( \hat \psi_1, \hat \gamma_3, \hat \ell_3 \right) + \left( 1 + \sqrt{1 - \delta_2^2} \, \cos \hat v_3 \right) \frac{\partial B_1}{\partial \hat v_3} \left( \hat \psi_1, \hat \gamma_3, \hat \ell_3 \right) \bigg]\label{eq_scatteringhat3}\\
	\end{align}
	where the function $\kappa$ is defined by 
	\begin{equation}\label{eq_kappadef}
		\kappa (x) = \sqrt{ \frac{2}{3}} \frac{L_1^2}{\chi} \left[ 1 - \frac{x}{\sinh x} \right]
	\end{equation}
	and the trigonometric polynomials $A_0, B_0, B_1$ were introduced in Lemma \ref{lemma_quad23exp}. 
\end{theorem}

In Section \ref{section_analysisofh0} we established the existence of two hyperbolic periodic orbits $Z^0_{\min}, Z^0_{\max}$ for the Hamiltonian $H_0^{12}$ (see \eqref{def:Zminmax}); moreover we found that the stable and unstable manifolds of these saddles coincide along a heteroclinic trajectory $Z^0$ introduced in Lemma \ref{lemma_separatrixformulas}. Furthermore, in the Poincar\'e variables $\xi, \eta$ (see \eqref{eq_poincarevariables}), the two saddles are reduced to a single hyperbolic periodic orbit, which we denote by $Z^0_*$, and the heteroclinic connection becomes a homoclinic connection to $Z^0_*$; for convenience we denote by $Z^0$ this homoclinic connection. Recall in Section \ref{section_analysisofh0} we considered the phase space of $H_0^{12}$ to be of dimension 4, as the Hamiltonian $H_0^{12}$ depends only on the variables $\gamma_1, \Gamma_1, \tilde \Gamma_2$ (or equivalently on $\xi, \eta, \tilde \Gamma_2$). However we can lift the dynamics to the full secular phase space in the obvious natural way, by setting all remaining secular variables (i.e. $\tilde \psi_1, \tilde \Psi_1, \tilde \gamma_3, \tilde \Gamma_3, \tilde \ell_3, \tilde L_3$) to be constant with respect to $H_0^{12}$. Considering $Z^0_*$, $Z^0$ as functions of all variables in this way, we obtain a parametrisation of the normally hyperbolic invariant manifold $\tilde \Lambda_0$ defined at the beginning of Section \ref{section_inner}, and a parametrisation of the separatrix as in Lemma \ref{lemma_separatrixformulas}.  

 Write $\tilde{F}_{\mathrm{quad}}^{12} = L_2^6 F_{\mathrm{quad}}^{12}$. Then, the integrable Hamiltonian $\tilde{F}_{\mathrm{quad}}^{12}$ possesses a hyperbolic periodic orbit $Z^{\mathrm{quad}}_*$ that is $O(L_2^{-1})$ close to $Z^0_*$; moreover there is a homoclinic orbit $Z^{\mathrm{quad}}$ to $Z^{\mathrm{quad}}_*$ that is $O(L_2^{-1})$ close to $Z^{0}$. Since $\tilde{F}_{\mathrm{quad}}^{12}$ is integrable (see Lemma \ref{lemma_quad12expansion}), the homoclinic trajectory $Z^{\mathrm{quad}}$ corresponds to a non-transverse homoclinic intersection of the stable and unstable manifolds of $Z^{\mathrm{quad}}_*$. 

Denote by $\bar{H} = L_2^6 F_{\mathrm{sec}} - \tilde F_{\mathrm{quad}}^{12} $ the Hamiltonian of the perturbation, and define the Poincar\'e-Melnikov potential by
\begin{align}
	\begin{split}\label{eq_melnikovdef}
		\LPM \left( \tilde{\gamma}_2, \tilde{\psi}_1, \tilde{\gamma}_3, \tilde{\Gamma}_2, \tilde{\Psi}_1, \tilde{\Gamma}_3, \tilde \ell_3, \tilde L_3 \right) ={}& \int_{- \infty}^{\infty} \Big[ \bar{H} \left( Z^{\mathrm{quad}} \left( t, \tilde{\gamma}_2, \tilde{\psi}_1, \tilde{\gamma}_3, \tilde{\Gamma}_2, \tilde{\Psi}_1, \tilde{\Gamma}_3, \tilde \ell_3, \tilde L_3 \right) \right)  \\
		& \quad- \bar{H} \left( Z_*^{\mathrm{quad}} \left( t, \tilde{\gamma}_2, 
		\tilde{\psi}_1, \tilde{\gamma}_3, \tilde{\Gamma}_2, \tilde{\Psi}_1, 
		\tilde{\Gamma}_3, \tilde \ell_3, \tilde L_3 \right) \right) \Big] \, dt. 
	\end{split}
\end{align}
As with $\bar{H}$ itself, the Poincar\'e-Melnikov potential $\LPM$ can be 
expanded in ratios of powers of $L_2$ and $L_3^*$. The following result gives an 
expression for the first-order term at which each angle $\tilde{\gamma}_2, 
\tilde{\psi}_1, \tilde{\gamma}_3, \tilde \ell_3$ appears in the expansion of $\LPM$.

\begin{proposition}\label{proposition_melnikovtildeexp}
	The expansion of the Poincar\'e-Melnikov potential $\LPM$ satisfies the following properties, where the notation $\alpha^{ij}_k$, $H^{ij}_k$ is as in Proposition \ref{proposition_secularexpansion}.
	\begin{enumerate}
		\item
		The first nontrivial term in the expansion of $\LPM$ is $\frac{1}{L_2^2} \alpha_2^{12} \LPM_2^{12}$ where
		\begin{align}
			\LPM_2^{12}\left( \tilde{\gamma}_2, \tilde{\Gamma}_2 \right) ={}& \int_{- \infty}^{\infty} \Big( H_2^{12} \Big(  Z^0 \left( t, \tilde{\gamma}_2, \tilde{\psi}_1, \tilde{\gamma}_3, \tilde \ell_3, \tilde{\Gamma}_2, \tilde{\Psi}_1, \tilde{\Gamma}_3, \tilde L_3 \right) \Big) \\
			& \quad - H_2^{12} \left( Z^0_* \left( t, \tilde{\gamma}_2, \tilde{\psi}_1, \tilde{\gamma}_3, \tilde \ell_3, \tilde{\Gamma}_2, \tilde{\Psi}_1, \tilde{\Gamma}_3, \tilde L_3 \right) \right)\Big) \, dt \\
			={}& \tilde{\LPM}_2^{12} \left( \tilde{\Gamma}_2\right) \sin \tilde{\gamma}_2
		\end{align}
		and where $\tilde{\LPM}_2^{12}$ is an analytic function of $\tilde{\Gamma}_2$ that is nonvanishing for $\tilde \Gamma_2 \in [\zeta_1,\zeta_2]$. 
		\item
		The angles $\tilde{\psi}_1, \tilde \gamma_3, \tilde \ell_3$ all appear for the first time in the expansion of $\LPM$ in the term $\frac{L_2^9}{\left( L_3^* \right)^6} \alpha_1^{23} \LPM_1^{23}$ where
		\begin{align}
			\LPM_1^{23} \left(\tilde{\gamma}_2, \tilde{\psi}_1, \tilde \gamma_3, \tilde \ell_3, \tilde{\Gamma}_2 \right) ={}& \int_{- \infty}^{\infty} \Big(  H_1^{23} \left( Z^0 \left( t, \tilde{\gamma}_2, \tilde{\psi}_1, \tilde{\gamma}_3, \tilde \ell_3, \tilde{\Gamma}_2, \tilde{\Psi}_1, \tilde{\Gamma}_3, \tilde L_3 \right) \right) \\
			& \quad - H_1^{23} \left( Z^0_* \left( t, \tilde{\gamma}_2, \tilde{\psi}_1, \tilde{\gamma}_3, \tilde \ell_3, \tilde{\Gamma}_2, \tilde{\Psi}_1, \tilde{\Gamma}_3, \tilde L_3 \right) \right)\Big) \, dt \\
			={}&  \left( 1 + \sqrt{1 - \delta_2^2} \, \cos v_3 \right)^3 \kappa \left( \frac{\pi \, \tilde \Gamma_2}{A_2 L_1^2} \right)  \\
			& \quad \times \left[ A_1 \left( \tilde \gamma_3, \tilde \psi_1, v_3 \right) \, \cos \tilde \gamma_2 + B_1 \left( \tilde \gamma_3, \tilde \psi_1, v_3 \right) \, \sin \tilde \gamma_2 \right]
		\end{align}
		where the function $\kappa$ is defined by \eqref{eq_kappadef}. 
	\end{enumerate}
\end{proposition}

\begin{proof}
	Part 1 of the proposition was proved in \cite{fejoz2016secular} (see also Proposition 27 and Appendix F of \cite{clarke2022why}). As for part 2, observe that, since $\tilde \psi_1, \tilde \gamma_3, \tilde \ell_3$ are constant with respect to $H_0^{12}$, we can write
	\begin{equation}\label{eq_melnpoin123division}
		\LPM_1^{23} \left(\tilde{\gamma}_2, \tilde{\psi}_1, \tilde \gamma_3, \tilde \ell_3, \tilde{\Gamma}_2 \right) =  \left( 1 + \sqrt{1 - \delta_2^2} \, \cos v_3 \right)^3  \, \left[ A_1 \left( \tilde \gamma_3, \tilde \psi_1, v_3 \right) \, \LPM_1 \left(\tilde \gamma_2, \tilde \Gamma_2 \right) + B_1 \left( \tilde \gamma_3, \tilde \psi_1, v_3 \right) \, \LPM_2 \left(\tilde \gamma_2, \tilde \Gamma_2 \right) \right]
	\end{equation}
	where
	\begin{align}
		\LPM_j \left(\tilde{\gamma}_2, \tilde{\Gamma}_2 \right) ={}& \int_0^{\infty} \left( \F_j \circ Z^0 \left( t, \tilde{\gamma}_2, \tilde{\psi}_1, \tilde{\gamma}_3, \tilde \ell_3, \tilde{\Gamma}_2, \tilde{\Psi}_1, \tilde{\Gamma}_3, \tilde L_3 \right) - \F_j \circ Z^0_{\mathrm{min}} \left( t, \tilde{\gamma}_2, \tilde{\psi}_1, \tilde{\gamma}_3, \tilde \ell_3, \tilde{\Gamma}_2, \tilde{\Psi}_1, \tilde{\Gamma}_3, \tilde L_3 \right) \right) \, dt + \\
		& \quad + \int_{- \infty}^0 \left( \F_j \circ Z^0 \left( t, \tilde{\gamma}_2, \tilde{\psi}_1, \tilde{\gamma}_3, \tilde \ell_3, \tilde{\Gamma}_2, \tilde{\Psi}_1, \tilde{\Gamma}_3, \tilde L_3 \right) - \F_j \circ Z^0_{\mathrm{max}} \left( t, \tilde{\gamma}_2, \tilde{\psi}_1, \tilde{\gamma}_3, \tilde \ell_3, \tilde{\Gamma}_2, \tilde{\Psi}_1, \tilde{\Gamma}_3, \tilde L_3 \right) \right) \, dt,
	\end{align}
	and where the functions
	\begin{equation}
		\F_1 = \sqrt{\Gamma_1^2 - \tilde{\Gamma}_2^2} \, \cos \tilde{\gamma}_2, \quad \F_2 = \sqrt{\Gamma_1^2 - \tilde{\Gamma}_2^2} \, \sin \tilde{\gamma}_2
	\end{equation}
	do not depend on $\tilde{\psi}_1, \tilde{\gamma}_3$. It was shown in Section 6.2 of \cite{clarke2022why} that we have
	\[
	\LPM_1 \left( \tilde \gamma_2, \tilde \Gamma_2 \right) = \kappa \left( \frac{\pi \, \tilde \Gamma_2}{A_2 L_1^2} \right) \cos \tilde \gamma_2, \quad \LPM_2 \left( \tilde \gamma_2, \tilde \Gamma_2 \right) = \kappa \left( \frac{\pi \, \tilde \Gamma_2}{A_2 L_1^2} \right) \sin \tilde \gamma_2
	\]
	where $\kappa$ is the function defined by \eqref{eq_kappadef}. Combining these formulas with \eqref{eq_melnpoin123division} completes the proof of the proposition. 
\end{proof}

\begin{lemma}\label{lemma_scatteringtilde}
	The secular Hamiltonian has two homoclinic channels corresponding to the normally hyperbolic invariant manifold $\L$, and there are two scattering maps defined globally on $\L$ by 
	\[
	\tilde S_{\pm} : \left(\tilde \gamma_2, \tilde \Gamma_2, \tilde \gamma_3, \tilde \Gamma_3, \tilde \psi_1, \tilde \Psi_1, \tilde \ell_3, \tilde L_3 \right) \longmapsto \left(\tilde \gamma_2^*, \tilde \Gamma_2^*, \tilde \gamma_3^*, \tilde \Gamma_3^*, \tilde \psi_1^*, \tilde \Psi_1^*, \tilde \ell_3^*, \tilde L_3^* \right)
	\]
	with
	\[
	\tilde \gamma_2^* = \tilde \gamma_2 + \Delta_0^{\pm} \left( \tilde \Gamma_2; \cdots \right), \quad \tilde \psi_1^* = \tilde \psi_1 + \frac{1}{L_2^2} \Delta_1^{\pm} \left(\tilde \Gamma_2 ; \cdots \right), 
	\]
	\[
	\tilde \gamma_3^* = \tilde \gamma_3 + \frac{L_2^8}{\left( L_3^* \right)^6} \, \Delta_2^{\pm} \left(\tilde \gamma_2, \tilde \Gamma_2, \tilde \gamma_3, \tilde \Gamma_3, \tilde \psi_1, \tilde \Psi_1, \tilde \ell_3, \tilde L_3 \right), \quad \tilde \ell_3^* = \tilde \ell_3 + \frac{L_2^9}{\left(L_3^* \right)^7 } \Delta_3^{\pm} \left(\tilde \gamma_2, \tilde \Gamma_2, \tilde \gamma_3, \tilde \Gamma_3, \tilde \psi_1, \tilde \Psi_1, \tilde \ell_3, \tilde L_3 \right), 
	\]
	\[
	\tilde \Gamma_2^* = \tilde \Gamma_2 + \frac{1}{L_2^3} \Theta_0^{\pm} \left(\tilde \gamma_2, \tilde \Gamma_2, \tilde \gamma_3, \tilde \Gamma_3, \tilde \psi_1, \tilde \Psi_1, \tilde \ell_3, \tilde L_3 \right), \quad \tilde \Psi_1^* = \tilde \Psi_1 + \frac{L_2^9}{\left(L_3^*\right)^6} \Theta_1^{\pm} \left( \tilde \psi_1, \tilde \gamma_3, \tilde \ell_3, \tilde \Gamma_2; \cdots \right),
	\]
	\[
	\tilde \Gamma_3^* = \tilde \Gamma_3 + \frac{L_2^9}{\left(L_3^*\right)^6} \Theta_2^{\pm} \left( \tilde \psi_1, \tilde \gamma_3, \tilde \ell_3, \tilde \Gamma_2; \cdots \right), \quad \tilde L_3^* = \tilde L_3 + \frac{L_2^9}{\left(L_3^*\right)^6} \Theta_3^{\pm} \left( \tilde \psi_1, \tilde \gamma_3, \tilde \ell_3, \tilde \Gamma_2; \cdots \right)
	\]
	where the ellipsis after the semicolon denotes dependence on the remaining variables at higher order, and where
	\[
	\Delta_0^{\pm} \left( \tilde \Gamma_2; \cdots \right) = 2 \, \arctan \chi^{-1} + \cdots, \quad \Delta_1^{\pm} \left( \tilde \Gamma_2; \cdots \right) = \frac{L_1}{6} \, \sqrt{\frac{3}{2}} \, \alpha_2^{12} \, \tilde{\Gamma}_2 \, \sqrt{1 - \frac{5}{3} \frac{\tilde \Gamma_2^2}{L_1^2}}+ \cdots
	\]
	\[
	\Delta_2^{\pm} \left(\tilde \gamma_2, \tilde \Gamma_2, \tilde \gamma_3, \tilde \Gamma_3, \tilde \psi_1, \tilde \Psi_1, \tilde \ell_3, \tilde L_3 \right) = O(1), \quad \Delta_3^{\pm} \left(\tilde \gamma_2, \tilde \Gamma_2, \tilde \gamma_3, \tilde \Gamma_3, \tilde \psi_1, \tilde \Psi_1, \tilde \ell_3, \tilde L_3 \right) = O(1)
	\]
	and
	\begin{align}
		\Theta_0^{\pm} \left(\tilde \gamma_2, \tilde \Gamma_2, \tilde \gamma_3, \tilde \Gamma_3, \tilde \psi_1, \tilde \Psi_1, \tilde \ell_3, \tilde L_3 \right) ={}& O(1), \\
		\Theta_1^{\pm} \left( \tilde \psi_1, \tilde \gamma_3, \tilde \ell_3, \tilde \Gamma_2; \cdots \right) ={}& \mp  \alpha_1^{23} \, \left( 1 + \sqrt{1 - \delta_2^2} \, \cos v_3 \right)^3 \kappa \left( \frac{\pi \, \tilde{\Gamma}_2}{A_2 \, L_1^2} \right) \,  \frac{\partial B_1}{\partial \tilde \psi_1} \left( \tilde \psi_1, \tilde \gamma_3, \tilde \ell_3 \right) + \cdots \\
		\Theta_2^{\pm} \left( \tilde \psi_1, \tilde \gamma_3, \tilde \ell_3, \tilde \Gamma_2; \cdots \right) ={}& \mp  \alpha_1^{23} \, \left( 1 + \sqrt{1 - \delta_2^2} \, \cos v_3 \right)^3 \kappa \left( \frac{\pi \, \tilde{\Gamma}_2}{A_2 \, L_1^2} \right) \,  \frac{\partial B_1}{\partial \tilde \gamma_3} \left( \tilde \psi_1, \tilde \gamma_3, \tilde \ell_3 \right) + \cdots, \\
		\Theta_3^{\pm} \left( \tilde \psi_1, \tilde \gamma_3, \tilde \ell_3, \tilde \Gamma_2; \cdots \right) ={}& \mp   \alpha_1^{23} \, \delta_2^{-3} \, \left( 1 + \sqrt{1 - \delta_2^2} \, \cos v_3 \right)^4 \, \kappa \left( \frac{\pi \, \tilde{\Gamma}_2}{A_2 \, L_1^2} \right) \bigg[ -3 \,\sqrt{1 - \delta_2^2} \, \sin v_3 \\
		& \quad \times  B_1 \left( \tilde \psi_1, \tilde \gamma_3, \tilde \ell_3 \right) + \left( 1 + \sqrt{1 - \delta_2^2} \, \cos v_3 \right) \frac{\partial B_1}{\partial v_3} \left( \tilde \psi_1, \tilde \gamma_3, \tilde \ell_3 \right) \bigg] + \cdots
	\end{align}
\end{lemma}

\begin{proof}
	Denote by $\left( \omega_0, \omega_1, \omega_2, \omega_3 \right)$ the frequency vector of 
	the angles $\left( \tilde{\gamma}_2, \tilde{\psi}_1, \tilde{\gamma}_3, \tilde \ell_3 \right)$ 
	on a torus on $\Lambda$ corresponding to fixed values of the actions 
	$\tilde{\Gamma}_2, \tilde{\Psi}_1, \tilde{\Gamma}_3, \tilde L_3$ for the Hamiltonian 
	$\tilde{F} = L_2^6 \left( \qot + F_{\mathrm{Kep}} - \tilde F_{\mathrm{Kep}} \right)$. The computations of Section \ref{sec_secularexpansion} therefore imply that 
	\[
	\omega_0=2\alpha_0^{12}\frac{\tilde\Gamma_2}{L_1^2}+O\left(\frac{1}{L_2}
	\right), \quad \omega_1=O\left(\frac{1}{L_2}
	\right),\quad \omega_2=0, \quad \omega_3 = \frac{L_2^6}{\left( L_3^* \right)^3} \, \akep + O \left( \frac{L_2^6}{\left( L_3^* \right)^4} \right). 
	\]
	Consider the function
	\begin{equation}\label{eq_melnikovfrequencymap}
		\tau \longmapsto \LPM \left( \tilde{\gamma}_2 - \omega_0 \, \tau, \tilde{\psi}_1 - \omega_1 \, \tau, \tilde{\gamma}_3 - \omega_2 \, \tau, \tilde \ell_3 - \omega_3 \, \tau, \tilde{\Gamma}_2, \tilde{\Psi}_1, \tilde{\Gamma}_3, \tilde L_3 \right) 
	\end{equation}
	where $\LPM$ is the Poincar\'e-Melnikov potential defined by 
	\eqref{eq_melnikovdef}. Results  of \cite{delshams2006biggaps} imply that 
	nondegenerate critical points of \eqref{eq_melnikovfrequencymap} correspond to 
	transverse homoclinic intersections of the stable and unstable manifolds of 
	$\Lambda$. Equation \eqref{eq_melnikovdef}, Proposition 
	\ref{proposition_secularexpansion}, and Proposition 
	\ref{proposition_melnikovtildeexp} imply that 
	\begin{equation}
		\LPM = \frac{1}{L_2^2} \alpha_2^{12} \LPM_2^{12} + \cdots 
	\end{equation}
	The function $\tau \mapsto \LPM_2^{12} \left( \tilde{\gamma}_2 - \omega_0 \, \tau, \tilde{\Gamma}_2 \right)$ has nondegenerate critical points $\tau_{\pm}$ where $\omega_0 \tau_{\pm} = \tilde{\gamma}_2 \pm \frac{\pi}{2} $. It follows that there are functions 
	\begin{equation}
		\tau^*_{\pm} \left( \tilde{\gamma}_2, \tilde{\psi}_1, \tilde{\gamma}_3, \tilde \ell_3, \tilde{\Gamma}_2, \tilde{\Psi}_1, \tilde{\Gamma}_3, \tilde L_3 \right) = \frac{1}{\omega_0} \left( \tilde{\gamma}_2 \pm \frac{\pi}{2} \right) + \cdots
	\end{equation}
	such that
	\begin{equation}
		\left. \frac{d}{d \tau} \right|_{\tau = \tau^*_{\pm}} \LPM \left( \tilde{\gamma}_2 - \omega_0 \, \tau, \tilde{\psi}_1 - \omega_1 \, \tau, \tilde{\gamma}_3 - \omega_2 \, \tau, \tilde \ell_3 - \omega_3 \, \tau, \tilde{\Gamma}_2, \tilde{\Psi}_1, \tilde{\Gamma}_3, \tilde L_3 \right)  = 0. 
	\end{equation}
	We now introduce the reduced Poincar\'e-Melnikov potentials
	\begin{equation}
		\LPM^*_{\pm} \left( \tilde{\gamma}_2, \tilde{\psi}_1, \tilde{\gamma}_3, \tilde \ell_3, \tilde{\Gamma}_2, \tilde{\Psi}_1, \tilde{\Gamma}_3, \tilde L_3 \right) = \LPM \left( \tilde{\gamma}_2 - \omega_0 \, \tau^*_{\pm}, \tilde{\psi}_1 - \omega_1 \, \tau^*_{\pm}, \tilde{\gamma}_3 - \omega_2 \, \tau^*_{\pm}, \tilde \ell_3 - \omega_3 \, \tau^*_{\pm}, \tilde{\Gamma}_2, \tilde{\Psi}_1, \tilde{\Gamma}_3, \tilde L_3 \right). 
	\end{equation}
	Now, following again \cite{delshams2006biggaps}, the changes in the actions 
	coming from the scattering maps $\tilde{S}_{\pm}$ are defined using the 
	functions $\LPM^*_{\pm}$ via
	\begin{equation}
		\tilde{\Gamma}_2^* = \tilde{\Gamma}_2 + \frac{\partial \LPM^*_{\pm}}{\partial 
			\tilde{\gamma}_2} + \cdots, \quad \tilde{\Psi}_1^* = \tilde{\Psi}_1 + 
		\frac{\partial \LPM^*_{\pm}}{\partial \tilde{\psi}_1} + \cdots, \quad 
		\tilde{\Gamma}_3^* = \tilde{\Gamma}_3 + \frac{\partial \LPM^*_{\pm}}{\partial 
			\tilde{\gamma}_3} + \cdots, \quad \tilde L_3^* = \tilde L_3 + \frac{\partial \LPM^*_{\pm}}{\partial \tilde \ell_3} + \cdots
	\end{equation}
	Note that the cylinder frequencies in the model in \cite{delshams2006biggaps} 
	all have the same time scale and moreover the first order of the perturbation 
	depends on all the angles. On the contrary, in our model there are multiple time scales and the angles $\tilde\psi_1$, $\tilde\gamma_3$, and $\tilde \ell_3$ appear 
	only at higher order terms (see Proposition 
	\ref{proposition_secularexpansion}). Still, one can easily check that the 
	statements in \cite{delshams2006biggaps} are still valid in the present setting. 
	The only difference is that the first order of the scattering maps in the 
	actions $\tilde\Psi_1$, $\tilde\Gamma_3$, $\tilde L_3$ come from higher orders of the 
	Melnikov potential.

By part 2 of Proposition \ref{proposition_melnikovtildeexp}, the first term in the expansion of the Poincar\'e-Melnikov potential depending on any of the angles $\tilde \psi_1, \tilde \gamma_3, \tilde \ell_3$ is $\frac{L_2^9}{\left(L_3^* \right)^6} \alpha_1^{23} \LPM_1^{23}$, and in fact all three of those angles appear in this term. It follows that the first term in the expansion of the reduced Poincar\'e-Melnikov potentials depending on the angles $\tilde \psi_1, \tilde \gamma_3, \tilde \ell_3$ is $\frac{L_2^9}{\left(L_3^* \right)^6} \alpha_1^{23} \left( \LPM_1^{23} \right)_{\pm}^*$ where
\begin{align}
\left( \LPM_1^{23} \right)_{\pm}^* =& \LPM_1^{23} \left( \tilde \gamma_2 - \omega_0 \, \tau_{\pm}^*, \tilde \psi_1 - \omega_1 \, \tau_{\pm}^*, \tilde \gamma_3 - \omega_2 \, \tau_{\pm}^*, \tilde \ell_3 - \omega_3 \, \tau_{\pm}^*, \tilde \Gamma_2 \right) \\
=& \LPM_1^{23} \left( - \frac{\pi}{2}, \tilde \psi_1, \tilde \gamma_3, \tilde \ell_3, \tilde \Gamma_2 \right) + \cdots \\
=& \mp \left( 1 + \sqrt{1 - \delta_2^2} \, \cos  v_3 \right)^3 \, \kappa \left( \frac{\pi \, \tilde \Gamma_2}{A_2 L_1^2} \right) \, B_1 \left( \tilde \gamma_3, \tilde \psi_1, \tilde \ell_3 \right) + \cdots
\end{align}
It follows that
	\begin{equation}
		\frac{\partial \LPM^*_{\pm}}{\partial \tilde{\psi}_1} = \frac{L_2^9}{\left(L_3^*\right)^6} \, \alpha_1^{23} \, \frac{\partial \left( \LPM_1^{23} \right)^*_{\pm}}{\partial \tilde{\psi}_1} + \cdots = \mp \frac{L_2^9}{\left(L_3^*\right)^6} \, \alpha_1^{23} \, \left( 1 + \sqrt{1 - \delta_2^2} \, \cos v_3 \right)^3 \kappa \left( \frac{\pi \, \tilde{\Gamma}_2}{A_2 \, L_1^2} \right) \,  \frac{\partial B_1}{\partial \tilde \psi_1} \left( \tilde \psi_1, \tilde \gamma_3, \tilde \ell_3 \right) + \cdots,
	\end{equation}
	\begin{equation}
		\frac{\partial \LPM^*_{\pm}}{\partial \tilde{\gamma}_3} = \frac{L_2^9}{\left(L_3^*\right)^6} \, \alpha_1^{23} \, \frac{\partial \left( \LPM_1^{23} \right)^*_{\pm}}{\partial \tilde{\gamma}_3} + \cdots = \mp \frac{L_2^9}{\left(L_3^*\right)^6} \, \alpha_1^{23} \, \left( 1 + \sqrt{1 - \delta_2^2} \, \cos v_3 \right)^3 \kappa \left( \frac{\pi \, \tilde{\Gamma}_2}{A_2 \, L_1^2} \right) \,  \frac{\partial B_1}{\partial \tilde \gamma_3} \left( \tilde \psi_1, \tilde \gamma_3, \tilde \ell_3 \right) + \cdots,
	\end{equation}
	and
	\begin{align}
		\frac{\partial \LPM^*_{\pm}}{\partial \tilde{\ell}_3} ={}& \frac{L_2^9}{\left(L_3^*\right)^6} \, \alpha_1^{23} \, \frac{\partial \left( \LPM_1^{23} \right)^*_{\pm}}{\partial \tilde{\ell}_3} + \cdots = \frac{L_2^9}{\left(L_3^*\right)^6} \, \alpha_1^{23} \, \frac{\partial \left( \LPM_1^{23} \right)^*_{\pm}}{\partial v_3} \frac{\partial v_3}{ \partial \tilde \ell_3}+ \cdots \\
		={}& \mp  \frac{L_2^9}{\left(L_3^*\right)^6} \, \alpha_1^{23} \, \delta_2^{-3} \, \left( 1 + \sqrt{1 - \delta_2^2} \, \cos v_3 \right)^4 \, \kappa \left( \frac{\pi \, \tilde{\Gamma}_2}{A_2 \, L_1^2} \right) \\
		& \quad \times  \left[ - 3 \sqrt{1 - \delta_2^2} \, \sin v_3 \, B_1 \left( \tilde \psi_1, \tilde \gamma_3, \tilde \ell_3 \right) + \left( 1 + \sqrt{1 - \delta_2^2} \, \cos v_3 \right) \frac{\partial B_1}{\partial v_3} \left( \tilde \psi_1, \tilde \gamma_3, \tilde \ell_3 \right) \right] + \cdots
	\end{align}
	where we have used \eqref{eq_keplersecondlaw}.

	For the angles $\tilde{\gamma}_2, \tilde{\psi}_1, \tilde{\gamma}_3, \tilde \ell_3$, the first-order term under application of the scattering map is a so-called phase shift. This is a change in the angle that comes from the integrable part of the Hamiltonian along the separatrix, and does not necessarily depend on the functions $\LPM^*_{\pm}$ at first order. The phase shifts in the angles $\tilde \gamma_2$, $\tilde \psi_1$ come from $\qot$, and are therefore the same as in \cite{clarke2022why} (see Lemma 30 and Appendix D). For the angles $\tilde \gamma_3$ and $\tilde \ell_3$, we simply estimate the order of the phase shift. Notice that the phase shift comes from the first appearance of the symplectic conjugate action in the Hamiltonian multiplying by functions of $\Gamma_1$, $\gamma_1$, $\tilde \gamma_2$, as these are the variables that behave differently on the separatrix and on $\L$. Since $\tilde \Gamma_3$ first appears in the secular Hamiltonian in the term of order $\frac{L_2^3}{\left( L_3^* \right)^6}$, and since this term does not provide a phase shift in $\tilde \gamma_3$, the highest-order term that can potentially produce a phase shift has an additional factor of $\frac{1}{L_2}$. Since we normalise the entire secular Hamiltonian by $L_2^6$, we see that the phase shift in $\tilde \gamma_3$ is of order $O \left(\frac{L_2^3}{\left( L_3^* \right)^6} \, \frac{1}{L_2} \, L_2^6 \right) = O \left( \frac{L_2^8}{\left( L_3^* \right)^6} \right)$. As for $\tilde \ell_3$, the phase shift comes from $H_1^{23}$ upon expanding its coefficient 
	\[
	\frac{L_2^3}{L_3^6} = \frac{L_2^3}{\left( L_3^* + \tilde L_3 \right)^6} = \frac{L_2^3}{\left( L_3^* \right)^6} - 6 \frac{L_2^3}{\left( L_3^* \right)^7} \tilde L_3 + \cdots
	\]
	in powers of $\frac{1}{\tilde L_3}$. Upon scaling this by $L_2^6$, we see that the first term that can provide a phase shift is of order $O \left( \frac{L_2^9}{\left( L_3^* \right)^7} \right)$, as required. 
\end{proof}

\begin{lemma}\label{lemma_scatteringhat}
	In the `hat' coordinates, the scattering maps $S_{\pm} : \L \to \L'$, introduced in Lemma \ref{lemma_scatteringtilde}, are given by
	\[
	S_{\pm} : \left(\hat \gamma_2, \hat \Gamma_2, \hat \gamma_3, \hat \Gamma_3, \hat \psi_1, \hat \Psi_1, \hat \ell_3, \hat L_3 \right) \longmapsto \left(\hat \gamma_2^*, \hat \Gamma_2^*, \hat \gamma_3^*, \hat \Gamma_3^*, \hat \psi_1^*, \hat \Psi_1^*, \hat \ell_3^*, \hat L_3^* \right)
	\]
	with
	\[
	\begin{dcases}
		\hat \Psi_1^* =& \hat \Psi_1 + \frac{L_2^9}{\left( L_3^* \right)^6} \S_1^{\pm}  \left( \hat \psi_1, \hat \gamma_3, \hat \ell_3, \hat \Gamma_2 \right) + \cdots \\
		\hat \Gamma_3^* =& \hat \Gamma_3 + \frac{L_2^9}{\left( L_3^* \right)^6} \S_2^{\pm}  \left( \hat \psi_1, \hat \gamma_3, \hat \ell_3, \hat \Gamma_2 \right) + \cdots \\
		\hat L_3^* =& \hat L_3 + \frac{L_2^9}{\left( L_3^* \right)^6} \S_3^{\pm} \left( \hat \psi_1, \hat \gamma_3, \hat \ell_3, \hat \Gamma_2 \right) + \cdots
	\end{dcases}
	\]
	where $\S_j^{\pm}$ are given by \eqref{eq_scatteringhat1}, \eqref{eq_scatteringhat2}, and \eqref{eq_scatteringhat3} respectively. 
\end{lemma}

\begin{proof}
	Recall we denote by $\Phi$ the coordinate transformation on $\L$ from `tilde' coordinates to `hat' coordinates constructed in Theorem \ref{theorem_inner}. Moreover, by Lemma \ref{lemma_inneraveraging}, we can write 
	\[
	\hat \Psi_1 = \tilde \Psi_1 + \frac{L_2^{11}}{\left( L_3^* \right)^6} \Phi_1 \left( \tilde \psi_1, \tilde \gamma_3, \tilde \ell_3, \tilde \Gamma_2; \cdots \right),  \quad  \hat \Gamma_3 = \tilde \Gamma_3 + \frac{L_2^4}{\left( L_3^* \right)^3} \Phi_2 \left( \tilde \gamma_3, \tilde \ell_3; \cdots \right) \quad \hat L_3 = \tilde L_3 + \frac{L_2^4}{\left( L_3^* \right)^3} \Phi_3 \left( \tilde \gamma_3, \tilde \ell_3; \cdots \right)
	\]
	with
	\begin{align}
		\Phi_1 \left( \tilde \psi_1, \tilde \gamma_3, \tilde \ell_3, \tilde \Gamma_2; \cdots \right) =& - \frac{\alpha_0^{23}}{\alpha_1^{12}}\frac{\left(1 + \sqrt{1 - \delta_2^2} \, \cos \tilde v_3 \right)^3}{2 \left(3 \frac{\left( \tilde \Gamma_2 \right)^2}{L_1^2} - 1 \right)} \left[ A_0 \left(\tilde \gamma_3, \tilde v_3 \right) \, \cos 2 \tilde \psi_1 + B_0 \left( \tilde \gamma_3, \tilde v_3 \right) \, \sin 2 \tilde \psi_1 \right] + \cdots\\
		\Phi_2 \left( \tilde \gamma_3, \tilde \ell_3; \cdots \right) =& - \frac{\alpha_0^{23}}{\akep} \frac{\beta_0 \, \delta_2^3}{6} \left[ \sqrt{1 - \delta_2^2} \, \left(\cos \left(3 \tilde v_3 - 2 \tilde \gamma_3 \right) + 3 \cos \left( \tilde v_3 - 2 \tilde \gamma_3 \right) \right) + 3 \cos \left(2 \tilde v_3 - 2 \tilde \gamma_3 \right) \right]  + \cdots  \\
		\Phi_3 \left( \tilde \gamma_3, \tilde \ell_3; \cdots \right) =& - \frac{\alpha_0^{23}}{\akep} \left( \left(1 + \sqrt{1 - \delta_2^2} \, \cos \tilde v_3 \right)^3  \left[ \beta_0 \, \sin^2 \left( \tilde v_3 - \tilde \gamma_3 \right) + \beta_1 \right] - \delta_2^3 \, \left( \frac{\beta_0}{2} + \beta_1 \right) \right) + \cdots
	\end{align}
	where the constants $\beta_j$ are defined by \eqref{eq_averagingbetaconstants}. Therefore, using the notation of Lemma \ref{lemma_scatteringtilde}, we have
	\begin{align}
		\hat \Psi_1^* =& \tilde \Psi_1^* + \frac{L_2^{11}}{\left( L_3^* \right)^6} \Phi_1 \left( \tilde \psi_1^*, \tilde \gamma_3^*, \tilde \ell_3^*, \tilde \Gamma_2^*; \cdots \right) \\
		=& \tilde \Psi_1 + \frac{L_2^9}{\left(L_3^*\right)^6} \Theta_1^{\pm} \left( \tilde \psi_1, \tilde \gamma_3, \tilde \ell_3, \tilde \Gamma_2; \cdots \right) \\
		& \quad +  \frac{L_2^{11}}{\left( L_3^* \right)^6} \Phi_1 \left( \tilde \psi_1 + \frac{1}{L_2^2} \Delta_1^{\pm} \left(\tilde \Gamma_2 ; \cdots \right), \tilde \gamma_3 + O \left(\frac{L_2^8}{\left( L_3^* \right)^6} \right), \tilde \ell_3 + O \left( \frac{L_2^9}{\left(L_3^* \right)^7 } \right), \tilde \Gamma_2 + O \left( \frac{1}{L_2^3} \right); \cdots \right) \\
		=& \hat \Psi_1 + \frac{L_2^9}{\left( L_3^* \right)^6} \left[ \Theta_1^{\pm} \left( \tilde \psi_1, \tilde \gamma_3, \tilde \ell_3, \tilde \Gamma_2; \cdots \right) + \partial_{\tilde \psi_1} \Phi_1 \left( \tilde \psi_1, \tilde \gamma_3, \tilde \ell_3, \tilde \Gamma_2; \cdots \right) \Delta_1^{\pm} \left(\tilde \Gamma_2 ; \cdots \right) \right] + \cdots \\
		=& \hat \Psi_1 + \frac{L_2^9}{\left( L_3^* \right)^6} \left[ \Theta_1^{\pm} \left( \hat \psi_1, \hat \gamma_3, \hat \ell_3, \hat \Gamma_2; \cdots \right) + \partial_{\hat \psi_1} \Phi_1 \left( \hat \psi_1, \hat \gamma_3, \hat \ell_3, \hat \Gamma_2; \cdots \right) \Delta_1^{\pm} \left(\hat \Gamma_2 ; \cdots \right) \right] + \cdots, \label{eq_psi1hatstarcomp}
	\end{align}
	\begin{align}
		\hat \Gamma_3^* =& \tilde \Gamma_3^* + \frac{L_2^4}{\left( L_3^* \right)^3} \Phi_2 \left( \tilde \gamma_3^*, \tilde \ell_3^*; \cdots \right) \\
		=& \tilde \Gamma_3 + \frac{L_2^9}{\left(L_3^*\right)^6} \Theta_2^{\pm} \left( \tilde \psi_1, \tilde \gamma_3, \tilde \ell_3, \tilde \Gamma_2; \cdots \right)   + \frac{L_2^4}{\left( L_3^* \right)^3} \Phi_2 \left( \tilde \gamma_3 + O \left(\frac{L_2^8}{\left( L_3^* \right)^6} \right), \tilde \ell_3 + O \left( \frac{L_2^9}{\left(L_3^* \right)^7 } \right); \cdots \right) \\
		=& \hat \Gamma_3 + \frac{L_2^9}{\left(L_3^*\right)^6} \Theta_2^{\pm} \left( \hat \psi_1, \hat \gamma_3, \hat \ell_3, \hat \Gamma_2; \cdots \right) + \cdots, \label{eq_gamma3hatstarcomp}
	\end{align}
	and
	\begin{align}
		\hat L_3^* =& \tilde L_3^* + \frac{L_2^4}{\left( L_3^* \right)^3} \Phi_3 \left( \tilde \gamma_3^*, \tilde \ell_3^*; \cdots \right) \\
		=& \tilde L_3 + \frac{L_2^9}{\left(L_3^*\right)^6} \Theta_3^{\pm} \left( \tilde \psi_1, \tilde \gamma_3, \tilde \ell_3, \tilde \Gamma_2; \cdots \right) + \frac{L_2^4}{\left( L_3^* \right)^3} \Phi_3 \left( \tilde \gamma_3 + O \left(\frac{L_2^8}{\left( L_3^* \right)^6} \right), \tilde \ell_3 + O \left( \frac{L_2^9}{\left(L_3^* \right)^7 } \right); \cdots \right) \\
		=& \hat L_3 + \frac{L_2^9}{\left(L_3^*\right)^6} \Theta_3^{\pm} \left( \hat \psi_1, \hat \gamma_3, \hat \ell_3, \hat \Gamma_2; \cdots \right) + \cdots \label{eq_l3hatstarcomp}
	\end{align}
	Combining \eqref{eq_gamma3hatstarcomp} and \eqref{eq_l3hatstarcomp} with the formulas given for $\Theta_2^{\pm}$ and $\Theta_3^{\pm}$ in Lemma \ref{lemma_scatteringtilde} already gives the expressions \eqref{eq_scatteringhat2} and \eqref{eq_scatteringhat3}. From \eqref{eq_psi1hatstarcomp}, the formula for $\Theta_1^{\pm}$ given in Lemma \ref{lemma_scatteringtilde} and the formula above for $\Phi_1$ we find that 
	\begin{align}
		\S_1^{\pm}  \left( \hat \psi_1, \hat \gamma_3, \hat \ell_3, \hat \Gamma_2 \right) =& \Theta_1^{\pm} \left( \hat \psi_1, \hat \gamma_3, \hat \ell_3, \hat \Gamma_2 \right) + \partial_{\hat \psi_1} \Phi_1 \left( \hat \psi_1, \hat \gamma_3, \hat \ell_3, \hat \Gamma_2 \right) \Delta_1^{\pm} \left(\hat \Gamma_2  \right) \\
		=& \mp  \alpha_1^{23} \, \left( 1 + \sqrt{1 - \delta_2^2} \, \cos \hat v_3 \right)^3 \kappa \left( \frac{\pi \, \hat{\Gamma}_2}{A_2 \, L_1^2} \right) \,  \frac{\partial B_1}{\partial \hat \psi_1} \left( \hat \psi_1, \hat \gamma_3, \hat \ell_3 \right) + \frac{\alpha_0^{23} \, \alpha_2^{12}}{\alpha_1^{12}}\frac{L_1}{6} \sqrt{\frac{3}{2}} \hat \Gamma_2 \\
		& \quad \times   \sqrt{1 - \frac{5}{3} \frac{\hat \Gamma_2^2}{L_1^2}} \,  \frac{\left(1 + \sqrt{1 - \delta_2^2} \, \cos \hat v_3 \right)^3}{\left(3 \frac{\left( \hat \Gamma_2 \right)^2}{L_1^2} - 1 \right)}  \left[ A_0 \left(\hat \gamma_3, \hat \ell_3 \right) \, \sin 2 \hat \psi_1 - B_0 \left( \hat \gamma_3, \hat \ell_3 \right) \, \cos 2 \hat \psi_1 \right],
	\end{align}
	which completes the proof of the lemma. 
\end{proof}

\section{Shadowing in a Poincar\'e Section}\label{sec:shadowing}

In this section we prove Theorem \ref{thm:MainHierarch:Deprit} by proving the existence of suitable transition chains of almost invariant tori, and applying the shadowing results of \cite{clarke2022topological} (which are summarised in Appendix \ref{appendix_shadowing}) to obtain orbits of the full four-body problem that visit arbitrarily small neighbourhoods of these tori in the prescribed order. As the shadowing results of \cite{clarke2022topological} are proved for mappings, we first choose a suitable Poincar\'e section, and analyse the corresponding first return map; this is done is Section \ref{subsec_poincaremap}. In Section \ref{section_transitionchains} we prove that we can find pseudo-orbits (that is orbits of the iterated function system consisting of the Poincar\'e map and the scattering maps) that connect (up to small errors) any given sequence of almost invariant tori on the normally hyperbolic cylinder. Finally in Section \ref{section_shadowingsubsec} we apply the results of \cite{clarke2022topological} to complete the proof of Theorem \ref{thm:MainHierarch:Deprit}. 

\subsection{Reduction to a Poincar\'e Map}\label{subsec_poincaremap}

Denote by $\tilde \D$ the region of the phase space where the coordinates $\left(\xi, \eta, \tilde \gamma_2,\tilde \Gamma_2, \tilde \psi_1, \tilde \Psi_1, \tilde \gamma_3, \tilde \Gamma_3, \tilde \ell_3, \tilde L_3 \right)$ take the following values: $\xi, \eta$ live in the open ball in $\mathbb{R}^2$ centred at the origin of radius $\sqrt{2 L_1}$; the angles $\tilde \gamma_2, \tilde \psi_1, \tilde \gamma_3, \tilde \ell_3 \in \mathbb{T}$; and the actions satisfy $\tilde \Gamma_2 \in [ \zeta_1, \zeta_2]$ and $\tilde \Psi_1, \tilde \Gamma_3, \tilde L_3 \in [-1,1]$, where the constants $\zeta_1, \zeta_2$ were defined in \eqref{eq_nhimparametersdef}. We now make the following further refinement to these constants:
\begin{equation}\label{eq_gamm2finalinterval}
	0 < \zeta_1 < \zeta_2 < \frac{L_1}{\sqrt{3}}. 
\end{equation}
This refinement guarantees that our inner map satisfies a twist condition; see Lemma \ref{lemma_twist} below. By slightly shrinking the region $\tilde \D$ if necessary, we obtain a region $\D$ in which there is a (not necessarily symplectic) near-identity coordinate transformation 
\begin{equation}
\Upsilon	: \left(\xi, \eta, \tilde \gamma_2,\tilde \Gamma_2, \tilde \psi_1, \tilde \Psi_1, \tilde \gamma_3, \tilde \Gamma_3, \tilde \ell_3, \tilde L_3 \right) \longmapsto \left(\xi, \eta, \hat \gamma_2, \hat \Gamma_2, \hat \psi_1, \hat \Psi_1, \hat \gamma_3, \hat \Gamma_3,\hat \ell_3, \hat L_3 \right)
\end{equation}
where $\left(\hat \gamma_2, \hat \Gamma_2, \hat \psi_1, \hat \Psi_1, \hat \gamma_3, \hat \Gamma_3,\hat \ell_3, \hat L_3 \right)$ are the coordinates provided by Theorem \ref{theorem_inner}. We consider the range of energies $\E = \{ \Fsec (z) : z \in \D \} \subset \mathbb{R}$. 

\begin{theorem}\label{theorem_poincare}
	Choose $E_0 \in \E$ and consider the Poincar\'e section $M = \{ \hat \gamma_2 = 0 \} \cap \{ \Fsec = E_0 \} \cap \D$. We have the following. 
	\begin{enumerate}
		\item
		The flow of $\Fsec$ gives rise to a well-defined Poincar\'e map $F : M \to M$, and the set $\hat \L = \L \cap M$ is a normally hyperbolic invariant manifold for $F$. 
		\item
		The variables $\left( \hat \psi_1, \hat \gamma_3, \hat \ell_3, \hat \Psi_1, \hat \Gamma_3, \hat L_3 \right)$ define coordinates on $\hat \L$ and the inner map $f = F|_{\hat \L}$ is of the form
		\begin{equation}
			f :  \label{eq_innermaptheorem}
			\begin{dcases}
				\left( \hat \psi_1^*, \hat \gamma_3^*, \hat \ell_3^* \right) &= \left( \hat \psi_1, \hat \gamma_3, \hat \ell_3 \right) + g  \left( \hat \Psi_1, \hat \Gamma_3, \hat L_3 \right) + O \left( \eps^{k_1 - 6} \mu^{k_2} \right)\\
				\left( \hat \Psi_1^*, \hat \Gamma_3^*, \hat L_3^* \right) &= \left( \hat \Psi_1, \hat \Gamma_3, \hat L_3 \right) + O \left( \eps^{k_1 - 6} \mu^{k_2} \right)
			\end{dcases}
		\end{equation}
		where $\eps = \frac{1}{L_2}$, $\mu = \frac{L_2}{L_3}$, where $k_1, k_2$ come from part 2 of Theorem \ref{theorem_inner}, and where
		\begin{equation}
			\det D g  \left( \hat \Psi_1, \hat \Gamma_3, \hat L_3 \right) \neq 0. 
		\end{equation}
		Moreover the bottom eigenvalue of $D g  \left( \hat \Psi_1, \hat \Gamma_3, \hat L_3 \right)$ is of order $\frac{\mu^6}{\eps^2}$. 
		\item
		There are two scattering maps $\hat S_{\pm} : \hat \L \to \hat \L'$ where $\hat \L'$ is an open cylinder in $\mathbb{T}^3 \times \mathbb{R}^3$ containing $\hat \L$. Moreover with $\left( \hat \psi_1^*, \hat \gamma_3^*, \hat \ell_3^*, \hat \Psi_1^*, \hat \Gamma_3^*, \hat L_3^* \right) = \hat S_{\pm} \left( \hat \psi_1, \hat \gamma_3, \hat \ell_3, \hat \Psi_1, \hat \Gamma_3, \hat L_3 \right)$ we have
		\[
		\begin{dcases}
			\hat \Psi_1^* =& \hat \Psi_1 + \frac{L_2^9}{\left( L_3^* \right)^6} \S_1^{\pm}  \left( \hat \psi_1, \hat \gamma_3, \hat \ell_3, \hat \Gamma_2 \right) + \cdots \\
			\hat \Gamma_3^* =& \hat \Gamma_3 + \frac{L_2^9}{\left( L_3^* \right)^6} \S_2^{\pm}  \left( \hat \psi_1, \hat \gamma_3, \hat \ell_3, \hat \Gamma_2 \right) + \cdots \\
			\hat L_3^* =& \hat L_3 + \frac{L_2^9}{\left( L_3^* \right)^6} \S_3^{\pm} \left( \hat \psi_1, \hat \gamma_3, \hat \ell_3, \hat \Gamma_2 \right) + \cdots
		\end{dcases}
		\]
		where $\S_j^{\pm}$ are given by \eqref{eq_scatteringhat1}, \eqref{eq_scatteringhat2}, and \eqref{eq_scatteringhat3} respectively. 
	\end{enumerate}
\end{theorem}

\begin{remark}\label{remark_inclinationimpliestwist}
The inequality \eqref{eq_gamm2finalinterval} gives us the true range of values of the mutual inclination $i_{12}$ along the diffusive orbits we have found (see also Remark \ref{remark_inclinationimplieshyperbolicity}). Indeed, the computation $\cos i_{12} = \frac{\tilde \Gamma_2}{\Gamma_1} + O \left( L_2^{-1} \right)$ that we made in Remark \ref{remark_inclinationimplieshyperbolicity} combined with the fact that $\tilde \Gamma_2 \in [\zeta_1, \zeta_2]$ where $\zeta_1, \zeta_2$ satisfy \eqref{eq_gamm2finalinterval} implies that, on the circular ellipse $\{ \Gamma_1 = L_1 \}$, we have $|\cos i_{12}| < \sqrt{\frac{1}{3}} + O \left( L_2^{-1} \right)$, which means that $i_{12}$ is more than roughly $55^{\circ}$.
\end{remark}

The proof of parts 1 and 3 of Theorem \ref{theorem_poincare} is equivalent to the proof of Lemmas 33 and 36 of \cite{clarke2022why}, so we do not repeat it here. The proof of part 2 of the theorem is contained in Lemma \ref{lemma_twist} below. 

\begin{lemma}\label{lemma_twist}
	The inner map $f = F |_{\hat \Lambda}$ has the form
	\begin{equation}\label{eq_innermapform}
		f : 
		\begin{dcases}
			\left( \hat \psi_1^*, \hat \gamma_3^*, \hat \ell_3^* \right) &= \left( \hat \psi_1, \hat \gamma_3, \hat \ell_3 \right) + g  \left( \hat \Psi_1, \hat \Gamma_3, \hat L_3 \right) + O \left( \eps^{k_1 - 6} \mu^{k_2} \right)\\
			\left( \hat \Psi_1^*, \hat \Gamma_3^*, \hat L_3^* \right) &= \left( \hat \Psi_1, \hat \Gamma_3, \hat L_3 \right) + O \left( \eps^{k_1 - 6} \mu^{k_2} \right)
		\end{dcases}
	\end{equation}
	where $\eps = \frac{1}{L_2}$, $\mu = \frac{L_2}{L_3}$, where $k_1, k_2$ come from part 2 of Theorem \ref{theorem_inner}, and where
	\begin{equation}\label{eq_twistcondition}
		\det D g  \left( \hat \Psi_1, \hat \Gamma_3, \hat L_3 \right) \neq 0
	\end{equation}
	as long as 
	\begin{equation}\label{eq_gamma2twistcondition}
		\hat \Gamma_2 \notin \left\{ 0, \frac{L_1}{\sqrt{3}} \right\}. 
	\end{equation}
	Moreover the bottom eigenvalue of $D g  \left( \hat \Psi_1, \hat \Gamma_3, \hat L_3 \right)$ is of order $\frac{\mu^6}{\eps^2}$. 
\end{lemma}

\begin{proof}
	The first order term of the normalised Hamiltonian $L_2^6 \hat F$ is $c_0 \hat \Gamma_2^2$ by Theorem \ref{theorem_inner}, so the frequency of $\hat \gamma_2$ is $2 c_0 \hat \Gamma_2 + O(\eps) = O(1)$. It follows that the return time to the Poincar\'e section $\{ \hat \gamma_2 = 0\}$ of the flow of the Hamiltonian function $L_2^6 \hat F$ is of order 1. Since we have averaged the angles on the cylinder from the Hamiltonian $L_2^6 \hat F$ up to terms of order $O\left(\eps^{k_1-6} \mu^{k_2}\right)$ (see Theorem \ref{theorem_inner}), the Poincar\'e map $f$ has the form \eqref{eq_innermapform}. 
	
	Now, define 
	\[
	\omega_i \left( P_0, P \right) = \frac{\partial \hat F_0}{\partial P_i}
	\]
	for $i=0,1,2,3$ where $P_0 = \hat \Gamma_2$, $P_1 = \hat \Psi_1$, $P_2 = \hat \Gamma_3$, $P_3 = \hat L_3$, and $P = (P_1,P_2,P_3)$. Restricting to an energy level $\{ \hat F_0 = E_0 \}$ with $E_0 \in \E$, the implicit function theorem implies that we can write $P_0 = \alpha(P)$ where
	\[
	\frac{\partial \alpha}{\partial P_j} (P) = - \frac{\hat \omega_j (P)}{\hat \omega_0 (P)}
	\]
	where we have defined $\hat \omega_i (P) = \omega_i (\alpha (P),P)$. It is not hard to see that $g_i(P) = \hat \omega_0 (P)^{-1} \hat \omega_i (P)$. In the following computation, borrowed from Lemma 34 of \cite{clarke2022why}, we suppress dependence of all functions on $P$ and $P_0 = \alpha (P)$ for convenience of notation:
	\begin{align}
		\hat{\omega}_0^3 \left( D g \right)_{ij} ={}& \hat{\omega}_0^3 \frac{\partial g_i}{\partial P_j} = \hat{\omega}_0^2 \left( \frac{\partial \omega_i}{\partial P_0} \frac{\partial \alpha}{\partial P_j} + \frac{\partial \omega_i}{\partial P_j} \right) - \hat{\omega}_0 \, \hat{\omega}_i \left( \frac{\partial \omega_0}{\partial P_0} \frac{\partial \alpha}{\partial P_j} + \frac{\partial \omega_0}{\partial P_j} \right) \\
		={}& - \hat{\omega}_0 \, \hat{\omega}_j \frac{\partial \omega_i}{\partial P_0} + \hat{\omega}_0^2 \frac{\partial \omega_i}{\partial P_j} + \hat{\omega}_i \, \hat{\omega}_j \frac{\partial \omega_0}{\partial P_0} - \hat{\omega}_0 \, \hat{\omega}_i \frac{\partial \omega_0}{\partial P_j} \\
		={}& - \frac{\partial \hat{F}_0}{\partial P_0} \frac{\partial \hat{F}_0}{\partial P_j} \frac{\partial^2 \hat{F}_0}{\partial P_0 \partial P_i} + \left( \frac{\partial \hat{F}_0}{\partial P_0} \right)^2 \frac{\partial^2 \hat{F}_0}{\partial P_j \partial P_i} + \frac{\partial \hat{F}_0}{\partial P_i} \frac{\partial \hat{F}_0}{\partial P_j} \frac{\partial^2 \hat{F}_0}{\partial P_0^2} - \frac{\partial \hat{F}_0}{\partial P_0} \frac{\partial \hat{F}_0}{\partial P_i} \frac{\partial^2 \hat{F}_0}{\partial P_j \partial P_0}. 
	\end{align}
	Inserting the formulas for the derivatives of $\hat F_0$ given in Appendix \ref{appendix_derivatives} into this formula, we see that
	\begin{align}
		\hat \omega_0^3 \left( D g \right)_{11} &= \eps^{20} \hat A_{11}, & \hat \omega_0^3 \left( D g \right)_{12} &= \eps^{16} \mu^6 \hat A_{12}, & \hat \omega_0^3 \left( D g \right)_{21} &= \eps^{16} \mu^6 \hat A_{21}, \\
		\hat \omega_0^3 \left( D g \right)_{22} &= \eps^{16} \mu^6 \hat A_{22}, & \hat \omega_0^3 \left( D g \right)_{13} &= \eps^{16} \mu^3 \hat A_{13}, & \hat \omega_0^3 \left( D g \right)_{23} &= \eps^{16} \mu^7 \hat A_{23}, \\
		\hat \omega_0^3 \left( D g \right)_{31} &= \eps^{16} \mu^3 \hat A_{31}, & \omega_0^3 \left( D g \right)_{32} &= \eps^{16} \mu^7 \hat A_{32}, & \omega_0^3 \left( D g \right)_{33} &= \eps^{16} \mu^4 \hat A_{33}
	\end{align}
	where $\hat A_{ij} = O(1)$ for each $i$, $j$, and where
	\begin{align}
		\hat A_{11} ={}&  C_{12}^3\, {{54\, \left(L_{1}^2-3\,\hat{\Gamma}_{2}^2\right)\,\left(L_{1}^2+\hat{\Gamma}_{2} ^2\right)}\over{L_{1}^6\,\delta_{1}^{11}}}+ \cdots\\
		\hat A_{22} ={}&  -  C_{12}^2\,C_{23}\, {{36\, \hat{\Gamma}_{2}^2\,\left(12\,\delta_{1}^2-20\right)}\over{L_{1}^4\,\delta_{1}^8\,\delta_{2}^3}} + \cdots \\
		\hat A_{33} ={}& - C_{12}^2 \, \akep  \frac{108 \, \hat \Gamma_2^2}{L_1^4 \, \delta_1^6} + \cdots
	\end{align}
	Noticing that $\hat \omega_0^3 = O \left( \eps^{18} \right)$, we define $A_{ij} = \eps^{18} \hat \omega_0^{-3} \hat A_{ij}$. We thus obtain
	\[
	D g = 
	\left(
	\begin{matrix}
		\eps^2 A_{11} & \frac{\mu^6}{\eps^2} A_{12} & \frac{\mu^3}{\eps^2} A_{13} \\
		\frac{\mu^6}{\eps^2} A_{21} & \frac{\mu^6}{\eps^2} A_{22} & \frac{\mu^7}{\eps^2} A_{23} \\
		\frac{\mu^3}{\eps^2} A_{31} & \frac{\mu^7}{\eps^2} A_{32} & \frac{\mu^4}{\eps^2} A_{33}
	\end{matrix}
	\right). 
	\]
	We must find a condition under which the determinant of $Dg$ is nonzero, and in addition we must estimate the order of its bottom eigenvalue. We compute
	\[
	\det Dg = \frac{\mu^{10}}{\eps^2} A_{11} A_{22} A_{33} + \cdots
	\]
where we have used \eqref{eq_mainassumptiondeprit}. It is not hard to see that the first-order approximations of the three eigenvalues of $Dg$ are therefore given by the entries on the diagonal. Comparing their respective sizes using assumption \eqref{eq_mainassumptiondeprit}, we thus see that the bottom eigenvalue is indeed of order $\frac{\mu^6}{\eps^2}$. In addition we see that the nondegeneracy condition \eqref{eq_twistcondition} is satisfied as long as 
	\[
	0 \neq \hat A_{11} \hat A_{22} \hat A_{33} = 209952 \, C_{12}^6 \, C_{23} \, \akep \,  \frac{ \hat \Gamma_2^4}{L_1^{14} \, \delta_1^{25} \, \delta_2^3} \, \left(L_{1}^2-3\,\hat{\Gamma}_{2}^2\right)\,\left(L_{1}^2+\hat{\Gamma}_{2} ^2\right) \,\left(12\,\delta_{1}^2-20\right) + \cdots
	\]
	which yields \eqref{eq_gamma2twistcondition}. 
\end{proof}

\subsection{Transition chains of almost invariant tori}\label{section_transitionchains}

We define a foliation of the normally hyperbolic manifold $\hat \L$ via the leaves
\begin{equation}\label{eq_leaves}
	\LPM \left( \hat \Psi_1^*, \hat \Gamma_3^*, \hat L_3^* \right) = \left\{ \left( \hat \psi_1, \hat \gamma_3, \hat \ell_3, \hat \Psi_1, \hat \Gamma_3, \hat L_3 \right) \in \hat \L : \hat \Psi_1 = \hat \Psi_1^*, \, \hat \Gamma_3 = \hat \Gamma_3^*, \, \hat L_3 = \hat L_3^* \right\}
\end{equation}
for each $\left( \hat \Psi_1^*, \hat \Gamma_3^*, \hat L_3^* \right) \in [-1,1]^3$. Observe that each leaf $\LPM \left( \hat \Psi_1^*, \hat \Gamma_3^*, \hat L_3^* \right)$ is almost invariant under the inner map $f$, as it takes the form \eqref{eq_innermaptheorem}. A \emph{transition chain} is a sequence $\{ \LPM_j \}$ of such leaves such that for each $j$ there is $\beta_j \in \{ +, - \}$ and $z \in \LPM_{j}$ such that $\hat S_{\beta_j} (z) \in \LPM_{j+1}$ and 
\begin{equation}
	T_{\hat S_{\beta_j}(z)} \hat \L = T_{\hat S_{\beta_j}(z)} \left( \hat S_{\beta_j} \left( \LPM_j \right) \right) \oplus T_{\hat S_{\beta_j}(z)} \LPM_{j+1}. 
\end{equation}
This is an essential condition of the shadowing theorems of \cite{clarke2022topological} (see Appendix \ref{appendix_shadowing}). 

Recall that the first order terms $\S^{\pm}_1, \S^{\pm}_2, \S^{\pm}_3$ of the images of the actions of a point $\left( \hat \psi_1, \hat \gamma_3, \hat \ell_3, \hat \Psi_1, \hat \Gamma_3, \hat L_3 \right)$ under the scattering maps $\hat S_{\pm} : \hat \L \to \hat \L'$ are given by \eqref{eq_scatteringhat1}, \eqref{eq_scatteringhat2}, \eqref{eq_scatteringhat3} respectively as a result of Theorem \ref{theorem_poincare}. Let $z = \left( \hat \psi_1, \hat \gamma_3, \hat \ell_3 \right) \in \mathbb{T}^3$. In what follows we fix some values of the actions, and suppress the dependence of the functions $\S^{\pm}_1, \S^{\pm}_2, \S^{\pm}_3$ on these actions for convenience of notation. Define the matrices $A^{\pm}(z)$ by
\begin{equation}\label{eq_scatteringmatrix}
	A^{\pm}(z) = \left(
	\begin{matrix}
		\partial_{\hat \psi_1} \S^{\pm}_1 (z) & \partial_{\hat \gamma_3} \S^{\pm}_1 (z) & \partial_{\hat \ell_3} \S^{\pm}_1 (z) \\
		\partial_{\hat \psi_1} \S^{\pm}_2 (z) & \partial_{\hat \gamma_3} \S^{\pm}_2 (z) & \partial_{\hat \ell_3} \S^{\pm}_2 (z) \\
		\partial_{\hat \psi_1} \S^{\pm}_3 (z) & \partial_{\hat \gamma_3} \S^{\pm}_3 (z) & \partial_{\hat \ell_3} \S^{\pm}_3 (z) \\
	\end{matrix}
	\right). 
\end{equation}

\begin{lemma}\label{lemma_scatteringjumps}
	There are constants $\nu_j >0$ and $C>0$ for $j=1,2,3$ such that for sufficiently large $L_2 \ll L_3^{\frac{1}{3}}$ and any leaf $\LPM^* = \LPM \left( \hat \Psi_1^0, \hat \Gamma_3^0, \hat L_3^0 \right)$ of the foliation of $\hat \L$, there is $\sigma \in \{+, - \}$ and there are open sets $U_j \subset \LPM^* \simeq \mathbb{T}^3$ for $j=1,\ldots,8$ such that 
	\begin{equation}\label{eq_transvmatrixnodeg}
		\det A^{\sigma} |_{\overline{U_j}} \neq 0
	\end{equation}
	where the matrices $A^{\pm}$ are defined by \eqref{eq_scatteringmatrix}, $\mu (U_j) > C$ where $\mu$ is the Lebesgue measure on $\mathbb{T}^3$ and moreover:
	\begin{enumerate}
		\item
		For all $z \in U_1$ we have $\S^{\sigma}_1 (z) > \nu_1$, $\S^{\sigma}_2 (z) > \nu_2$, and $\S^{\sigma}_3 (z) > \nu_3$. 
		\item
		For all $z \in U_2$ we have $\S^{\sigma}_1 (z) > \nu_1$, $\S^{\sigma}_2 (z) > \nu_2$, and $\S^{\sigma}_3 (z) < - \nu_3$. 
		\item
		For all $z \in U_3$ we have $\S^{\sigma}_1 (z) > \nu_1$, $\S^{\sigma}_2 (z) < - \nu_2$, and $\S^{\sigma}_3 (z) > \nu_3$. 
		\item
		For all $z \in U_4$ we have $\S^{\sigma}_1 (z) > \nu_1$, $\S^{\sigma}_2 (z) < - \nu_2$, and $\S^{\sigma}_3 (z) < - \nu_3$. 
		\item
		For all $z \in U_5$ we have $\S^{\sigma}_1 (z) < - \nu_1$, $\S^{\sigma}_2 (z) > \nu_2$, and $\S^{\sigma}_3 (z) > \nu_3$. 
		\item
		For all $z \in U_6$ we have $\S^{\sigma}_1 (z) <- \nu_1$, $\S^{\sigma}_2 (z) > \nu_2$, and $\S^{\sigma}_3 (z) < - \nu_3$. 
		\item
		For all $z \in U_7$ we have $\S^{\sigma}_1 (z) <- \nu_1$, $\S^{\sigma}_2 (z) < - \nu_2$, and $\S^{\sigma}_3 (z) > \nu_3$. 
		\item
		For all $z \in U_8$ we have $\S^{\sigma}_1 (z) <- \nu_1$, $\S^{\sigma}_2 (z) < - \nu_2$, and $\S^{\sigma}_3 (z) < - \nu_3$. 
	\end{enumerate}
\end{lemma}

\begin{proof}
Fix any leaf $\LPM^* = \LPM \left( \hat \Psi_1^0, \hat \Gamma_3^0, \hat L_3^0 \right)$ of the foliation of $\hat \L$, and for $z = \left( \hat \psi_1, \hat \gamma_3, \hat \ell_3 \right) \in \mathbb{T}^3$, consider the map $\S^{\pm} :\mathbb{T}^3 \to \mathbb{R}^3$ given by $\S^{\pm}(z) = \left( \S^{\pm}_1(z), \S^{\pm}_2(z), \S^{\pm}_3(z) \right)$. Observe that, if we can prove that there is some $z_0 \in \mathbb{T}$ and $\sigma \in \{ +,- \}$ such that $\S^{\sigma}(z_0)=0$ and $\det A^{\sigma}(z_0) \neq 0$, then the inverse function theorem implies that $\S^{\sigma}$ surjectively maps a neighbourhood of $z_0$ in $\mathbb{T}^3$ onto a neighbourhood of the origin in $\mathbb{R}^3$, and this would complete the proof of the lemma. We search for such a $z_0$.

Consider the point $z_0 = \left( \hat \psi_1 = \frac{\pi}{2}, \hat \gamma_3 = 0, \hat \ell_3 = 0 \right)$. Notice that the true anomaly $\hat v_3=0$ whenever the mean anomaly $\hat \ell_3 = 0$. Using the formulas \eqref{eq_scatteringhat1}, \eqref{eq_scatteringhat2}, \eqref{eq_scatteringhat3}, it is not hard to see that $\S^{\pm}(z)=0$. In addition, differentiating these formulas with respect to the angles $\hat \psi_1, \hat \gamma_3, \hat \ell_3$, and using the notation
\[
C_1^{\pm} = \mp \alpha_1^{23} \kappa \left( \frac{\pi \, \hat \Gamma_2^0}{A_2 \, L_1^2} \right), \quad c = \sqrt{1 - \delta_2^2}, \quad \beta = 4 \, \delta_1^2 - 5,
\]
\[
C_2 = \frac{\alpha_0^{23} \alpha_2^{12}}{\alpha_1^{12}} \frac{L_1}{6} \sqrt{\frac{3}{2}} \, \hat \Gamma_2 \, \frac{\sqrt{1 - \frac{5}{3} \, \frac{\hat \Gamma_2^2}{L_1^2}}}{ 3 \, \frac{\hat \Gamma_2^2}{L_1^2} - 1}
\]
and the relation
\[
\frac{\partial \hat v_3}{\partial \hat \ell_3} = \delta_2^{-3} \left( 1 + c \, \cos \hat v_3 \right)^2 + \cdots
\]
we obtain (up to higher order terms)
\begin{align}
\partial_{\hat \psi_1} \S_1^{\pm}(z_0) &= - 30 \, C_2 (1 + c)^3 \, (1 - \delta_1^2), \\
\partial_{\hat \gamma_3} \S_1^{\pm}(z_0) &= C_1^{\pm} \, (1 +c)^3 \, \beta - 30 \, C_2 (1 + c)^3 \, (1 - \delta_1^2), \\
\partial_{\hat \ell_3} \S^{\pm}_1 (z_0) &= -C_1^{\pm} \, \delta_2^{-3} \, (1 +c)^5 \, \beta + C_2 \, 30 \, (1+c)^5 \, \frac{\delta_3}{\delta_1} \, (1 - \delta_1^2) \delta_2^{-3}
\end{align}
and
\begin{align}
\partial_{\hat \psi_1} \S^{\pm}_2(z_0) &= C_1^{\pm} \, (1 + c)^3 \, \beta, & \partial_{\hat \gamma_3} \S^{\pm}_2(z_0) &= -2 \, C_1^{\pm} \, (1 + c)^3 \, \frac{\delta_3}{\delta_1} \, \beta, \\
 \partial_{\hat \ell_3} \S^{\pm}_2(z_0) &= -2 \, C_1^{\pm} \, (1 + c)^5 \, \frac{\delta_3}{\delta_1} \, \delta_2^{-3} \, \beta, & \partial_{\hat \psi_1} \S^{\pm}_3(z_0) &= -C_1^{\pm} \, (1 + c)^5 \, \delta_2^{-3} \, \beta, \\
\partial_{\hat \gamma_3} \S^{\pm}_3(z_0) &= -2 \, C_1^{\pm} \, (1 + c)^5 \, \frac{\delta_3}{\delta_1} \, \delta_2^{-3} \, \beta, & \partial_{\hat \ell_3} \S^{\pm}_3(z_0) &= 2 \, C_1^{\pm} \, (1 + c)^7 \, \frac{\delta_3}{\delta_1} \, \delta_2^{-6} \, \beta. 
\end{align}
Since we only need to check that $\det A^{\pm}(z_0) \neq 0$, we divide each row and each column of $A^{\pm}(z_0) $ by some common factors to simplify the computation. After dividing: the first row by $(1 + c)^3$; the second row by $C_1^{\pm} \, (1 +c)^3 \, \beta$; the third row by $C_1^{\pm} \, (1 +c)^5 \, \delta_2^{-3} \, \beta$; the second column by $\frac{\delta_3}{\delta_1}$; and the third column by $(1+c)^2 \, \frac{\delta_3}{\delta_1}{ \, \delta_2^{-3}}$, we are left with the matrix
\[
\hat A^{\pm} = \left(
\begin{matrix}
-30 \, C_2 \, (1 - \delta_1^2) & C_1^{\pm} \, \frac{\delta_1}{\delta_3} \, \beta - 30 \, C_2 \, (1 - \delta_1^2) & -C_1^{\pm} \frac{\delta_1}{\delta_3} \beta+ 30 \, C_2 (1 - \delta_1^2) \\
1 & -2 & -2 \\
-1 & -2 & 2 \\
\end{matrix}
\right).
\]
Therefore we have $\det A^{\pm} (z_0) \neq 0$ if and only if
\begin{equation}\label{eq_transvnondegencond}
0 \neq \det \hat A^{\pm} = 4 \, C_1^{\pm} \, \frac{\delta_1}{\delta_3} \, \beta - 240 \, C_2 \, (1 - \delta_1^2). 
\end{equation}
Recall that $\delta_1 \in (0,1)$, $\delta_3 > 0$, and notice that $C_1^{\pm}, \, C_2, \, \beta \neq 0$; moreover $C_1^+$ and $C_1^-$ have equal magnitude and opposite signs. It follows that \eqref{eq_transvnondegencond} must hold for at least one value of $\sigma \in \{ +,- \}$, so the lemma is proved. 
\end{proof}

\begin{lemma}\label{lemma_transitionchains}
Suppose $z \in \LPM_0 \cap U_j$ and $z^* = \hat S_{\sigma}(z) \in \LPM_1$ where $\sigma \in \{ +,- \}$ and the open sets $U_j$ were found in Lemma \ref{lemma_scatteringjumps}, and where $\LPM_0$, $\LPM_1$ are leaves of the foliation of $\hat \L$. Then $\hat S_{\sigma}$ maps $\LPM_0$ transversely across $\LPM_1$ at the point $z^* = \hat S_{\sigma}(z)$ in the sense that 
\[
T_{z^*} \hat \L = T_{z^*} \left( \hat S_{\sigma} \left(\LPM_0 \right)\right) \oplus T_{z^*} \LPM_1. 
\]
Moreover the order of transversality (in the sense of Definition \ref{def:transvfoliation} in Appendix \ref{appendix_shadowing}) is $\frac{L_2^9}{\left( L_3^* \right)^6}$.  
\end{lemma}

\begin{proof}
Suppose $z \in U_j$ for some $j=1, \ldots, 8$. The tangent space $T_{z^*} \hat \L$ is a real vector space of dimension 6, and its elements are of the form $v = (Q,P)$ where $Q \in \mathbb{R}^3$ represents tangents in the $(\hat \psi_1, \hat \gamma_3, \hat \ell_3)$ directions, and $P \in \mathbb{R}^3$ represents tangents in the $(\hat \Psi_1, \hat \Gamma_3, \hat L_3)$ directions. The scattering map $\hat S_{\sigma}$ is smooth, and so $D_{z} \hat S_{\sigma} \left( T_{z} \LPM_0 \right) = T_{z^*} \hat S_{\sigma} \left( \LPM_0 \right)$. With $z = \left( \hat \psi_1, \hat \gamma_3, \hat \ell_3, \hat \Psi_1, \hat \Gamma_3, \hat L_3 \right)$ and $z^* = \left( \hat \psi_1^*, \hat \gamma_3^*, \hat \ell_3^*, \hat \Psi_1^*, \hat \Gamma_3^*, \hat L_3^* \right)$ we can write 
\[
D_z \hat S_{\sigma} = \left(
\begin{matrix}
A_1 & A_2 \\
A_3 & A_4 \\
\end{matrix}
\right)
\]
where
\[
A_1 = \frac{\partial \left( \hat \psi_1^*, \hat \gamma_3^*, \hat \ell_3^* \right)}{\partial \left( \hat \psi_1, \hat \gamma_3, \hat \ell_3 \right)}, \quad A_2 = \frac{\partial \left( \hat \psi_1^*, \hat \gamma_3^*, \hat \ell_3^* \right)}{\partial \left( \hat \Psi_1, \hat \Gamma_3, \hat L_3 \right)}, \quad A_3 = \frac{\partial \left( \hat \Psi_1^*, \hat \Gamma_3^*, \hat L_3^* \right)}{\partial \left( \hat \psi_1, \hat \gamma_3, \hat \ell_3 \right)}, \quad A_4 = \frac{\partial \left( \hat \Psi_1^*, \hat \Gamma_3^*, \hat L_3^* \right)}{\partial \left( \hat \Psi_1, \hat \Gamma_3, \hat L_3 \right)}. 
\]
Therefore
\[
A_1 = \left( 
\begin{matrix}
1 & 0 & 0 \\
0 & 1 & 0 \\
0 & 0 & 1 \\
\end{matrix}
\right) + \cdots, \quad A_3 = A^{\sigma} \left( \hat \psi_1, \hat \gamma_3, \hat \ell_3 \right) + \cdots
\]
where the matrices $A^{\pm}$ are defined by \eqref{eq_scatteringmatrix}. 

Let $v_0 \in T_{z} \LPM_0$ and $v_1 \in T_{z^*} \LPM_1$. Since the leaves are defined as tori with constant actions, we have $v_j = (Q_j,0)$ with $Q_j \in \mathbb{R}^3$ for $j=0,1$. Therefore we have
\[
\left(
\begin{matrix}
\bar Q \\
\bar P
\end{matrix}
\right) = 
D_z \hat S_{\sigma} (v_0) + v_1 = \left(
\begin{matrix}
Q_0 + Q_1 \\
A^{\sigma} \left( \hat \psi_1, \hat \gamma_3, \hat \ell_3 \right) Q_0
\end{matrix}
\right) + \cdots
\]  
By Lemma \ref{lemma_scatteringjumps}, the matrix $A^{\sigma} \left( \hat \psi_1, \hat \gamma_3, \hat \ell_3 \right)$ is nonsingular. Thus by varying $v_0 \in T_z \LPM_0$ and $v_1 \in T_{z^*} \LPM_1$ we can obtain any tangent vector in $T_{z^*} \hat \L$, which is precisely the transversality we want to prove. The fact that the order of transversality is $\frac{L_2^9}{\left( L_3^* \right)^6}$ follows from the fact that the order of the jumps in the scattering maps in the direction of each of the actions is $\frac{L_2^9}{\left( L_3^* \right)^6}$. 
\end{proof}

\subsection{Orbits of the four-body problem shadowing the transition chains}\label{section_shadowingsubsec}

In this section we apply the shadowing theorems of \cite{clarke2022topological} (see also Appendix \ref{appendix_shadowing}) to the secular Hamiltonian and to the Hamiltonian of the full four-body problem in order to complete the proof of Theorem \ref{thm:MainHierarch:Deprit}. In fact, our first observation is that we have already proved that the Poincar\'e map $F:M \to M$ constructed in Theorem \ref{theorem_poincare} satisfies assumptions [A1-3] of Theorem \ref{theorem_main1}. Indeed, Theorem \ref{theorem_poincare} implies that $F$ has a normally hyperbolic invariant manifold $\hat \L \simeq \mathbb{T}^3 \times [0,1]^3$, and that the inner map $f = F |_{\hat \L}$ is a near-integrable twist map satisfying a non-uniform twist condition of order $\frac{L_2^8}{\left( L_3^* \right)^6}$ as per Definitions \ref{def_nearlyintegrabletwist} and \ref{def_nonuniformtwist}. In addition the stable and unstable manifolds of $\hat \L$ have a transverse homoclinic intersection along (at least) two homoclinic channels, and the order of the transversality of the invariant manifolds is $\frac{1}{L_2^2}$. These homoclinic channels give rise to two scattering maps $\hat S: \hat \L \to \hat \L'$ due to Theorems \ref{theorem_outer} and \ref{theorem_poincare}. In Section \ref{section_transitionchains} we constructed a foliation of $\hat \L'$, the leaves of which are given by constant values of the actions, and are almost invariant under the inner map. In addition Lemmas \ref{lemma_scatteringjumps} and \ref{lemma_transitionchains} imply that, by iterating one of the scattering maps in appropriate neighbourhoods in the leaves of the foliation, we may move in any direction, connecting leaves that are $O \left( \frac{L_2^9}{\left( L_3^* \right)^6} \right)$ apart, and mapping leaves transversely across leaves. 

In Section \ref{sec_seexp} (see Remark \ref{remark_coordtransftilde} in particular) we made a symplectic change of coordinates \eqref{eq_changeofcoordstilde} to the `tilde' variables, which were the basis for all further analysis in the paper. However this coordinate transformation is local, whereas the drift in eccentricity, inclination, and the semimajor axis described in Theorem \ref{thm:main} is global. In order to define these coordinates  we introduced constants $L_3^*$ and $\delta_j$ for $j=1,2,3$. Here $\delta_2$ is the coefficient of total angular momentum, and is therefore fixed for the secular system. The constants $L_3^*$, $\delta_1, \delta_3$ on the other hand are allowed to vary, and by varying them we simply obtain a different system of `tilde' coordinates. It is not hard to see that the subsequent analysis of this paper holds equally for any value of the constants $L_3^* \gg L_2^3$, $\delta_1 \in (0,1)$, and $\delta_3 \in (-1,1)$. Denote by $\tilde{C}_{L_3^*,\delta_1, \delta_3}$ the system of `tilde' coordinates corresponding to the values $L_3^*, \delta_1, \delta_3$. Then the results of Section \ref{section_analysisofh0} apply in $\tilde{C}_{L_3^*, \delta_1, \delta_3}$ coordinates for each relevant value of $L_3^*, \delta_1, \delta_3$, so we have a normally hyperbolic invariant manifold $\Lambda_{L_3^*, \delta_1, \delta_3}$ in each such system of coordinates. Moreover, since the cylinder depends smoothly on the parameters $L_3^*, \delta_1, \delta_3$, this construction allows us to obtain one large normally hyperbolic invariant cylinder $\Lambda^*$. 

Observe that the contents of Sections \ref{subsec_poincaremap} and \ref{section_transitionchains} apply equally in each system of coordinates $\tilde{C}_{L_3^*, \delta_1, \delta_3}$. Furthermore, since the $\hat \gamma_2$ variable does not depend on $L_3^*, \delta_1, \delta_3$, the Poincar\'e section is global, and so we obtain a large 6-dimensional cylinder $\hat \Lambda^*$ for the return map to the Poincar\'e section. We now fix a global transition chain on the cylinder $\hat \Lambda^*$ such that the actions $\hat \Psi_1, \hat \Gamma_3, \hat L_3$ drift by an amount of order $L_3^*$ along the chain (note that the values of $L_3$ that we choose belong to some bounded set, as in \eqref{eq_l3driftrange}), and choose some sequence $\{ \tilde{C}_{L_3^{*,k}, \delta_1^k, \delta_3^k} \}_{k =1, \ldots, K}$ so that we have an appropriate system of coordinates to apply the analysis of the earlier sections near each torus in the chain. The analysis of Sections \ref{subsec_poincaremap} and \ref{section_transitionchains} applies in each coordinate system $\tilde{C}_{L_3^{*,k}, \delta_1^k, \delta_3^k}$. Note that the shadowing results of \cite{clarke2022topological} apply equally well using the many different coordinate systems, as the coordinates used in that paper are purely local. Thus the assumptions of Theorem \ref{theorem_main1} apply to the global transition chain on the cylinder $\hat \Lambda^*$. 

Denoting by $\{ \LPM_i \}$ the global transition chain, Theorem \ref{theorem_main1} implies that, for any $\eta >0$ and sufficiently large $L_2, L_3$, there is a sequence $\{ z_j \}_{j \in \mathbb{N}_0}$ in the secular phase space and times $t_j >0$ such that 
\begin{equation}
z_{j+1} = \phi_{\mathrm{sec}}^{t_j} (z_j), \quad d(z_j, \LPM_j) < \eta
\end{equation}
for each $j \in \mathbb{N}$ where $\phi_{\mathrm{sec}}^{t}$ is the flow of the secular Hamiltonian. Moreover, the time to move a distance of order $L_3^{*,1}$ in the $L_3$ variable and a distance of order $L_2$ in the $\hat \Psi_1$, $\hat \Gamma_3$ directions is of order
\begin{equation}\label{eq_seculararnoldtimeestimate}
L_3 \, L_2^6 \, \frac{L_3^{12}}{L_2^{18}} \, \frac{L_3^6}{L_2^8} \, \frac{L_3^6}{L_2^9} = \frac{L_3^{25}}{L_2^{29}}.
\end{equation}
This follows from formula \eqref{eq_timeestimate1} and the following facts: we move a distance of order $L_3$ in the actions $\hat \Psi_1$, $\hat \Gamma_3$, $\hat L_3$; the return times to the Poincar\'e section are themselves of order $L_2^6$ as this is the reciprocal of the order of the frequency of $\hat \gamma_2$; the order of the splitting is $\frac{1}{L_2^2}$; the size of the jumps in the scattering map are of order $\frac{L_2^9}{L_3^6}$; and the order of the twist condition is $\frac{L_2^8}{L_3^6}$. 

Now, consider the Hamiltonian $F$ of the full four-body problem after averaging the mean anomalies $\ell_1, \ell_2$, defined by \eqref{eq_4bphamave}. Fix some $L^{\pm}_1, \, L^{\pm}_2 \in \mathbb{R}$ with $0 < L^-_1 < L^+_1 \ll L^-_2 < L^+_2$, so that if $(L_1^0,L_2^0) \in [L_1^-,L_1^+] \times [L_2^-,L_2^+]$ then there exist initial conditions $L_3^0 \in \mathbb{R}$ such that $(L_1^0,L_2^0,L_3^0)$ satisfies our assumption \eqref{eq_mainassumptiondeprit}. Let $\Sigma = \mathbb{T}^2 \times [L_1^-,L_1^+] \times [L_2^-,L_2^+]$, and recall from the beginning of Section \ref{subsec_poincaremap} the definition of the subset $\mathcal{D}$ of the secular phase space. We consider the values of energy of the full four-body problem belonging to the set $\mathcal{E}_{\mathrm{4bp}} = \left\{ F (z, \ell_1, \ell_2, L_1, L_2) : z \in \mathcal{D}, \, (\ell_1, \ell_2, L_1, L_2) \in \Sigma \right\}$. Fix $E_1 \in \mathcal{E}_{\mathrm{4bp}}$, and denote by $\Psi$ the Poincar\'e map of the flow of $F$ to the section $\widehat M = \left( \mathcal{D} \times \Sigma \right) \cap \{ \hat \gamma_2 = 0 \} \cap \{ F = E_1 \}$. Continuing to denote by $z$ a point in $\mathcal{D}$, we write $( \bar z, \bar \ell_1, \bar \ell_2, \bar L_1, \bar L_2) = \Psi (z, \ell_1, \ell_2, L_1, L_2)$ with $\bar z = G(z, \ell_1, \ell_2, L_1, L_2)$ and $( \bar \ell_1, \bar \ell_2, \bar L_1, \bar L_2) = \phi (z, \ell_1, \ell_2, L_1, L_2)$. Since the Hamiltonian $F$ is obtained by averaging the mean anomalies $\ell_1, \ell_2$, there are $\hat k_1, \hat k_2 \in \mathbb{N}$ such that the variables $\ell_1, \ell_2$ do not appear in $F$ until terms of order $\eps^{\hat k_1} \mu^{\hat k_2}$ where $\eps = \frac{1}{L_2}$ and $\mu = \frac{L_2}{L_3}$, and moreover we can choose $\hat k_1, \hat k_2$ to be as large as we like. Therefore the map $G$ takes the form $G(z, \ell_1, \ell_2, L_1, L_2) = \widetilde G (z; L_1, L_2) + O \left( \eps^{\hat k_1-6} \mu^{\hat k_2} \right)$ where the higher-order terms are uniformly bounded in the $C^r$ topology for any $r \in \mathbb{N}$, and where for any fixed values of $L_1, L_2$, the map $z \mapsto \widetilde G (z; L_1, L_2)$ is a Poincar\'e map of the type constructed in Theorem \ref{theorem_poincare}, and thus satisfies the assumptions [A1-3] of Theorem \ref{theorem_main1}. Consequently the map $\Psi$ satisfies assumption [B1] of Theorem \ref{theorem_main2}. In addition, if we write $\phi (z, \ell_1, \ell_2, L_1, L_2) = \left( \phi_1 (z, \ell_1, \ell_2, L_1, L_2), \phi_2 (z, \ell_1, \ell_2, L_1, L_2) \right)$ so that $\left( \bar \ell_1, \bar \ell_2 \right) = \phi_1 (z, \ell_1, \ell_2, L_1, L_2)$ and $\left( \bar L_1, \bar L_2 \right) = \phi_2 (z, \ell_1, \ell_2, L_1, L_2)$ then we have $\phi_2 (z, \ell_1, \ell_2, L_1, L_2) = \left( L_1, L_2 \right) + O \left( \eps^{\hat k_1-6} \mu^{\hat k_2} \right)$ with the higher-order terms uniformly $C^r$-bounded, and so assumption [B2] of Theorem \ref{theorem_main2} is also satisfied. As explained in Appendix \ref{appendix_shadowing}, a consequence of results of \cite{delshams2008geometric} is that the map $\Psi$ has a normally hyperbolic locally invariant manifold $\widetilde \Lambda$ that is close to $\hat \Lambda \times \Sigma$. Therefore we can use the variables $(w, \ell_1, \ell_2, L_1, L_2)$ as coordinates on $\widetilde \Lambda$ where $w$ are coordinates on $\hat \Lambda$, and construct a foliation of $\widetilde \Lambda$ by the leaves
\[
\tilde \LPM \left( \hat \Psi_1^*, \hat \Gamma_3^*, \hat L_1^*, \hat L_2^*, \hat L_3^* \right) = \left\{ \left( w, \ell_1, \ell_2, L_1, L_2 \right): w \in \LPM \left( \hat \Psi_1^*, \hat \Gamma_3^*, \hat L_3^* \right), \, L_1 = \hat L_1^*, \, L_2 =  \hat L_2^* \right\}
\]
where $\LPM \left( \hat \Psi_1^*, \hat \Gamma_3^*, \hat L_3^* \right)$ is the leaf of the foliation of $\hat \Lambda$ defined by \eqref{eq_leaves}. Fix $\eta > 0$ and $K_1, K_2 \in \mathbb{N}$, and choose some initial values $L_1^1, L_2^1$ of the variables $L_1, L_2$ so that $(\ell_1, \ell_2, L_1^1, L_2^1) \in \mathrm{Int} \, \Sigma$ for any $(\ell_1, \ell_2) \in \mathbb{T}^2$. Choose $N \leq \eps^{-K_1} \mu^{-K_2}$, and values $P_*^1, \ldots, P_*^N$ of the actions $\hat \Psi_1, \hat \Gamma_3, \hat L_3$ such that the leaves $\LPM \left( P_*^j \right)$ of the foliation of $\hat \Lambda$ are connected by one of the scattering maps of the secular Hamiltonian with $L_1 = L_1^1$, $L_2 = L_2^1$ in the sense of Lemma \ref{lemma_scatteringjumps}. Then by Theorem \ref{theorem_main2}, there are $(L_1^2, L_2^2), \ldots, (L_1^N, L_2^N) \in \left[ L_1^-, L_1^+ \right] \times \left[ L_2^-, L_2^+ \right]$ such that, with $\tilde \LPM_j = \tilde \LPM \left( P_*^j, L_1^j, L_2^j\right)$, there are points $(z^1, \ell_1^1, \ell_2^1, L_1^1, L_2^1), \ldots, (z^N, \ell_1^N, \ell_2^N, L_1^N, L_2^N)$ in the phase space of the full four-body problem and times $t_j^* > 0$ such that 
\[
(z^{i+1}, \ell_1^{i+1}, \ell_2^{i+1}, L_1^{i+1}, L_2^{i+1}) = \phi^{t_j^*}_{\mathrm{4bp}} (z^{i}, \ell_1^{i}, \ell_2^{i}, L_1^{i}, L_2^{i}), \quad d \left( (z^{i}, \ell_1^{i}, \ell_2^{i}, L_1^{i}, L_2^{i}), \tilde \LPM_i \right) < \eta
\]
where $\phi^{t}_{\mathrm{4bp}}$ is the flow of the full four-body problem. Moreover the time estimate \eqref{eq_seculararnoldtimeestimate} also holds in this case as the order of the time required to move a distance of order $L_3$ in the $L_3$ variable and a distance of order $L_2$ in the $\hat \Psi_1$, $\hat \Gamma_3$ directions. This completes the proof of Theorem \ref{thm:MainHierarch:Deprit}.

\section{Proof of Theorems \ref{thm:main} and \ref{thm:main:planetary:Deprit}}

In Section \ref{sec:shadowing} the proof of Theorem \ref{thm:MainHierarch:Deprit} was completed; the purpose of this section is to show that the analysis of Sections \ref{sec_secularexpansion} - \ref{sec:shadowing} extends to a complete proof of Theorems \ref{thm:main} and \ref{thm:main:planetary:Deprit}. 

\subsection{The planetary regime: Proof of Theorem \ref{thm:main:planetary:Deprit}}
\label{sec:planetary}
We consider now the planetary regime where the masses of bodies 1,2, and 3 are small with respect to body 0. However, in order to make this work we will see that the semimajor axis $a_3$ depends on the small mass parameter.  


Up to now we have investigated what we have called the strongly hierarchical regime, where the semimajor axes satisfy
\begin{equation} \label{eq_assumption1bis}
O(1) = a_1 \ll a_2  \to \infty \qquad \text{and}\qquad
a_2^{3}  \ll a_3 
\end{equation} 
('strong' meaning that  the ratio between $a_2$ and $a_3$ has to be rather large). 

Now we assume that \begin{equation}
  \label{eq:planetary}
  m_1, m_2,m_3 \sim \rr\to 0.
\end{equation}
That is, three masses are small and of comparable size. We scale the Deprit actions via
\begin{equation}
L = \rho \check L, \quad \Gamma = \rho \,  \check \Gamma, \quad \Psi = \rho \, \check \Psi
\end{equation} 
where $L = (L_1, L_2, L_3)$, $\Gamma = (\Gamma_1, \Gamma_2, \Gamma_3)$, and $\Psi = (\Psi_1, \Psi_2)$. 

\begin{proposition}
  The instability mechanism which we have shown to exist in the hierarchical regime in Sections~\ref{sec_secularexpansion}-\ref{sec:shadowing} continues in the planetary regime (i.e.  as $\rho$ tends to $0$) as long as the scaled total angular momentum $\check{\Psi}_2$ satisfies
  \[
  \check{\Psi}_2\gtrsim \rr^{-1/3}.
  \]
  Moreover, the instability time is of the order of $\rho^{-31/3}$.
\end{proposition}

\begin{proof}
  Write $m_j = \rho \bar m_j$, $j=1,2,3$,   so that, when $\rho\ll 1$, 
  \[M_j \sim 1, \quad \sigma_{0,j} \sim 1, \quad \sigma_{ij} \sim \rho \quad \mbox{and} \quad \mu_j \sim \rho \quad (i,j=1,2,3).\]
Consider the Keplerian Hamiltonian of the third planet (see \eqref{eq_fkepdeprit}),
\[
F_{\mathrm{Kep},3} = -  \frac{\mu_3^3 \, M_3^2}{2 \, L_3^2}.
\]
If we express it in term of the scaled actions \eqref{def:planetaryscaling}, we obtain
\begin{equation}\label{eq_ell3dotscalingplanetary}
F_{\mathrm{Kep},3}\sim   \frac{\rho}{\check{L}_3^2}.
\qquad\text{and}\qquad\dot \ell_3\sim   \frac{1}{\check{L}_3^3}
\end{equation}
Proceeding analogously for the perturbing function, one can see that it scales differently in $\rr$. Indeed, consider for instance the part of the perturbing function $F_{\mathrm{per}}^{12}$ regarding planets $1$ and $2$ (first introduced in equation~\eqref{eq_perfn12}), written in terms of Legendre polynomials. that is 
  \[F_{\mathrm{per}}^{12} = - \frac{\mu_1 m_2}{\| q_2 \|} \sum_{n=2}^{\infty} \tilde{\sigma}_{1,n} P_n (\cos \zeta_{1}) \left( \frac{\| q_1 \|}{\| q_2 \|} \right)^n\sim \rr^2\frac{\check{L}_1^4}{\check{L}_2^6}\]
  where
  \[\tilde{\sigma}_{1,n} = \sigma_{01}^{n-1} + (-1)^n \sigma_{11}^{n-1} \sim 1.\]
Note that if one considers $\|q_1\|$, $\|q_2\|$ independent of $\rr$, this Hamiltonian has size of order $\rr^2$. One obtains the same behavior for  $F_{\mathrm{per}}^{23}$. 

Now, for $\check{L}_3$ independent of $\rr$: if one lets $\rr\to 0$, the frequency of $\dot\ell_3$ becomes much faster than the other secular variables. This alters the hierarchy of time scales considered in the proof of Theorem \ref{thm:MainHierarch:Deprit}.

To keep the hierarchy it is enough to choose $\check{\Psi}_2\sim \check{L}_3$ large enough depending on $\rr$. Indeed, note that in Theorem \ref{thm:MainHierarch:Deprit}, the only condition on $\Psi_2$ is a lower bound and therefore one can take 
$1\sim \check{L}_1^0\ll \check{L}_2^0$ independent of $\rr$ and $\check{\Psi}_2\gtrsim \rr^{-1/3}$, where the exponent of $- \frac{1}{3}$ comes from \eqref{eq_ell3dotscalingplanetary}. Scaling time by a factor $\rr^{2}$ and applying the change of coordinates \eqref{eq_changeofcoordstilde}, one obtains a secular Hamiltonian whose associated Keplerian (for planet 3), quadrupolar, octupolar Hamiltonians have first orders independent of $\rr$. Then, the proof of Theorem  \ref{thm:MainHierarch:Deprit} applies. For this scaled model, the estimate \eqref{eq_seculararnoldtimeestimate} implies the innstability time is 
\[
\check{T}\lesssim \frac{\check{L}_3^{25}}{\check{L}_2^{29}}\lesssim \check{L}_3^{25} \rr^{-29/3}.
\]
Scaling back time by the factor $\rr^{2}$, one obtains
\[
T\sim \rr^{-35/3}. 
\]
\end{proof}

\subsection{Proof of Theorem \ref{thm:main}}

Theorem \ref{thm:main} and the comments stated after it are a direct consequence of the proof of Theorems \ref{thm:MainHierarch:Deprit} and \ref{thm:main:planetary:Deprit}. Indeed, to prove Theorem \ref{thm:main} it is enough to deduce the evolution of the orbital elements and the normalised angular momentum vector from that of the Deprit variables in Theorem \ref{thm:MainHierarch:Deprit} and their definitions in Section \ref{sec:coordinates} (see, in particular \eqref{def:eccentricity}, \eqref{def:inclinationi12} and \eqref{def:inclinationi23}). The normalised angular momentum vector $\tilde C_2$ is determined by the eccentricity $e_2$, the mutual inclination $\theta_{23}$ of bodies 2 and 3, and the longitude $h_2$ of the node of planet 2. It can easily be checked that Theorems \ref{thm:MainHierarch:Deprit} and \ref{thm:main:planetary:Deprit} imply that we can vary the eccentricity $e_2$, and the mutual inclination $\theta_{23}$ in any way we like up to some small error term by varying $\Gamma_2$, $\Gamma_3$ (see \eqref{def:eccentricity} and \eqref{def:inclinationi23}). The angle $h_2$ is determined by angles on the normally hyperbolic cylinder. Although not stated explicitly in Theorems \ref{thm:MainHierarch:Deprit} and \ref{thm:main:planetary:Deprit} or in Section \ref{sec:shadowing}, it is clear that the shadowing methods of \cite{clarke2022topological} also allow us to shadow values of the angles on the cylinder by time shifts of the flow, and so we can control $h_2$ in the same way that we control the actions. See Section 9 of \cite{clarke2022why} for an elaboration of this discussion. The time estimates in \eqref{def:timefinal} are also provided by Theorem \ref{thm:MainHierarch:Deprit}. Finally the analysis of the planetary regime and the time estimates in \eqref{def:timefinal2} are a consequence of Theorem \ref{thm:main:planetary:Deprit}.




\appendix

\section{A refinement of Moser's trick}\label{appendix_moser}

In Section \ref{section_inner} (see Lemma \ref{lemma_straightsymp}), we prove that there is a near-identity coordinate transformation that straightens the restriction to the normally hyperbolic invariant manifold of the symplectic form. If we were to proceed with this proof, for example, as in Lemma 23 of \cite{clarke2022why}, we would see that the coordinate transformation in the action $\tilde \Psi_1$ would dominate the transformation coming from the averaging (i.e. Lemma \ref{lemma_inneraveraging}). This would be problematic, as the computation of the coordinate transformation that straightens the symplectic form is significantly more complicated than that of the coordinate transformation that averages the inner angles. Instead, we notice that the first two terms in the Taylor expansion of the parametrisation of the normally hyperbolic manifold do not depend on the angle $\tilde \psi_1$ (see Lemma \ref{lemma_graphexpansion}); from a symplectic point of view, it seems natural therefore that the coordinate transformation straightening the symplectic form should not alter its symplectic conjugate $\tilde \Psi_1$. However, this coordinate transformation (provided by Moser's trick from his proof of Darboux's theorem) is obviously not symplectic, and so it is not clear a priori that this is the case. Below we provide a proof of this fact. 

Denote by $M$ a symplectic manifold of dimension 6, and by $(q_0, q_1, q_2, p_0, p_1, p_2)$ coordinates on $M$ such that the symplectic form is 
\[
\Omega = d \lambda = \sum_{i=0}^2 d q_i \wedge d p_i \quad \textrm{where} \quad \lambda = -\sum_{i=0}^2 p_i \, d q_i.
\]
Write $q=(q_1,q_2)$, $p=(p_1,p_2)$, and suppose $\Lambda \subset M$ is a submanifold of $M$ that can be written as the graph of a function $\rho = (\rho_q, \rho_p)$ in the sense that
\[
\Lambda = \left\{ (q_0,q,p_0,p) \in M : q_0 = \rho_q (q,p), \, p_0 = \rho_p (q,p), \, (q,p) \in U \right\}
\]
for some domain $U$. Suppose moreover that $\rho = O(\eps)$ does not depend on $q_2$; this can be expressed as $\rho = \eps \, \rho' (q_1,p)$, where $\rho' = (\rho_q', \rho_p')$. Let $\hat \eps = \eps^2$, and write
\[
\Omega' = \left. \Omega \right|_{\Lambda} = \hat \eps \, d \rho_q' \wedge d \rho_p' + \Omega_0, \quad \lambda' = \left. \lambda \right|_{\Lambda} = - \hat \eps \, \rho_p' \, d \rho_q' + \lambda_0
\]
where
\[
\Omega_0 = \sum_{i=1}^2 d q_i \wedge d p_i, \quad \lambda_0 = -\sum_{i=1}^2 p_i \, d q_i . 
\]
The goal is to find a coordinate transformation $h$ on $\L$ that is $O(\hat \eps)$-close to the identity such that $h^* \Omega' = \Omega_0$. Moser's trick (see the proof of Lemma \ref{lemma_straightsymp} for the precise construction) is to construct $h$ as the time-$\hat \eps$ map of a nonautonomous vector field $X_t$ satisfying
\begin{equation}\label{eq_appmoser}
	i_{X_{\hat \eps}} \Omega' = - \frac{d}{d \hat \eps} \lambda'. 
\end{equation}

\begin{lemma}
	Denote by $h$ the time-$\hat \eps$ map of the vector field $X_t$ defined by \eqref{eq_appmoser}. Then $h$ has the form
	\[
	h : 
	\begin{dcases}
		\bar q_1 ={}& q_1 + \hat \eps \, f_1 (q_1,p) \\ 
		\bar q_2 ={}& q_2 + \hat \eps \, f_2 (q_1,p) \\
		\bar p_1 ={}& p_1 + \hat \eps \, f_3 (q_1,p) \\
		\bar p_2 ={}& p_2.  \\
	\end{dcases}
	\]
\end{lemma}

\begin{proof}
	Notice we can write
	\[
	d \rho_q' \wedge d \rho_p' = A_{13} \, d q_1 \wedge dp_1 + A_{14} \, dq_1 \wedge dp_2 + A_{34} \, d p_1 \wedge dp_2
	\]
	where
	\[
	\begin{dcases}
		A_{13} ={}& \partial_{q_1} \rho_q' \, \partial_{p_1} \rho_p' -  \partial_{q_1} \rho_p' \, \partial_{p_1} \rho_q' \\
		A_{14} ={}& \partial_{q_1} \rho_q' \, \partial_{p_2} \rho_p' -  \partial_{q_1} \rho_p' \, \partial_{p_2} \rho_q' \\
		A_{34} ={}& \partial_{p_1} \rho_q' \, \partial_{p_2} \rho_p' -  \partial_{p_1} \rho_p' \, \partial_{p_2} \rho_q'. \\
	\end{dcases}
	\]
	Writing $\Omega'$ in matrix form, we have 
	\begin{equation}\label{eq_sympmatrixmoser1}
		\Omega' = \Omega_0 + \hat \eps \, A \quad \textrm{where} \quad \Omega_0 = 
		\left(
		\begin{matrix}
			0 & I \\
			-I & 0 \\
		\end{matrix}
		\right), \quad
		I = 
		\left(
		\begin{matrix}
			1 & 0 \\
			0 & 1 \\
		\end{matrix}
		\right)
	\end{equation}
	and
	\begin{equation}\label{eq_sympmatrixmoser2}
		A = 
		\left(
		\begin{matrix}
			0 & 0 & A_{13} & A_{14} \\
			0 & 0 & 0 & 0 \\
			-A_{13} & 0 & 0 & A_{34} \\
			- A_{14} & 0 & - A_{34} & 0 \\
		\end{matrix}
		\right). 
	\end{equation}
	On the other hand we can write
	\[
	\rho_p' \, d \rho_q' = a_1 \, dq_1 + a_3 \, dp_1 + a_4 \, dp_2
	\]
	where
	\[
	\begin{dcases}
		a_1 ={}& \rho_p' \, \partial_{q_1} \rho_q' \\
		a_3 ={}& \rho_p' \, \partial_{p_1} \rho_q' \\
		a_4 ={}& \rho_p' \, \partial_{p_2} \rho_q'
	\end{dcases}
	\]
	so we can write $\lambda'$ in vector form as
	\begin{equation}\label{eq_moserliouvillevector}
		\lambda' = \lambda_0 + \hat \eps \, a \quad \textrm{where} \quad \lambda_0 = - 
		\left(
		\begin{matrix}
			p_1 \\
			p_2 \\
			0 \\
			0 \\
		\end{matrix}
		\right), 
		\quad a = -
		\left(
		\begin{matrix}
			a_1 \\
			0 \\
			a_3 \\
			a_4 \\
		\end{matrix}
		\right). 
	\end{equation}
Now, let $X_{\hat \eps} = \left(\dot q, \dot p \right)$. Then from \eqref{eq_sympmatrixmoser1} and \eqref{eq_sympmatrixmoser2}, we see that the left-hand side of \eqref{eq_appmoser} is 
	\begin{equation}\label{eq_moserlhs}
		i_{X_{\hat \eps}} \Omega' = X_{\hat \eps}^T \, \Omega_0 + \hat \eps \, X_{\hat \eps}^T \, A = \left(
		\begin{matrix}
			- \dot p_1 \\
			- \dot p_2 \\
			\dot q_1 \\
			\dot q_2 \\
		\end{matrix}
		\right) + \hat \eps \, \left(
		\begin{matrix}
			- A_{13} \, \dot p_1 - A_{14} \, \dot p_2 \\
			0 \\
			A_{13} \, \dot q_1 - A_{34} \, \dot p_2 \\
			A_{14} \, \dot q_1 + A_{34} \, \dot p_1 \\
		\end{matrix}
		\right). 
	\end{equation}
	Meanwhile \eqref{eq_moserliouvillevector} implies that the right-hand side of \eqref{eq_appmoser} is
	\begin{equation}\label{eq_moserrhs}
		- \frac{d}{d \hat \eps} \lambda' = 
		\left(
		\begin{matrix}
			a_1+\hat \eps\partial_{\hat \eps}a_1 \\
			0 \\
			a_3+\hat \eps\partial_{\hat \eps}a_3 \\
			a_4+\hat \eps\partial_{\hat \eps}a_4 \\
		\end{matrix}
		\right).
	\end{equation}
	Combining \eqref{eq_moserlhs} and \eqref{eq_moserrhs} yields
	\[
	\left(
	\begin{matrix}
		\dot q_1 \\
		\dot q_2 \\
		\dot p_1 \\
		\dot p_2 \\
	\end{matrix}
	\right) = \left(
	\begin{matrix}
		a_3 +\hat \eps\partial_{\hat \eps}a_3- \hat \eps \, \left( A_{13} \, \dot q_1 - A_{34} \, \dot p_2 \right) \\
		a_4 +\hat \eps\partial_{\hat \eps}a_4- \hat \eps \, \left( A_{14} \, \dot q_1 + A_{34} \, \dot p_1 \right) \\
		-a_1 -\hat \eps\partial_{\hat \eps}a_1+ \hat \eps \,  \left( A_{13} \, \dot p_1 + A_{14} \, \dot p_2 \right) \\
		0 \\
	\end{matrix}
	\right)
	\]
	which completes the proof of the lemma. 
\end{proof}

\section{The scattering map of a normally hyperbolic invariant manifold}\label{app:scattering}

In this section we denote by $M$ a $C^r$ smooth manifold, and by $\phi^t : M \to M$ a smooth flow with $\left. \frac{d}{dt} \right|_{t=0} \phi^t = X$ where $X \in C^r (M, TM)$. Let $\Lambda \subset M$ be a compact $\phi^t$-invariant submanifold, possibly with boundary. By $\phi^t$-invariant we mean that $X$ is tangent to $\Lambda$, but that orbits can escape through the boundary (a concept sometimes referred to as \emph{local} invariance). 

\begin{definition}
We call $\Lambda$ a \emph{normally hyperbolic invariant manifold} for $\phi^t$ if there is $0 < \lambda < \mu^{-1}$, a positive constant $C$ and an invariant splitting of the tangent bundle
\begin{equation}
T_{\Lambda} M = T \Lambda \oplus E^s \oplus E^u
\end{equation}
such that:
\begin{equation} \label{eq_normalhyperbolicity}
\def\arraystretch{1.5}
\begin{array}{c}
\| D \phi^t |_{E^s} \| \leq C \lambda^t \mbox{ for all } t \geq 0, \\
\| D \phi^t |_{E^{u}} \| \leq C \lambda^{-t} \mbox{ for all } t \leq 0, \\
\| D \phi^t |_{T \Lambda} \| \leq C \mu^{| t |} \mbox{ for all } t \in \mathbb{R}.
\end{array}
\end{equation}
Moreover, $\Lambda$ is called an $r$-\emph{normally hyperbolic invariant manifold} if it is $C^r$ smooth, and
\begin{equation}\label{eq_largespectralgap}
0 < \lambda < \mu^{-r} < 1
\end{equation}
 for $r \geq 1$. This is called a \emph{large spectral gap} condition.
\end{definition}

This definition guarantees the existence of stable and unstable invariant manifolds $W^{s,u} (\Lambda)\subset M$ defined as follows. The local stable manifold $W^{s}_{\mathrm{loc}}(\Lambda)$ is the set of points in a small neighbourhood of $\Lambda$ whose forward orbits never leave the neighbourhood, and tend exponentially to $\Lambda$. The local unstable manifold $W^{u}_{\mathrm{loc}}(\Lambda)$ is the set of points in the neighbourhood whose backward orbits stay in the neighbourhood and tend exponentially to $\Lambda$.  We then define
\begin{equation}
W^s(\Lambda) = \bigcup_{t \geq 0}^{\infty} \phi^{-t} \left( W^{s}_{\mathrm{loc}}(\Lambda) \right), \quad W^u(\Lambda) = \bigcup_{t \geq 0}^{\infty} \phi^{t} \left( W^{u}_{\mathrm{loc}}(\Lambda) \right).
\end{equation}
On the stable and unstable manifolds we have the strong stable and strong unstable foliations, the leaves of which we denote by $W^{s,u}(x)$ for $x \in \Lambda$. For each $x \in \Lambda$, the leaf $W^s(x)$ of the strong stable foliation is tangent at $x$ to $E^s_x$, and the leaf $W^u(x)$ of the strong unstable foliation is tangent at $x$ to $E^u_x$. Moreover the foliations are invariant in the sense that $\phi^t \left( W^s (x) \right) = W^s \left( \phi^t (x) \right)$ and $\phi^t \left( W^u (x) \right) = W^u \left( \phi^t (x) \right)$ for each $x \in \Lambda$ and $t \in \mathbb{R}$. We thus define the \emph{holonomy maps} $\pi^{s,u} : W^{s,u} (\Lambda) \to \Lambda$ to be projections along leaves of the strong stable and strong unstable foliations. That is to say, if $x \in W^s (\Lambda)$ then there is a unique $x_+ \in \Lambda$ such that $x \in W^s(x_+)$, and so $\pi^s(x)=x_+$. Similarly, if $x \in W^u (\Lambda)$ then there is a unique $x_- \in \Lambda$ such that $x \in W^u(x_-)$, in which case $\pi^u(x)=x_-$.

Now, suppose that $x \in \left(W^s(\Lambda) \pitchfork W^u (\Lambda)\right) \setminus \Lambda$ is a transverse homoclinic point such that $x \in W^s(x_+) \cap W^u(x_-)$. We say that the homoclinic intersection at $x$ is \emph{strongly transverse} if
\begin{equation}
\begin{split} \label{eq_strongtransversality}
T_x W^s (x_+) \oplus T_x \left( W^s(\Lambda) \cap W^u(\Lambda) \right) = T_x W^s (\Lambda), \\
T_x W^u (x_-) \oplus T_x \left( W^s(\Lambda) \cap W^u(\Lambda) \right) = T_x W^u (\Lambda).
\end{split}
\end{equation}
In this case we can take a sufficiently small neighbourhood $\Gamma$ of $x$ in $W^s(\Lambda) \cap W^u(\Lambda)$ so that \eqref{eq_strongtransversality} holds at each point of $\Gamma$, and the restrictions to $\Gamma$ of the holonomy maps are bijections onto their images. We call $\Gamma$ a \emph{homoclinic channel} (see Figure \ref{figure_homoclinicchannel}). We can then define the scattering map as follows \cite{delshams2008geometric}.

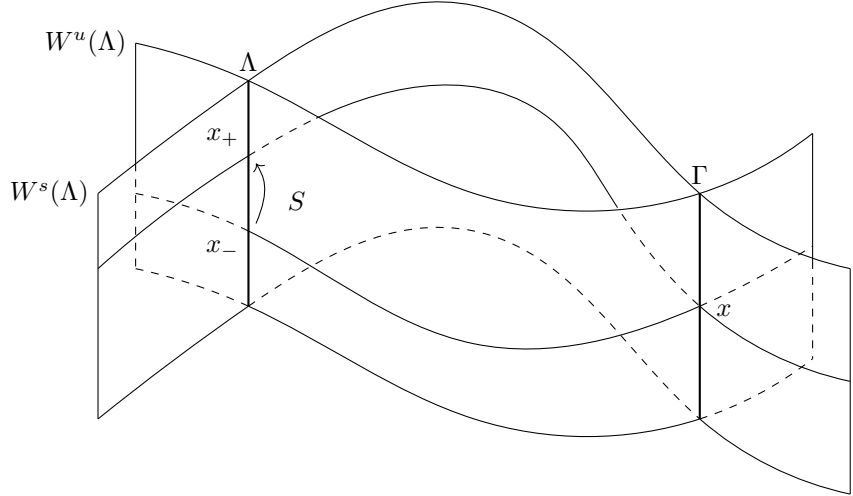
\begin{figure}
\centering
  \begin{tikzpicture}[use Hobby shortcut]

    \draw (0,0) .. (2,1.5) .. (5,2.5) .. (8,0) .. (10,-1);
    \draw (0,0) node[anchor=east] {$W^s(\Lambda)$};
    \draw (0,-3) .. (2,-1.5) .. ([blank=soft]5,-0.5) .. ([blank=soft]8,-3) .. (10,-4);
    \draw[dashed,use previous Hobby path={invert soft blanks,disjoint}]; 
    \draw (0.5,2) .. (2,1.5) .. (5,0) .. (8,0) .. (9.5,0.8);
    \draw (0.5,2) node[anchor=east] {$W^u(\Lambda)$};
    \draw (0.5,-1) .. ([blank=soft]2,-1.5) .. (5,-3) .. (8,-3)
    .. ([blank=soft]9.5,-2.2); 
    \draw[dashed,use previous Hobby path={invert soft blanks,disjoint}]; 
    \draw (0,0) -- (0,-3);
    \draw (10,-1) -- (10,-4);
    \draw (0.5,2) .. (0.5,0.4) .. ([blank=soft]0.5,-1);
    \draw[dashed,use previous Hobby path={invert soft blanks,disjoint}]; 
    \draw (9.5,0.8) .. (9.5,-0.9) .. ([blank=soft]9.5,-2.2);
    \draw[dashed,use previous Hobby path={invert soft blanks,disjoint}]; 
    \draw[thick] (2,-1.5) -- (2,1.5) node[anchor=south] {$\Lambda$};
    \draw[thick] (8,-3) -- (8,0) node[anchor=south] {$\Gamma$};
    \draw (0,-1) .. (2,0.5) .. ([blank=soft]2.9,1) .. (6,1)
    .. (6.85,-0.1) .. ([blank=soft]8,-1.5) .. (10,-2.5);
    \draw[dashed,use previous Hobby path={invert soft blanks,disjoint}]; 
    \draw (0.5,0) .. ([blank=soft]2,-0.5) .. (5,-2) .. (8,-1.5) .. ([blank=soft]9.5,-0.7);
    \draw[dashed,use previous Hobby path={invert soft blanks,disjoint}]; 
    \draw (2,-0.5) node[anchor=north east] {$x_-$};
    \draw (2,0.5) node[anchor=south east] {$x_+$};
    \draw[->] (2.1,-0.4) .. (2.4,-0.1) node [anchor=west] {$S$} .. (2.1,0.4);
    \draw (8.1,-1.55) node[anchor=west] {$x$};
  \end{tikzpicture}
  \caption{The scattering map $S$ takes a point $x_- \in \Lambda$, follows the unique leaf of the strong unstable foliation passing through $x_-$ to the point $x$ in the homoclinic channel $\Gamma$, and from there follows the unique leaf of the strong stable foliation passing through $x$ to the point $x_+$ on $\Lambda$.} \label{figure_homoclinicchannel}
\end{figure}

\begin{definition}\label{def:scattering}
Let $y_-\in \pi^u \left( \Gamma \right)$, let $y = \left(\left. \pi^u \right|_{\Gamma} \right)^{-1} (y_-)$, and let $y_+ = \pi^s(y)$. The \emph{scattering map} $S : \pi^u (\Gamma) \to \pi^s (\Gamma)$ is defined by
\begin{equation}
S = \pi^s \circ \left( \pi^u \right)^{-1} : y_- \longmapsto y_+.
\end{equation}
\end{definition}

Suppose now that the smoothness $r$ of $M$ and $X$ is at least $2$, suppose the normally hyperbolic invariant manifold $\Lambda$ is a $C^r$ submanifold of $M$, and suppose the large spectral gap condition \eqref{eq_largespectralgap} holds. This implies $C^{r-1}$ regularity of the strong stable and strong unstable foliations \cite{hirsch1970invariant}, which in turn implies that the scattering map $S$ is $C^{r-1}$ \cite{delshams2008geometric}.

In general, the scattering map may be defined only locally, as the transverse homoclinic intersection of stable and unstable manifolds can be very complicated; however in the setting of the present paper, the scattering maps we study turn out to be globally defined.

\section{A general shadowing argument}\label{appendix_shadowing} 
We follow the notation and exposition of \cite{clarke2022topological}. Let $M$ be a $C^r$ manifold of dimension $d=2(m+n)$ where $r \geq 4$. Let $F \in \Diff^4(M)$, and assume $F$ depends smoothly on a small parameter $\epsilon$, with uniformly bounded derivatives. Suppose $F$ has a normally hyperbolic invariant (or locally invariant) manifold $\Lambda \subset M$ of dimension $2n$ satisfying the large spectral gap condition \eqref{eq_largespectralgap}; suppose moreover that $\Lambda$ is diffeomorphic to $\mathbb{T}^n \times [0,1]^n$. Furthermore, we assume that $\dim W^s (\Lambda) = \dim W^u (\Lambda) = m + 2n$. In order to state the remaining assumptions and the shadowing theorems, we must consider some definitions. 

Suppose the scattering map $S$ is defined relative to a homoclinic channel $\Gamma$ for all sufficiently small $\epsilon >0$. We allow for the possibility that the angle between $W^{s,u}(\Lambda)$ along the homoclinic channel $\Gamma$ goes to 0 as $\epsilon \to 0$. Denote by $\alpha (v_1, v_2)$ the angle between two vectors $v_1, v_2$ in the direction that yields the smallest result (i.e. $\alpha (v_1,v_2) \in [0, \pi]$). For $x \in \Gamma$, let 
\begin{equation}
\alpha_{\Gamma} (x) = \inf \alpha (v_+, v_-)
\end{equation}
where the infimum is over all $v_+ \in T_x W^s (\Lambda)^{\perp}$ and $v_- \in T_x W^u (\Lambda)^{\perp}$ such that $\| v_{\pm} \| = 1$. 
\begin{definition}
For $\sigma \geq 0$, we say that \emph{the angle of the splitting along $\Gamma$ is of order $\epsilon^{\sigma}$} if there is a positive constant $C$ (independent of $\epsilon$) such that
\begin{equation}
\alpha_{\Gamma} (x) \geq C \epsilon^{\sigma}\qquad\text{for all}\quad x \in \Gamma.
\end{equation}
\end{definition}

Recall we have assumed that the normally hyperbolic invariant manifold $\Lambda$ is diffeomorphic to $\mathbb{T}^n \times [0,1]^n$, and denote by $(q,p) \in \mathbb{T}^n \times [0,1]^n$ smooth coordinates on $\Lambda$. Define $f \coloneqq F|_{\Lambda}$, which also depends on the small parameter $\epsilon$.
\begin{definition}\label{def_nearlyintegrabletwist}
We say that $f: \Lambda \to \Lambda$ is a \emph{near-integrable twist map} if there is some $k \in \mathbb{N}$ such that
\begin{equation}\label{eq_inttwistmap}
f:
\begin{cases}
\bar{q} = q + g(p) + O(\epsilon^k) \\
\bar{p} = p + O(\epsilon^k)
\end{cases}
\end{equation}
where
\begin{equation}
\det D g (p) \neq 0
\end{equation}
for all $p \in [0,1]^n$, and where the higher order terms are uniformly bounded in the $C^1$ topology. If the higher order terms are 0 then $f$ is an \emph{integrable twist map}.
\end{definition}

It follows from the definition that if $f : \Lambda \to \Lambda$ is a near-integrable twist map, then there exist twist parameters $T_+ > \widetilde{T}_- >0$ such that 
\begin{equation}\label{eq_twistcondition2}
\widetilde{T}_- \| v \| \leq \left\| D g(p) v \right\| \leq T_+ \| v \|
\end{equation}
for all $p \in [0,1]^n$ and all $v \in \mathbb{R}^n$. We can always choose $T_+$ to be independent of $\epsilon$. Our formulation of the problem allows the parameter $\widetilde{T}_-$ to depend on $\epsilon$: there is $\tau \in \mathbb{N}_0$ and a strictly positive constant $T_-$ (independent of $\epsilon$) such that $\widetilde{T}_- = \epsilon^{\tau} T_-$. 
\begin{definition}\label{def_nonuniformtwist}
Suppose $f : \Lambda \to \Lambda$ is a near-integrable twist map. Denote by $T_+ > \widetilde{T}_- = \epsilon^{\tau} T_- >0$ the twist parameters. We say that $f$ satisfies:
\begin{itemize}
\item
A \emph{uniform twist condition} if $\tau=0$; 
\item
A \emph{non-uniform twist condition (of order $\epsilon^{\tau}$)} if $\tau>0$, and the order $\epsilon^k$ of the error terms in the definition of the near-integrable twist map $f$ is such that $k > \tau$.
\end{itemize}
\end{definition}

In the coordinates $(q,p)$, we may define a foliation of $\Lambda$, the leaves of which are given by
\begin{equation}\label{eq_foliationleaves}
\L (p^*) = \left\{ (q,p) \in \Lambda : p=p^* \right\}.
\end{equation}
If $f : \Lambda \to \Lambda$ is a near-integrable twist map in the sense of Definition \ref{def_nearlyintegrabletwist}, then each leaf of the foliation is almost invariant under $f$, up to terms of order $\epsilon^k$. Denote by $U \subset \Lambda$ the domain of definition of the scattering map $S$.

\begin{definition}\label{def:transvfoliation}
We say that the scattering map $S$ is \emph{transverse to leaves along leaves, and that the angle of transversality is of order $\epsilon^{\upsilon}$ (with respect to the leaves \eqref{eq_foliationleaves} of the foliation of $\Lambda$)} if there are $c, \, C > 0$ such that for all $p_0^* \in [0,1]^n$ and all $p^* \in [0,1]^n$ satisfying $\| p^* - p_0^* \| < c \, \epsilon^{\upsilon}$ we have 
\[
S \left( \L (p_0^*) \cap U \right) \pitchfork \L (p^*) \neq \emptyset
\]
and there is $x \in S \left( \L (p_0^*) \cap U \right) \pitchfork \L (p^*)$ such that
\[
\inf \alpha (v_0,v) \geq C \epsilon^{\upsilon}
\]
where the infimum is taken over all $v_0 \in T_x S \left( \L (p_0^*) \cap U \right)$ and $v \in T_x \L (p^*)$ such that $\| v_0 \| = \| v \| = 1$. 
\end{definition}

Using these definitions, we may now state the main assumptions of the first shadowing theorem, which will be applied to the secular Hamiltonian \eqref{eq_fsec} to prove the existence of drifting orbits in the secular subsystem. 
\begin{enumerate}[{[}{A}1{]}]
\item
The stable and unstable manifolds $W^{s,u} (\Lambda)$ have a strongly transverse homoclinic intersection along a homoclinic channel $\Gamma$, and so we have an open set $U \subseteq \Lambda$ and a scattering map $S : U \to \Lambda$. The angle of the splitting along $\Gamma$ is of order $\epsilon^{\sigma}$. 
\item
The inner map $f = F|_{\Lambda}$ is a near-integrable twist map with error terms of order $\epsilon^k$ satisfying a non-uniform (or uniform) twist condition of order $\epsilon^{\tau}$. 
\item
The scattering map $S$ is transverse to leaves along leaves (with respect to the leaves \eqref{eq_foliationleaves} of the foliation of $\Lambda$), and the angle of transversality is of order $\epsilon^{\upsilon}$. 
\end{enumerate}

\begin{theorem} \label{theorem_main1} 
Fix $\eta >0$, let $\epsilon > 0$ be sufficiently small, and suppose
$k \geq 2 \left( \rho + \tau \right) + 1$
where
$\rho = \max \{2 \sigma,  2\upsilon, \tau \}$.
Choose $\{ p_j^* \}_{j=1}^{\infty} \subset [0,1]^n$ such that 
\begin{equation}
S \left( \L_j \cap U \right) \cap \L_{j+1} \neq \emptyset,
\end{equation}
and $S \left( \L_j \cap U \right)$ is transverse to $\L_{j+1}$, where $\L_j = \L (p_j^*)$. Suppose the distance between $\L_j$ and $\L_{j+1}$ is of order $\epsilon^{\upsilon}$ for each $j$. Then there are $\{z_i\}_{i=1}^{\infty} \subset M$ and $n_i \in \mathbb{N}$ such that 
$z_{i+1} = F^{n_i} (z_i)$
and 
\begin{equation}
d ( z_i, \L_i) < \eta.
\end{equation}
Moreover, the time to move a distance of order 1 in the $p$-direction is bounded from above by a term of order
\begin{equation} \label{eq_timeestimate1}
\epsilon^{- \rho - \tau - \upsilon}.
\end{equation}
\end{theorem}

Observe that Theorem \ref{theorem_main1} cannot be applied to \eqref{eq_4bphamave}. Indeed, a crucial assumption in Theorem \ref{theorem_main1} is that the scattering map $S$ is transverse to leaves along leaves. For \eqref{eq_4bphamave}, we have no information about the behaviour of the scattering map in the $L_i$ directions, and so we cannot check assumption [A3] for the Hamiltonian \eqref{eq_4bphamave}. Theorem \ref{theorem_main2} below generalises Theorem \ref{theorem_main1} to settings where transversality is only known in some directions, and thus allows us to complete the proof of Theorem \ref{thm:MainHierarch:Deprit}. 

To state Theorem \ref{theorem_main2} we consider, as before, a $C^r$ manifold $M$ of dimension $2(m+n)$ where $r \geq 4$ and $m,n \in \mathbb{N}$. Let $\Sigma = \mathbb{T}^{\ell_1} \times [0,1]^{\ell_2}$ for some $\ell_1, \ell_2 \in \mathbb{N}_0$, and denote by $(\theta, \xi) \in \mathbb{T}^{\ell_1} \times [0,1]^{\ell_2}$ coordinates on $\Sigma$. Write $\widetilde{M} = M \times \Sigma$. Suppose $\Psi \in \Diff^4 \left(\widetilde{M} \right)$ such that 
\begin{equation}
\Psi (z, \theta, \xi) = \left( G(z, \theta, \xi), \phi (z, \theta, \xi) \right)
\end{equation}
where $z \in M$, $G \in C^4 \left( \widetilde{M}, M \right)$, and $\phi \in C^4 \left( \widetilde{M}, \Sigma \right)$. Suppose $\Psi$ depends on a small parameter $\epsilon$. We make the following assumptions on $\Psi$. 
\begin{enumerate}[{[}{B}1{]}]
\item
There is some $L \in \mathbb{N}$ such that
\begin{equation}
G (z, \theta, \xi) = \widetilde{G}(z; \xi ) + O \left(\epsilon^L \right)
\end{equation}
where the higher order terms are uniformly bounded in the $C^4$ topology, and for each $\xi \in [0,1]^{\ell_2}$ the map
\begin{equation}
\widetilde{G} ( \cdot ; \xi ): z \in M \longmapsto \widetilde{G} ( z ; \xi ) \in M
\end{equation}
satisfies the assumptions [A1-3] of Theorem \ref{theorem_main1}.
\item
Moreover, the map $\phi$ has the form
\begin{equation}
\phi:
\begin{cases}
\bar{\theta} = \phi_1 (z, \theta, \xi) \\
\bar{\xi} = \phi_2(z, \theta, \xi) = \xi + O \left( \epsilon^L \right)
\end{cases}
\end{equation}
where the higher order terms are uniformly bounded in the $C^4$ topology. 
\end{enumerate}

Results from \cite{delshams2008geometric} imply that $\Psi$ has a normally hyperbolic invariant manifold $\widetilde{\Lambda}$ that is $O \left(\epsilon^L \right)$ close in the $C^4$ topology to $\Lambda \times \Sigma$ where $\Lambda \subset M$ is the normally hyperbolic invariant manifold of $\widetilde{G} ( \cdot ; \xi )$. Moreover there is an open set $\widetilde{U} \subset \widetilde{\Lambda}$ and a scattering map $\widetilde{S} : \widetilde{U} \to \widetilde{\Lambda}$ such that the $z$-component of $\widetilde{S}(z, \theta, \xi)$ is $O \left( \epsilon^L \right)$ close in the $C^3$ topology to $S \left( z ; \xi \right)$ where $S \left( \cdot ; \xi \right) : U \to \Lambda$ is the scattering map corresponding to $\widetilde{G} ( \cdot ; \xi )$. 

We use the coordinates $(q,p, \theta, \xi)$ on $\widetilde{\Lambda}$ where $(q,p)$ are the coordinates on $\Lambda$ and $(\theta, \xi)$ are the coordinates on $\Sigma$. Notice that the sets
\begin{equation}
\widetilde{\L} \left(p^*, \xi^* \right) = \left\{ (q,p, \theta, \xi) \in \widetilde{\Lambda}: p = p^*, \xi=\xi^* \right\} = \L \left(p^* \right) \times \mathbb{T}^{\ell_1} \times \left\{ \xi^* \right\}
\end{equation}
for $p^* \in [0,1]^n$ and $\xi^* \in [0,1]^{\ell_2}$ define the leaves of a foliation of $\widetilde{\Lambda}$, where $\L(p^*)$ are the leaves of the foliation of $\Lambda$ defined by \eqref{eq_foliationleaves}. 

\begin{theorem}\label{theorem_main2}
Fix $\eta>0$ and $K\in\mathbb{N}$ and let $\epsilon >0$ be sufficiently small. Choose $N \in \mathbb{N}$ satisfying
\[
 N\leq \frac{1}{\epsilon^K}
\]
$\xi^*_1 \in \mathrm{Int} \left([0,1]^{\ell_2} \right)$ so that $\widetilde{G} ( \cdot ; \xi^*_1 )$ satisfies assumptions [A1-3],
and $p_1^*, \ldots, p_N^* \in [0,1]^n$ as in Theorem \ref{theorem_main1} such that
\begin{equation}
S \left( \L_j \cap U ; \xi^*_1 \right) \cap \L_{j+1} \neq \emptyset
\end{equation}
and $S \left( \L_j \cap U; \xi^*_1 \right)$ is transverse to $\L_{j+1}$, where $\L_j = \L (p_j^*)$. Suppose the distance between $\L_j$ and $\L_{j+1}$ is of order $\epsilon^{\upsilon}$ for each $j$, and $L >0$ is sufficiently large, depending on $K$. Then there are $\xi^*_2, \ldots, \xi^*_N \in [0,1]^{\ell_2}$ such that, with $\widetilde{\L}_j = \widetilde{\L} \left(p^*_j, \xi^*_j \right)$, there are $w_1, \ldots, w_N \in \widetilde{M}$ and $n_i \in \mathbb{N}$ such that the $\xi$ component of $w_1$ is $\xi^*_1$,
\begin{equation}
w_{i+1} = \Psi^{n_i} (w_i),
\end{equation}
and
\begin{equation}
d \left(w_i, \widetilde{\L}_i \right) < \eta
\end{equation}
where $\rho, \sigma, \tau$ are as in the statement of Theorem \ref{theorem_main1}. Moreover, the time to move a distance of order 1 in the $p$-direction is of order
$\epsilon^{- \rho - \tau - \upsilon}$.
\end{theorem}

Note that the transition chain obtained in Theorem \ref{theorem_main2} is only of finite length, while the one obtained in Theorem \ref{theorem_main1} may be infinite.

\section{Derivatives of the inner Hamiltonian}\label{appendix_derivatives}

Recall the definition of the integrable part $\hat F_0$ of the inner secular Hamiltonian, after straightening the symplectic form and averaging the inner angles, as constructed in Theorem \ref{theorem_inner}. In order to prove in Lemma \ref{lemma_twist} that the Poincar\'e map satisfies a twist condition, we need to compute the first and second partial derivatives of $\hat F_0$ with respect to the inner actions, or in some cases simply estimate the order. This information is provided in the following lemma. 

\begin{lemma}\label{lemma_innerderivatives}
The first and second-order partial derivatives of $\hat F_0$ with respect to the actions $\hat \Gamma_2$, $\hat \Psi_1$, $\hat \Gamma_3$, and $\hat L_3$ are as follows:
\begin{align}
\frac{\partial \hat F_0}{\partial \hat \Gamma_2} =& \eps^6 \, C_{12} \frac{6 \hat \Gamma_2}{L_1^2 \, \delta_1^3} + \cdots, \quad & \frac{\partial \hat F_0}{\partial \hat \Psi_1} =& 3 \, \eps^7 \, C_{12} \frac{L_1^2 - 3 \, \hat \Gamma_2^2}{L_1^2 \, \delta_1^4} + \cdots, \\
\frac{\partial \hat F_0}{\partial \hat \Gamma_3} =& \eps^3 \, \mu^6 C_{23} \frac{(20 - 12 \, \delta_1^2) \, \delta_3}{\delta_1^2 \, \delta_2^3} + \cdots , \quad & \frac{\partial \hat F_0}{\partial \hat L_3} =& \eps^3 \, \mu^3 \, \alpha_{\mathrm{Kep}} + \cdots, \\
\frac{\partial^2 \hat F_0}{\partial \hat \Gamma_2^2} =& \eps^6 \, C_{12} \, \frac{6}{L_1^2 \, \delta_1^3} + \cdots, \quad & \frac{\partial^2 \hat F_0}{\partial \hat \Gamma_2 \, \partial \hat \Psi_1 } =& - \eps^7 \, C_{12} \, \frac{18 \, \hat \Gamma_2}{L_1^2 \, \delta_1^4} + \cdots , \\
\frac{\partial^2 \hat F_0}{\partial \hat \Gamma_2 \, \partial \hat \Gamma_3} =& \eps^4 \, \mu^6 \, C_{23} \, \frac{24 \, \delta_3}{\delta_1 \, \delta_2^3} + \cdots, \quad & \frac{\partial^2 \hat F_0}{\partial \hat \Psi_1^2  } =& 12 \, \eps^8 \, C_{12} \, \frac{3 \, \hat \Gamma_2^2 - L_1^2}{L_1^2 \, \delta_1^5} + \cdots, \\
\frac{\partial^2 \hat F_0}{\partial \hat \Psi_1 \, \partial \hat \Gamma_3} =& - \eps^4 \, \mu^6 \, C_{23} \, \frac{40 \, \delta_3}{\delta_1^3 \, \delta_2^3} + \cdots, \quad & \frac{\partial^2 \hat F_0}{\partial \hat \Gamma_3^2  } =& \eps^4 \, \mu^6 \, C_{23} \, \frac{20 - 12 \, \delta_1^2}{\delta_1^2 \, \delta_2^3} + \cdots, \\
\frac{\partial^2 \hat F_0}{\partial \hat L_3^2} =& - 3 \, \eps^4 \, \mu^4 \, \alpha_{\mathrm{Kep}} + \cdots, \quad & \frac{\partial^2 \hat F_0}{\partial \hat L_3 \, \partial \hat \Gamma_2} =& O \left( \eps^4 \, \mu^7 \right), \\
\frac{\partial^2 \hat F_0}{\partial \hat L_3 \, \partial \hat \Psi_1} =& O \left( \eps^4 \, \mu^7 \right), \quad & \frac{\partial^2 \hat F_0}{\partial \hat L_3 \, \partial \hat \Gamma_3} =& O \left( \eps^4 \, \mu^7 \right), \\
\end{align}
where $\eps = \frac{1}{L_2}$, $\mu = \frac{L_2}{L_3^*}$, where $C_{12}$, $C_{23}$ are nonzero constants independent of $L_2$ and $L_3^*$ coming from $F_{\mathrm{quad}}^{12}$, $F_{\mathrm{quad}}^{23}$ respectively, and where the nonzero constant $\alpha_{\mathrm{Kep}}$ is defined by \eqref{eq_keplerconst}.
\end{lemma}

\begin{proof}
Observe that $\left. F_{\mathrm{quad}}^{12} \right|_{\Lambda}$ and $\left. F_{\mathrm{quad}}^{23} \right|_{\Lambda}$ are the same, after we average the inner angles, as the analogous objects in \cite{clarke2022why} up to higher order terms depending also on $\hat L_3$. Therefore all of the derivatives taken with respect to the variables $\hat \Gamma_2$, $\hat \Psi_1$, $\hat \Gamma_3$ are given at first order by Lemma 25 of \cite{clarke2022why}. 

From \eqref{eq_fkepexpansion}, \eqref{eq_keplerconst}, and \eqref{eq_fsec}, we see that 
\begin{align}
\frac{\partial \hat F_0}{\partial \hat L_3} = \frac{\partial F_{\mathrm{Kep}}}{\partial \hat L_3} + \cdots = \frac{1}{\left( L_3^* \right)^3} \, \alpha_{\mathrm{Kep}}+ \cdots, \quad \frac{\partial^2 \hat F_0}{\partial \hat L_3^2} = \frac{\partial^2 F_{\mathrm{Kep}}}{\partial \hat L_3^2} + \cdots = - 3 \, \frac{1}{\left( L_3^* \right)^4} \, \alpha_{\mathrm{Kep}}+ \cdots. 
\end{align}
As for the mixed second partial derivatives with respect to $\hat L_3$ and the other actions, we estimate the order as follows. Products of $\hat L_3$ and the other actions come, at first order, in the expansion of $F_{\mathrm{quad}}^{23}$. We can find these by normalising $F_{\mathrm{quad}}^{23}$ to obtain
\[
\tilde F_{\mathrm{quad}}^{23} = \frac{L_3^6}{L_2^4} \, \frac{1}{\left( 2 \, \pi \right)^4} \int_{\mathbb{T}^4} F_{\mathrm{quad}}^{23} \, d \tilde \gamma_2 \, d \tilde \psi_1 \, d  \tilde \gamma_3 \, d \tilde \ell_3,
\]
expanding the coefficient 
\begin{equation}\label{eq_quad23coefficientexpappendix}
\frac{L_2^4}{L_3^6} = \frac{L_2^4}{\left(L_3^* + \tilde L_3 \right)^6} = \frac{L_2^4}{\left(L_3^* \right)^6} - 6 \, \frac{L_2^4}{\left( L_3^* \right)^7} \, \tilde L_3 + O \left( \frac{L_2^4}{\left( L_3^* \right)^8} \right),
\end{equation}
and noticing that the first appearance of the actions $\tilde \Gamma_2$, $\tilde \Psi_1$, and $\tilde \Gamma_3$ in the expansion of $\tilde F_{\mathrm{quad}}^{23}$ can be estimated by
\begin{equation}\label{eq_quad23normalisedderivativesappendix}
\frac{\partial \tilde F_{\mathrm{quad}}^{23}}{\partial \left( \tilde \Gamma_2, \tilde \Psi_1, \tilde \Gamma_3 \right) } = O \left( \frac{1}{L_2} \right)
\end{equation}
because the first order term in the expansion of $F_{\mathrm{quad}}^{23}$ (see $H_0^{23}$, defined by \eqref{eq_h023h123def}) does not depend on any of the actions. Combining \eqref{eq_quad23coefficientexpappendix} and \eqref{eq_quad23normalisedderivativesappendix}, and defining $\left\langle F_{\mathrm{quad}}^{23}\right\rangle = \frac{L_2^4}{L_3^6} \tilde F_{\mathrm{quad}}^{23}$ yields
\[
\frac{\partial^2 \left\langle F_{\mathrm{quad}}^{23} \right\rangle }{\partial \tilde L_3 \, \partial \left( \tilde \Gamma_2, \tilde \Psi_1, \tilde \Gamma_3 \right)} = O \left( \frac{L_2^3}{\left( L_3^* \right)^7} \right) = O \left( \eps^4 \, \mu^7 \right),
\]
and so
\begin{equation}\label{eq_innerquadderivativesesimateappendix}
\frac{\partial^2 \left\langle F_{\mathrm{quad}}^{23} \right\rangle}{\partial \hat L_3 \, \partial \left( \hat \Gamma_2, \hat \Psi_1, \hat \Gamma_3 \right)} = O \left( \frac{L_2^3}{\left( L_3^* \right)^7} \right) = O \left( \eps^4 \, \mu^7 \right)
\end{equation}
since the transformation from the `tilde' to the `hat' coordinates is close to the identity in $C^r$. To see that this implies the estimates given in the statement of the lemma, it only remains to check that restricting $F_{\mathrm{sec}}$ to the cylinder $\Lambda$ and then taking derivatives does not spoil the estimates. Indeed, recall the first order of the Hamiltonian that depends on the Poincar\'e variables $\xi, \eta$ is of order $\frac{1}{L_2^6}$ (i.e. the first order term $H_0^{12}$ in the expansion of the secular Hamiltonian, defined by \eqref{eq_H012inpoincarevariables}). Moreover the first order term of the graph $\rho$ defining $\Lambda$ that depends on the variable $\tilde L_3$ is of order $\frac{L_2^{10}}{\left( L_3^* \right)^7}$ (see Lemma \eqref{lemma_graphexpansion}). Since $H_0^{12}$ depends quadratically on $\xi$ and $\eta$ (see \eqref{eq_H012inpoincarevariables}), the term $H_1^{12}$ is the lowest-order term that could contain products of the form $\tilde L_3 \, P$ for $P \in \{ \tilde \Gamma_3, \tilde \Psi_1, \tilde \Gamma_3 \}$; since the coefficient of $H_1^{12}$ is of order $\frac{1}{L_2^7}$, the order of such terms is $\frac{1}{L_2^7} \, \frac{L_2^{10}}{\left( L_3^* \right)^7} = \eps^4 \, \mu^7$. Therefore the estimates \eqref{eq_innerquadderivativesesimateappendix} imply the estimates on the mixed partial derivatives of $\hat F_0$ with respect to $\hat L_3$ and $\hat \Gamma_2$, $\hat \Psi_1$, $\hat \Gamma_3$ as required. 
\end{proof}

{\small
\bibliographystyle{abbrv}
\bibliography{semimajor_axis} 

\begin{thebibliography}{10}

\bibitem{Arnold:1963}
V.~I. Arnold.
\newblock Small denominators and problems of stability of motion in classical
  and celestial mechanics.
\newblock {\em Uspehi Mat. Nauk}, 18(6 (114)):91--192, 1963.

\bibitem{arnold1964instability}
V.~I. Arnold.
\newblock Instability of dynamical systems with many degrees of freedom.
\newblock {\em Dokl. Akad. Nauk SSSR}, 156:9--12, 1964.

\bibitem{Arnold:2006}
V.~I. Arnold, V.~V. Kozlov, and A.~I. Neishtadt.
\newblock {\em Mathematical aspects of classical and celestial mechanics},
  volume~3 of {\em Encyclopaedia of Mathematical Sciences}.
\newblock Springer-Verlag, Berlin, third edition, 2006.

\bibitem{Bernard08}
P.~Bernard.
\newblock The dynamics of pseudographs in convex {H}amiltonian systems.
\newblock {\em J. Amer. Math. Soc.}, 21(3):615--669, 2008.

\bibitem{Kaloshin:2016}
P.~Bernard, V.~Kaloshin, and K.~Zhang.
\newblock Arnold diffusion in arbitrary degrees of freedom and normally
  hyperbolic invariant cylinders.
\newblock {\em Acta Math.}, 217(1):1--79, 2016.

\bibitem{Bolotin:1999}
S.~Bolotin and D.~Treschev.
\newblock Unbounded growth of energy in nonautonomous {H}amiltonian systems.
\newblock {\em Nonlinearity}, 12(2):365--388, 1999.

\bibitem{CapinskiGidea}
M.~J. Capi\'{n}ski and M.~Gidea.
\newblock Arnold diffusion, quantitative estimates, and stochastic behavior in
  the three-body problem.
\newblock {\em Communications on Pure and Applied Mathematics}, 2021.

\bibitem{capinski2017diffusion}
M.~J. Capi\'{n}ski, M.~Gidea, and R.~de~la Llave.
\newblock Arnold diffusion in the planar elliptic restricted three-body
  problem: mechanism and numerical verification.
\newblock {\em Nonlinearity}, 30(1):329--360, 2017.

\bibitem{Cheng:2017}
C.~Cheng.
\newblock Dynamics around the double resonance.
\newblock {\em Camb. J. Math.}, 5(2):153--228, 2017.

\bibitem{ChengY04}
C.~Cheng and J.~Yan.
\newblock Existence of diffusion orbits in a priori unstable {H}amiltonian
  systems.
\newblock {\em J. Differential Geom.}, 67(3):457--517, 2004.

\bibitem{chierchia2011deprit}
L.~Chierchia and G.~Pinzari.
\newblock Deprit’s reduction of the nodes revisited.
\newblock {\em Celestial Mechanics and Dynamical Astronomy}, 109(3):285--301,
  2011.

\bibitem{Chierchia:2011}
L.~Chierchia and G.~Pinzari.
\newblock The planetary {$N$}-body problem: symplectic foliation, reductions
  and invariant tori.
\newblock {\em Invent. Math.}, 186(1):1--77, 2011.

\bibitem{clarke2022why}
A.~Clarke, J.~Fejoz, and M.~Guardia.
\newblock Why are inner planets not inclined?
\newblock {\em arXiv:2210.11311}, 2022.

\bibitem{clarke2022topological}
A.~Clarke, J.~Fejoz, and M.~Guàrdia.
\newblock Topological shadowing methods in arnold diffusion: weak torsion and
  multiple time scales.
\newblock {\em Nonlinearity}, 36(1):426 – 457, 2023.
\newblock Cited by: 0; All Open Access, Green Open Access.

\bibitem{clarke2022arnold}
A.~Clarke and D.~Turaev.
\newblock Arnold diffusion in multi-dimensional convex billiards.
\newblock {\em Duke Mathematical Journal}, to appear, 2022.

\bibitem{DelshamsLS00}
A.~Delshams, R.~de~la Llave, and T.~Seara.
\newblock A geometric approach to the existence of orbits with unbounded energy
  in generic periodic perturbations by a potential of generic geodesic flows of
  $\mathbb{T}\sp 2$.
\newblock {\em Comm. Math. Phys.}, 209(2):353--392, 2000.

\bibitem{DelshamsLS06b}
A.~Delshams, R.~de~la Llave, and T.~Seara.
\newblock Orbits of unbounded energy in quasi-periodic perturbations of
  geodesic flows.
\newblock {\em Adv. Math.}, 202(1):64--188, 2006.

\bibitem{delshams2006biggaps}
A.~Delshams, R.~de~la Llave, and T.~M. Seara.
\newblock A geometric mechanism for diffusion in {H}amiltonian systems
  overcoming the large gap problem: heuristics and rigorous verification on a
  model.
\newblock {\em Mem. Amer. Math. Soc.}, 179(844):viii+141, 2006.

\bibitem{delshams2008geometric}
A.~Delshams, R.~De~La~Llave, and T.~M. Seara.
\newblock Geometric properties of the scattering map of a normally hyperbolic
  invariant manifold.
\newblock {\em Advances in Mathematics}, 217(3):1096--1153, 2008.

\bibitem{MR3479576}
A.~Delshams, R.~de~la Llave, and T.~M. Seara.
\newblock Instability of high dimensional {H}amiltonian systems: multiple
  resonances do not impede diffusion.
\newblock {\em Adv. Math.}, 294:689--755, 2016.

\bibitem{DelshamsH05}
A.~Delshams and G.~Huguet.
\newblock Geography of resonances and {A}rnold diffusion in a priori unstable
  {H}amiltonian systems.
\newblock {\em Nonlinearity}, 22(8):1997--2077, 2009.

\bibitem{delshams2019instability}
A.~Delshams, V.~Kaloshin, A.~de~la Rosa, and T.~M. Seara.
\newblock Global instability in the restricted planar elliptic three body
  problem.
\newblock {\em Comm. Math. Phys.}, 366(3):1173--1228, 2019.

\bibitem{deprit1983}
A.~Deprit.
\newblock Elimination of the nodes in problems of {$n$} bodies.
\newblock {\em Celestial Mech.}, 30(2):181--195, 1983.

\bibitem{Farhat:2022}
M.~Farhat, P.~Auclair-Desrotour, G.~Bou{\'{e} }, and J.~Laskar.
\newblock The resonant tidal evolution of the earth-moon distance.
\newblock {\em Astronomy {\&} Astrophysics}, 2022.

\bibitem{fejoz2002quasiperiodic}
J.~Fejoz.
\newblock Quasiperiodic motions in the planar three-body problem.
\newblock {\em Journal of Differential Equations}, 183(2):303--341, 2002.

\bibitem{fejoz2004arnold}
J.~F{\'e}joz.
\newblock D{\'e}monstration du `th{\'e}or{\`e}me d'{A}rnold' sur la
  stabilit{\'e} du syst{\`e}me plan{\'e}taire (d'apr{\`e}s {H}erman).
\newblock {\em Ergodic Theory Dynam. Systems}, 24(5):1521--1582, 2004.

\bibitem{fejoz2016secular}
J.~Fejoz and M.~Guardia.
\newblock Secular instability in the three-body problem.
\newblock {\em Archive for rational mechanics and analysis}, 221(1):335--362,
  2016.

\bibitem{fejoz2016kirkwood}
J.~F\'{e}joz, M.~Gu\`ardia, V.~Kaloshin, and P.~Rold\'{a}n.
\newblock Kirkwood gaps and diffusion along mean motion resonances in the
  restricted planar three-body problem.
\newblock {\em J. Eur. Math. Soc. (JEMS)}, 18(10):2315--2403, 2016.

\bibitem{Fejoz:2013:200}
J.~F{\'e}joz and S.~Serfaty.
\newblock {\em Deux cents ans apr{\`e}s Lagrange}.
\newblock Journ{\'e}e Annuelle, Paris, le 28 juin 2013. Soci{\'e}t{\'e}
  Math{\'e}matique de France, 2013.

\bibitem{fenichel1971persistence}
N.~Fenichel.
\newblock Persistence and smoothness of invariant manifolds for flows.
\newblock {\em Indiana University Mathematics Journal}, 21(3):193--226, 1971.

\bibitem{fenichel1974asymptotic}
N.~Fenichel.
\newblock Asymptotic stability with rate conditions.
\newblock {\em Indiana University Mathematics Journal}, 23(12):1109--1137,
  1974.

\bibitem{fenichel1977asymptotic}
N.~Fenichel.
\newblock Asymptotic stability with rate conditions, ii.
\newblock {\em Indiana University Mathematics Journal}, 26(1):81--93, 1977.

\bibitem{Gelfreich:2008}
V.~Gelfreich and D.~Turaev.
\newblock Unbounded energy growth in {H}amiltonian systems with a slowly
  varying parameter.
\newblock {\em Comm. Math. Phys.}, 283(3):769--794, 2008.

\bibitem{GelfreichTuraev2017}
V.~Gelfreich and D.~Turaev.
\newblock Arnold diffusion in a priori chaotic symplectic maps.
\newblock {\em Comm. Math. Phys.}, 353(2):507--547, 2017.

\bibitem{gidea2006topological}
M.~Gidea and R.~de~la Llave.
\newblock Topological methods in the instability problem of hamiltonian
  systems.
\newblock {\em Discrete \& Continuous Dynamical Systems}, 14(2):295, 2006.

\bibitem{MR4033892}
M.~Gidea, R.~de~la Llave, and T.~M-Seara.
\newblock A general mechanism of diffusion in {H}amiltonian systems:
  qualitative results.
\newblock {\em Comm. Pure Appl. Math.}, 73(1):150--209, 2020.

\bibitem{Gomes:2004}
R.~Gomes, A.~Morbidelli, and H.~F. Levison.
\newblock Planetary migration in a planetesimal disk: why did neptune stop at
  30 au?
\newblock {\em Icarus}, 170:492--507, 2004.

\bibitem{GuardiaPS23}
M.~Guardia, J.~Paradela, and T.~Seara.
\newblock A degenerate {A}rnold diffusion mechanism in the {R}estricted 3
  {B}ody {P}roblem.
\newblock Preprint, available at \url{https://arxiv.org/abs/2302.06973}, 2023.

\bibitem{Guzzo:2011}
M.~Guzzo, E.~Lega, and C.~Froeschl\'{e}.
\newblock First numerical investigation of a conjecture by {N}. {N}.
  {N}ekhoroshev about stability in quasi-integrable systems.
\newblock {\em Chaos}, 21(3):033101, 12, 2011.

\bibitem{Harrington:1968}
R.~S. Harrington.
\newblock Dynamical evolution of triple stars.
\newblock {\em Astronom. J.}, pages 190--194, 1968.

\bibitem{herman1998icm}
M.~Herman.
\newblock Some open problems in dynamical systems.
\newblock In {\em Proceedings of the {I}nternational {C}ongress of
  {M}athematicians ({B}erlin, 1998)}, volume Extra Vol. II, pages 797--808
  (electronic), 1998.

\bibitem{hirsch1970invariant}
M.~W. Hirsch, C.~C. Pugh, and M.~Shub.
\newblock Invariant manifolds.
\newblock {\em Bulletin of the American Mathematical Society}, 76(5), 1970.

\bibitem{Kaloshin:2020}
V.~Kaloshin and K.~Zhang.
\newblock {\em Arnold diffusion for smooth systems of two and a half degrees of
  freedom}, volume 208 of {\em Annals of Mathematics Studies}.
\newblock Princeton University Press, Princeton, NJ, 2020.

\bibitem{Lagrange:1853}
J.~Lagrange and J.~Bertrand.
\newblock {\em M{\'e}canique analytique}.
\newblock Mallet-Bachelier, 1853.

\bibitem{Lagrange:1808}
J.-L. Lagrange.
\newblock M{\'e}moire sur la th{\'e}orie des variations des {\'e}l{\'e}ments de
  plan{\`e}tes et en particulier des variations des grands axes de leurs
  orbites.
\newblock {\em {\OE}uvres, t. VI}, pages 713--768, 1808.

\bibitem{Laplace:1776}
P.-S. Laplace.
\newblock Sur le principe de la gravitation universelle et sur les
  in\'egalit\'es s\'eculaires des plan\`etes qui en d\'ependent. {M}\'emoire de
  l'{A}cad\'emie des sciences de {P}aris, {S}avants \'etrangers, ann\'ee 1773.
\newblock {\em {\OE}uvres, t. VIII}, page 201, 1776.

\bibitem{Laskar:2010:Poincare}
J.~Laskar.
\newblock Le syst\`eme solaire est-il stable?
\newblock {\em S\'eminaire Poincar\'e}, pages 221--246, 2010.

\bibitem{Mathieu:1875}
E.~Mathieu.
\newblock M\'{e}moire sur les in\'{e}galit\'{e}s s\'{e}culaires des grands axes
  des orbites des plan\`etes.
\newblock {\em J. Reine Angew. Math.}, 80:97--127, 1875.

\bibitem{moeckel2002drift}
R.~Moeckel.
\newblock Generic drift on {C}antor sets of annuli.
\newblock In {\em Celestial mechanics ({E}vanston, {IL}, 1999)}, volume 292 of
  {\em Contemp. Math.}, pages 163--171. Amer. Math. Soc., Providence, RI, 2002.

\bibitem{Nekhoroshev:1977}
N.~N. Nehoro\v{s}ev.
\newblock An exponential estimate of the time of stability of nearly integrable
  {H}amiltonian systems.
\newblock {\em Uspehi Mat. Nauk}, 32(6(198)):5--66, 287, 1977.

\bibitem{Niederman:1996}
L.~Niederman.
\newblock Stability over exponentially long times in the planetary problem.
\newblock {\em Nonlinearity}, 9(6):1703--1751, 1996.

\bibitem{pinzari2009kolmogorov}
G.~Pinzari.
\newblock {\em On the {Kolmogorov} set for many-body problems}.
\newblock PhD thesis, Universit{}\`a degli Studi di Roma Tre, 2009.

\bibitem{Poisson:1809}
S.~Poisson.
\newblock M\'emoire sur les in\'egalit\'es s\'eculaires des moyens mouvements
  des plan\`etes.
\newblock {\em J. \'Ecole Polytechnique}, 8:1--56, 1809.

\bibitem{Treschev04}
D.~Treschev.
\newblock Evolution of slow variables in a priori unstable hamiltonian systems.
\newblock {\em Nonlinearity}, 17(5):1803--1841, 2004.

\bibitem{Treschev:2012}
D.~Treschev.
\newblock Arnold diffusion far from strong resonances in multidimensional {\it
  a priori} unstable {H}amiltonian systems.
\newblock {\em Nonlinearity}, 25(9):2717--2757, 2012.

\bibitem{Xue:2014:4bp}
J.~Xue.
\newblock Arnold diffusion in a restricted planar four-body problem.
\newblock {\em Nonlinearity}, 27(12):2887--2908, 2014.

\end{thebibliography}
}
\Addresses

\end{document}